\documentclass{amsart}
\usepackage{amssymb}
\usepackage{amsbsy}
\usepackage{amscd}
\usepackage{amsmath}
\usepackage{amsthm}
\usepackage{amsxtra}
\usepackage{latexsym}

\usepackage{mathrsfs}
\usepackage{upgreek}
\usepackage{wasysym}
\usepackage[all,cmtip,dvipdfmx]{xy}

\usepackage[dvipdfmx]{xcolor}
\definecolor{rouge}{rgb}{0.85,0.1,.4}
\definecolor{bleu}{rgb}{0.1,0.2,0.9}
\definecolor{violet}{rgb}{0.7,0,0.8}

\usepackage[
colorlinks=true,linkcolor=bleu,urlcolor=violet,citecolor=rouge]{hyperref}

\newcommand{\bra}{{\langle}}
\newcommand{\ket}{{\rangle}}

\newcommand{\Lam}{\Lambda}

\newcommand{\cprime}{$'$}
\newcommand{\on}{\operatorname}
\newcommand{\+}{\mathop{\oplus}}
\renewcommand{\*}{{\otimes}}
\newcommand{\mc}{\mathcal}
\newcommand{\mf}{\mathfrak}

\newcommand{\fing}{\mf{g}}

\newcommand{\affg}{\widehat{\mf{g}}}

\newcommand{\isomap}{{\;\stackrel{_\sim}{\to}\;}}
\newcommand{\Z}{\mathbb{Z}}
\newcommand{\C}{\mathbb{C}}
\newcommand{\N}{\mathbb{N}}

\newcommand{\W}{\mathcal{W}}  
\newcommand{\ra}{\rightarrow}
\newcommand{\lam}{\lambda}

\newcommand{\bs}{\boldsymbol}

\input{Dynkin.sty}

\def\g{\mathfrak{g}}
\def\l{\mathfrak{l}}
\def\z{\mathfrak{z}}
\def\h{\mathfrak{h}}
\def\n{\mathfrak{n}}
\def\reg{{\rm reg}}
\def\J{{J}_{G}}
\def\O{\mathbb{O}}
\def\SS{\mathbb{S}}
\def\Slo{\mathscr{S}}

\def\IS{{I}}

\def\JJ{\mathscr{J}}
\def\P{\mathscr{P}}
\def\sl{\mathfrak{sl}}
\def\so{\mathfrak{so}}

\def\eps{\varepsilon}

\def\leq{\leqslant}
\def\geq{\geqslant}

\DeclareMathOperator{\End}{End}
\DeclareMathOperator{\Spec}{Spec}
\DeclareMathOperator{\gr}{gr}

\DeclareMathOperator{\ad}{ad}
\DeclareMathOperator{\Hom}{Hom}

\theoremstyle{theorem}
\newtheorem{Th}{Theorem}[section]

\newtheorem{Pro}[Th]{Proposition}

\newtheorem{Lem}[Th]{Lemma}
\newtheorem{lemma}[Th]{Lemma}
\newtheorem{Co}[Th]{Corollary}

\newtheorem{Conj}{Conjecture}
\theoremstyle{remark}

\newtheorem{Rem}[Th]{Remark}
\newtheorem{Claim}{Claim}[section]

\newtheorem{ques}{Question}

\title{Sheets and associated varieties of affine vertex algebras}
\subjclass[2010]{17B67, 17B69, 81R10}
\keywords{sheet, nilpotent orbit, associated variety, affine Kac-Moody algebra,
 affine vertex algebra, affine $W$-algebra}
\author{Tomoyuki Arakawa}
\address{Research Institute for Mathematical Sciences, Kyoto University,
 Kyoto 606-8502 JAPAN}
\email{arakawa@kurims.kyoto-u.ac.jp}
\author{Anne Moreau}
\address{Laboratoire Painlev\'{e}, CNRS U.M.R. 8524, 
   59655 Villeneuve d'Ascq Cedex, France}
\email{anne.moreau@math.univ-lille1.fr}

\begin{document}

 \begin{abstract}
  We show that sheet closures
appear as associated varieties of  affine vertex algebras.
 Further, we give new examples of non-admissible affine vertex algebras
 whose associated variety is  contained in the nilpotent cone. 
 We also prove some conjectures from our previous paper and give new examples of lisse affine $W$-algebras.
 \end{abstract}

\maketitle

\section{Introduction}
It is known \cite{Li05} that  every vertex algebra $V$ is canonically filtered 
and therefore it can be considered as a quantization of  its associated graded Poisson vertex algebra
$\gr V$.
The generating  subring  $R_V$ of $\gr V$ is called the {\em Zhu $C_2$-algebra} 
of $V$ \cite{Zhu96}
and has the structure of a Poisson algebra.
Its spectrum
\begin{align*}
\tilde{X}_V=\Spec R_V
\end{align*}
is called 
the  {\em associated scheme} 
of $V$ and
the corresponding reduced scheme
$X_V=\on{Specm}R_V$ is called the {\em associated variety}
of $V$ (\cite{Ara12,Apisa}).
Since it is Poisson, 
the coordinate ring of its arc space $J_{\infty}\tilde{X}_V$ has
a natural structure of a Poisson vertex algebra (\cite{Ara12}),
and there is a natural surjective homomorphism
$\C[J_{\infty}\tilde{X}_V]\ra \gr V$, which is in many cases an isomorphism.
We have \cite{Ara12} that $\dim \Spec (\gr V)=0$ if and only if
 $\dim X_V=0$, and in this case
$V$ is called {\em lisse} or {\em $C_2$-cofinite}. 


Recently,
associated varieties  of vertex algebras  have  caught attention of physicists
since 
it turned out that 
the associated variety of a  vertex algebra coming \cite{BeeLemLie15} from a {\em four dimensional} $N=2$ superconformal field theory
should coincide with the {\em Higgs branch} of the corresponding  four dimensional theory (\cite{R}).

In the case that $V$ is the simple affine vertex algebra
$V_k(\fing)$ associated with a finite-dimensional simple Lie algebra $\fing$ at level $k\in\C$,
$X_V$ is a Poisson subscheme of $\fing^*$
which is $G$-invariant and conic, where $G$ is the adjoint group of $\g$.
Note that on the contrary to 
the associated variety of a primitive ideal of $U(\fing)$, 
the variety $X_{V_k(\fing)}$ is not necessarily
contained in the nilpotent cone $\mc{N}$ of $\fing$.
In fact, $X_{V_k(\fing)}=\fing^*$ for a generic $k$.
On the other hand, $X_{V_k(\fing)}=\{0\}$
 if and only if $V_k(\fing)$ is integrable, that is,
 $k$ is a non-negative integer.
Except for a few cases, 
 the description of $X_V$ is fairly open
 even for $V=V_k(\fing)$,
 despite of
 its connection with 
 four dimensional superconformal field theories.

 In \cite{Ara09b}, the first named author showed that $X_{V_k(\fing)}$
  is  the closure of some nilpotent orbit
 of $\fing^*$ in the case that $V_k(\fing)$ is {\em admissible} \cite{KacWak89}.

 In the previous article \cite{AM15}, we showed that
 $X_{V_k(\fing)}$  is  the minimal nilpotent orbit closure
 in the case that $\fing$ belongs to the Deligne exceptional series 
 \cite{De96} and $k=-h^{\vee}/6-1$,
 where $h^{\vee}$ is the dual Coxeter number of $\fing$.
 Note that the level  $k=-h^{\vee}/6-1$  is not admissible  for the types
 $D_4$, $E_6$, $E_7$, $E_8$.

 In all the above cases, $X_{V_k(\fing)}$
 is a closure of a nilpotent orbit $\mathbb{O} \subset \mc{N}$,   
 or $X_{V_k(\fing)}= \fing^*$. 
Therefore it is natural to ask the following. 
 \begin{ques}\label{ques2}
  Are there cases
  when 
  $
X_{V_k(\fing)} \not\subset \mc{N}$
  and $X_{V_k(\fing)}$ is a proper subvariety of $ \fing^*$? 
  For example, are there cases
  when $X_{V_k(\fing)}$ is  the closure of a non-nilpotent {\em 
  Jordan class} (cf.~\S \ref{sec:sheet})? 
\end{ques}

Identify $\g$ with $\g^*$ through a non-degenerate bilinear form of $\g$. 

Given $m \in \N$, let $\g^{(m)}$ be the set of elements $x\in\g$ such that $\dim \g^x=m$, 
with $\g^x$ the centralizer of $x$ in $\g$. A subset $\SS \subset \g$ is called a {\em sheet} 
of $\g$ if it is an irreducible component of one of the locally closed sets $\g^{(m)}$. 
It is $G$-invariant and conic and  by \cite{ImHof}, and it is smooth if $\g$ is classical. 
The sheet closures are the closures of certain Jordan classes 
and they are parameterized by the  
$G$-conjugacy classes of pairs $(\l,\O_\l)$ where $\l$ is a Levi subalgebra 
of $\g$ and $\O_\l$ is a {\em rigid} nilpotent orbit of $\l$, 
i.e., which cannot be properly induced in the sense of 
Lusztig-Spaltenstein \cite{BK,Borh} (see also \cite[\S39]{TY}). 
The pair $(\l,\O_\l)$ is called the {\em datum} of the corresponding sheet. 
When $\O_\l$ is zero, the sheet is called {\em Dixmier}, 
meaning that it contains a semisimple element \cite{Dix75,Dix76}. 
We will denote by $\SS_\l$ the sheet with datum $(\l,\{0\})$. 
We refer to \S\ref{sec:sheet} for more details about this topic.

It is known that sheets appear in the representation theory of 
finite-dimensional Lie algebras,
see,~e.g.,~\cite{BorhBry82,BorhBry85,BorhBry89},  and more recently of
finite $W$-algebras, \cite{PT14,Pre14}.

Since the sheet closures 
are $G$-invariant conic algebraic
varieties which are not necessarily contained in $\mc{N}$, 
one may expect that
there are simple affine vertex algebras
whose
associated variety
is the closure of some sheet.
This  is indeed the case.
 \begin{Th}\label{Th:main1}
 \begin{enumerate}
  \item For $n\geq 4$,
\begin{align*}
 \tilde{X}_{V_{-1}(\mf{sl}_n)}\cong \overline{\SS_{\l_1}}
\end{align*}
	as schemes,
	where $\l_1$ is the 
			standard Levi subalgebra of $\mf{sl}_n$ generated by 
all simple roots except $\alpha_{1}$.
	Moreover
	$V_{-1}(\mf{sl}_n)$ is a quantization of 
		the infinite jet scheme 
	$J_{\infty}\overline{\SS_{\l_1}}$ of $\overline{\SS_{\l_1}}$ 
	in the sense that 
	\begin{align*}
	SS(V_{-1}(\mf{sl}_n)) \cong J_{\infty}\overline{\SS_{\l_1}} 
		\end{align*}
		as topological spaces, that is, $SS(V_{-1}(\mf{sl}_n))_{\on{red}}\cong (J_{\infty}\overline{\SS_{\l_1}})_{\on{red}}.$ 
	  \item For $m\geq 2$,
\begin{align*}
\tilde{ X}_{V_{-m}(\mf{sl}_{2m})}\cong \overline{\SS_{\l_0}}
\end{align*}
	as schemes,
		where $\l_0$ is the
		standard Levi subalgebra of $\mf{sl}_{2m}$ generated by 
all simple roots except $\alpha_{m}$.
 \end{enumerate}
 \end{Th}
  
 Further,
 we show that
 the Zhu algebras 
 of  the above vertex algebras are
 naturally embedded  into the algebra
 of global differential operators on $Y=G/[P,P]$,
 where $P$ is the
connected parabolic subgroup of $G$ corresponding 
to a parabolic subalgebra $\mf{p}$ having the above Levi 
subalgebra
as Levi factor, see Theorem \ref{Th:Zhu-1} and 
Theorem \ref{Th:Zhu-2}.

 The vertex algebra $V_{-1}(\mf{sl}_n)$ has appeared in
 the work of Adamovi{\'c} and Per{\v{s}}e \cite{AdaPer14},
 where they studied the fusion rules and the complete reducibility of
 $V_{-1}(\mf{sl}_n)$-modules. It would be very interesting to know
 whether the vertex algebra $V_{-m}(\mf{sl}_{2m})$ has similar properties.

%
%

    \begin{Th} \label{Th:reducible}
     Let $r$ be an odd integer, 
     and let
    $\l_r$ be the Levi subalgebras of $\fing=\mf{so}_{2r}$
generated by the simple roots $\alpha_1,\ldots,\alpha_{r-2},\alpha_{r-1}$. 
Denote by $\SS_{\l_r}$ the corresponding Dixmier sheet. 
Then  
 \begin{align*}
X_{V_{2-r}(\mf{so}_{2r})}=\overline{\SS_{\l_{r}}}.
   \end{align*}
  \end{Th}

  The vertex algebra
  $V_{2-r}(\mf{so}_{2r})$
  has been studied by 
  Per{\v{s}}e \cite{Per13}
  for all $r$,
  and by
   Adamovi{\'c}
   and Per{\v{s}}e
   \cite{{AdaPer14}}
   for odd $r$.
   The proof of Theorem \ref{Th:reducible}
   uses the fact proved in \cite{{AdaPer14}} that,
   for odd $r$, 
   $V_{2-r}(\mf{so}_{2r})$
   has infinitely many simple objects in the category $\mc{O}$.
   
     Remarkably, it turned out that    the structure of the vertex algebra
  $V_{2-r}(\mf{so}_{2r})$
  substantially differs depending on the parity of $r$.
  \begin{Th}\label{Th:Dr-for-even-r}
 Let $r$ be an even integer such that $r\geq 6$.
 Then 
\begin{align*}
\overline{\mathbb{O}_{min}}\subsetneqq X_{V_{2-r}(\mf{so}_{2r})}
 \subset\overline{ \mathbb{O}_{(2^{r-2},1^4)}},
\end{align*} 
where $\O_{\min}$ is the minimal nilpotent orbit of $\mf{so}_{2r}$ 
and $\mathbb{O}_{(2^{r-2},1^4)}$ is the nilpotent orbit of $\mf{so}_{2r}$ 
associated with the partition $(2^{r-2},1^4)$ of $2r$. 

In particular, $X_{V_{2-r}(\mf{so}_{2r})}$
 is contained in $ \mc{N}$,
 and hence, there are only finitely many simple $V_{2-r}(\mf{so}_{2r})$-modules
 in the category $\mc{O}$.
\end{Th}
The above theorem gives new examples
of non-admissible affine vertex algebras
whose associated varieties are contained in the nilpotent cone (cf.\
\cite{AM15}).
 In fact we conjecture\footnote{This conjecture is now confirmed in
 our paper \cite{AM17}.} 
 that
 $ X_{V_{2-r}(\mf{so}_{2r})}
 =\overline{ \mathbb{O}_{(2^{r-2},1^4)}}$.
This conjecture is confirmed for
$r=6$, see  Theorem \ref{Th:r-6}.
Notice that for $r=4$, 
$X_{V_{-2}(\mf{so}_8)}
    =\overline{\mathbb{O}_{min}}=\overline{\mathbb{O}_{(2^2,1^4)}}$
    by \cite{AM15}. So the conjecture  also holds for $r=4$.
    
Our proof of the above stated results is
based on the analysis of singular
vectors
of degree $2$ \cite{AM15} and the theory of $W$-algebras
\cite{KacRoaWak03,KacWak04,Ara05,Ginz,Ara08-a,Ara09b}.
 This method  works for some other types as well,  in particular in types $B$ and $C$,
which will be studied in our subsequent work. 

\smallskip

We conjecture that
$X_{V_k(\fing)}$  is  always equidimensional,
and that 
$X_{V_k(\fing)}$  is  
irreducible provided that
$X_{V_k(\fing)}\subset \mc{N}$, see Conjecture \ref{Conj:equidim}.

\smallskip

By \cite[Theorem 4.23]{Ara09b} (cf.~Theorem \ref{Th:W-algebra-variety}) 
we know that the simple $W$-algebra
$\W_k(\fing,f)$
 associated with $(\g,f)$ 
at level $k$ (\cite{KacRoaWak03})
 is lisse if $X_{V_k(\g)}=\overline{G.f}$. 
For a minimal nilpotent element $f \in \O_{min}$, we prove that the converse is also
true provided that $k\not\in \Z_{\geq 0}$. 
\begin{Th} \label{Th:lisse-rig}  
Suppose that
$k\not\in \Z_{\geq 0}$, and let  $f \in \O_{min}$. 
Then
the minimal $W$-algebra ${\W}_k(\g,f)$ is lisse if and only if $
X_{V_k(\g)}=\overline{\O_{min}}$. 
\end{Th}

We also show the following results, which were conjectured in \cite{AM15}. 

 \begin{Th}\label{Th:G2}
  Let $\fing$ of type $G_2$.
  Then $\W_k(\fing,f_{\theta})$ is lisse if and only if
$k$ is admissible with denominator $3$,
or an integer equal to or greater than $-1$.
 \end{Th}

Thus, we obtain a new family of lisse minimal $W$-algebras
   $\W_k(G_2,f_{\theta})$, for $k=-1,0,1,2,3\dots$.
    \begin{Th}  \label{Th:classification}
   Suppose that $\fing$ is not of type $A$.
   Then
Conjecture 2 of \cite{AM15} holds, 
   that is,
   \begin{align*}
X_{V_k(\fing)}
    =\overline{\mathbb{O}_{min}}
   \end{align*}
   if and only if the one of the following conditions holds:
      \begin{enumerate}
    \item $\fing$ is of type $C_r$ $(r\geq 2)$ or 
	  $F_4$, and $k$ is admissible with denominator $2$, 
    \item $\fing$ is of type $G_2$, and $k$ is admissible with denominator $3$,
	  or $k=-1$, 
    \item $\fing$ is of type $D_4$, $E_6$, $E_7$ or $E_8$ and
	  $k$ is an integer such that
	  \begin{align*}
	   -\frac{h^{\vee}}{6}-1\leq k\leq -1, 
	  \end{align*}
	      \item $\fing$ is of type $D_r$ with $r\geq 5$, and $k=-2,-1$.
   \end{enumerate}
  \end{Th}


The rest of the paper is organized as follows.
In \S \ref{sec:sheet} we recollect some results concerning sheets
that will be needed later
and we give a description of Dixmier sheets of rank one.
In \S \ref{section:Ginzburg}, we use Slodowy slices and Ginzburg's results on 
finite $W$-algebras to state some useful lemmas.
In \S \ref{sec:associated-variety} we state 
results and conjectures on the associated
variety of a vertex algebra.
In \S \ref{sec:Zhu} we recall some fundamental results on Zhu's algebras of vertex
algebras.
In \S \ref{sec:affine-W-algebras} we recall and state some fundamental results on $W$-algebras.
In \S \ref{sec:A-theta1} we study level $-1$ affine vertex algebras of type $A_{n-1}$, $n\geq 4$,
and prove Theorem \ref{Th:main1} (1).
In \S \ref{sec:A-theta}  we study level $-m$ affine vertex algebras of type $A_{2m-1}$, $m\geq 2$,
and prove Theorem \ref{Th:main1} (2). 
In \S \ref{sec:BD} we study level $2-r$ affine vertex algebras of type $D_{2r}$, $r\geq 5$,
and prove Theorem \ref{Th:reducible} and Theorem \ref{Th:Dr-for-even-r}.
In \S \ref{sec:others}, we prove Theorem \ref{Th:lisse-rig}, 
Theorem \ref{Th:G2} and Theorem \ref{Th:classification} and some other results related to 
our previous work \cite{AM15}. 
In particular we obtain a new family of lisse minimal $W$-algebras.

\subsection*{Notations}
As a rule, for $U$ a $\g$-submodule of $S(\g)$, 
we shall denote by $\IS_U$ the ideal of $S(\g)$ 
generated by $U$, and for $I$ an ideal in $S(\g)\cong \C[\g^*]$, 
we shall denote by $V(I)$ the zero locus of $I$ in $\g^*$.  

{Let $$\affg=\fing[t,t^{-1}]\+ \C K \+ \C D$$be the affine Kac-Moody Lie algebra
associated with $\fing$ {and the inner product }
$(~|~)=1/2h^{\vee}\times $ Killing form (see \S \ref{sec:associated-variety}). }
For $\lambda \in \h^*$ (resp.~$\hat{\h}^*$), 
$L_\g(\lambda)$ (resp.~$L(\lambda)$) 
denotes the irreducible 
highest weight representation of $\g$ (resp.~$\hat{\g}$)  
with highest weight $\lambda$ where $\h$ is a Cartan subalgebra of $\g$ 
and $\hat{\h}=\h \oplus \C K\+ \C D$.


\subsection*{Acknowledgments} 
We are indebted to the CIRM  Luminy for its hospitality during our stay 
as ^^ ^^ Research in pairs'' in October, 2015.
Some part of the work was done while we were at the conference
``Representation Theory XIV'',  Dubrovnik,
June 2015. We are grateful to the organizers of the conference.

Results in this paper were presented by 
the first named author
in part in ``Vertex operator algebra and related topics'',
Chendgu,
September, 2015.
He thanks
the organizers of this  conferences.
He is also grateful  to
Universidade de  S\~{a}o Paulo for its hospitality during his stay
in November and December, 2015.
His research is supported 
 by JSPS KAKENHI Grant Number 17H01086,
 17K18724.

The second named author would like to thank Micha\"{e}l Bulois 
for helpful comments about sheets. 
Her research is supported by the ANR Project GeoLie Grant number ANR-15-CE40-0012. 

We are very thankful to David Vogan for indicating us an error 
in \cite[Lemma 7.3.2(ii)]{CMa} that induced a mistake in the preliminary version of this paper 
(see Section \ref{sec:BD} for more details). 

\section{Jordan classes and sheets}  \label{sec:sheet}
Most of results presented in this section come from \cite{BK, Borh} or \cite{Kat}. 
Our main reference for basics about Jordan classes and sheets is \cite[\S39]{TY}.

Let $\fing$ be a simple Lie algebra over $\C$ and 
$(~|~)=1/2h^{\vee}\times $ Killing form,  as in the introduction.
We often identify $\fing$ with $\fing^*$ via $(~|~)$. 

For $\mf{a}$ a subalgebra of $\g$, denote by $\z(\mf{a})$ its center. 
For $Y$ a subset of $\g$, denote by $Y^{\rm reg}$ the set 
of $y \in Y$ for which $\g^y$ has the minimal dimension 
with $\g^y$ the centralizer of $y$ in $\g$. 
In particular, if $\l$ is a Levi subalgebra of $\g$, then 
$$\z(\l)^{\reg}:=\{y \in \g \; |\; \z(\g^{y}) = \z(\l) \},$$
and $\z(\l)^{\reg}$ is a dense open subset of $\z(\l)$. 
For $x\in\fing$, denote by $x_s$ and $x_n$ the semisimple 
and the nilpotent components of $x$ respectively. 

The {\em Jordan class} of $x$ is
$$\J(x):=G.(\z(\g^{x_s})^{\reg} + x_n).$$
It is a $G$-invariant, irreducible, and locally closed subset of $\g$. 
To a Jordan class $J$, we associate its {\em datum}  
which is the pair $(\l,\O_\l)$ defined as follows. 
Pick $x \in J$. Then $\l$ is the Levi subalgebra $\g^{x_s}$ 
and $\O_\l$ is the nilpotent orbit in $\l$ of $x_n$. 
The pair $(\l,\O_\l)$ does not depend on $x\in J$ up to 
$G$-conjugacy, and there is a one-to-one correspondence 
between the set of pairs $(\l,\O_\l)$ as above, up to $G$-conjugacy, 
 and the  set of Jordan classes. 

A {\em sheet} is an irreducible component of the subsets 
$$\g^{(m)} = \{x\in \g \;| \; \dim \g^{x} = m\},\qquad m \in \N.$$ 
It is a finite disjoint union of Jordan classes. 
So a sheet $\SS$ contains a unique dense open Jordan class $J$ 
and we can define the {\em datum} of $\SS$ as the datum $(\l,\O_\l)$ 
of the Jordan class $J$. We have 
$$\overline{\SS} = \overline{J} \quad \text{ and }\quad \SS= (\overline{J})^{\reg}.$$
A sheet is called {\em Dixmier} if it contains a semisimple element 
of $\g$. 
A sheet $\SS$ with datum $(\l,\O_\l)$ is Dixmier if and only if $\O_\l=\{0\}$. 
We shall simply denote by $\SS_\l$ the Dixmier sheet with datum $(\l,\{0\})$. 

A nilpotent orbit is called {\em rigid} if  
it cannot be properly induced in the sense of 
Lusztig-Spaltenstein. 
A closure of a Jordan class with datum $(\l,\O_\l)$ is a sheet closure if and only if $\O_\l$ 
is rigid in $\l$. So we get a one-to-one correspondence 
between the set of pairs $(\l,\O_\l)$, up to $G$-conjugacy, 
with $\l$ a Levi subalgebra 
of $\g$ and $\O_\l$ a rigid nilpotent orbit of $\l$, and the  
set of sheets. 

Each sheet contains a unique nilpotent orbit. 
Namely, if $\SS$ is a sheet with datum $(\l,\O_\l)$ 
then the induced nilpotent orbit ${\rm Ind}_{\l}^\g(\O_\l)$ of $\g$ 
from $\O_\l$ in $\l$ is the unique nilpotent orbit 
contained in $\SS$. 
Note that a nilpotent orbit $\O$ is itself a sheet if and only if $\O$ is rigid. 
For instance, outside the type $A$, the minimal nilpotent orbit 
$\O_{min}$ is always a sheet since it is rigid. 

The {\em rank} of a sheet $\SS$ with datum $(\l,\O_\l)$ is by definition 
$${\rm rank}(\SS) := \dim \SS -\dim {\rm Ind}_{\l}^\g(\O_\l) = \dim \z(\l).$$
If $\SS=\SS_\l$ is Dixmier, then $\O_\l=0$ and we have 
\begin{align*}
\overline{\SS} = G.[\mf{p},\mf{p}]^\perp= G.(\z(\l)+\mf{p}_u) \quad \text{ and } 
\quad 
\SS = (G.[\mf{p},\mf{p}]^\perp)^{\reg}. 
\end{align*}
where $\mf{p}=\l\oplus\mf{p}_u$ is a parabolic subalgebra of $\g$ with Levi factor $\l$ and 
nilradical $\mf{p}_u$ (cf.~\cite[Proposition 39.2.4]{TY}). 

Let $\h$ be a Cartan subalgebra of $\fing$. 
For $\SS$ a sheet with datum $(\l,\O_\l)$, one can assume without loss  
of generality that 
$\h$ is a Cartan subalgebra of $\l$. In particular, $\z(\l) \subset \h$. 
 
\begin{Lem}\label{Lem:sheet}
   \begin{enumerate}
\item Let $\mathbb{S}_{\mf{l}}$ be  a
	 Dixmier sheet of rank one, that is, 
$\z(\l)= \C \lambda$ for some $\lam\in \h \setminus\{0\}$. 
Then 
\begin{align*}
&\overline{\SS_\l} = \overline{G.\C^*\lambda} = G.(\C \lambda+\mf{p}_u) 
=  G.\C^*\lambda \cup \overline{{\rm Ind}_\l^\g(0)},  
\end{align*}
and 
\begin{align*}
\SS_\l = G.\C^*\lambda \cup {\rm Ind}_\l^\g(0). &
\end{align*}

    \item 
	 Let  $\mathbb{S}_{\mf{l}_1},\dots, \SS_{\l_n}$ be  
	 Dixmier sheets of rank one, that is, 
	  $\mf{z}(\mf{l}_i)=\C \lam_i$ for some $\lam_i\in \h \setminus\{0\}$,
	  such that
$\dim \on{Ind}^\g_{\l_i} (0)=\dim \on{Ind}^\g_{\l_j}(0)$ for all $i,j$.
Let $X$ be a $G$-invariant, conic, Zariski closed subset
  of $\fing^*$
  such that
  \begin{align*}
   X\cap \mf{h}=\bigcup_{i=1}^n\C \lam_i,\quad X\cap
   \mc{N}\subset \bigcup_{i=1}^n\overline{\on{Ind}_{\mf{l}_i}^\fing(0)}.
  \end{align*}
  Then $X=\bigcup_{i=1}^n\overline{\mathbb{S}_{\mf{l}_i}}$.
   \end{enumerate}
  \end{Lem}
 
Part (1) of the lemma is probably well-known. 
We give a proof for the convenience of the reader.

    \begin{proof}
(1) The equalities $\overline{\SS_\l} = \overline{G.\C^*\lambda} = G.(\C \lambda+\mf{p}_u)$ 
are clear by \cite[Corollaries~39.1.7 and~39.2.4]{TY}. 
Let us prove that $\overline{\SS_\l} = G.\C^*\lambda \cup \overline{{\rm Ind}_\l^\g(0)}$. 
The inclusion $G.\C^*\lambda \cup \overline{{\rm Ind}_\l^\g(0)} \subset 
\overline{\SS_\l}$ 
is known \cite[Proposition 39.3.5]{TY} (and its proof). So it suffices to prove that 
$\C \lambda+\mf{p}_u \subset G.\C^*\lambda \cup \overline{{\rm Ind}_\l^\g(0)}$ 
since $\overline{\SS}$ is $G$-invariant. 

Let $x= c \lam + y\in \C \lambda+\mf{p}_u $ with $c \in \C$ and $y \in \mf{p}_u$. 
Assume that $c \in \C^*$. Then $x_s$ and $c \lam$ are $G$-conjugate. 
Since $x \in \overline{\SS_\l}$, $\dim \g^{x} \geq \dim \g^\lam$. 
But $\dim \g^{x} \geq \dim \g^\lam$ if and only if $x_n=0$ since 
$\g^{x}= (\g^{x_s})^{x_n}$. Hence $x$ is $G$-conjugate to $c \lam$,  
and so $x \in G.\C^*\lam$. 
If $c=0$, then $x \in \mf{p}_u$ and so $x$ is nilpotent. 
But $(\C \lambda+\mf{p}_u )\cap \mc{N} \subset \overline{\SS_\l} \cap \mc{N}$ 
and $\overline{\SS_\l} \cap \mc{N}=
\overline{{\rm Ind}_\l^\g(0)}$ by \cite[\S\S39.2.6 and 39.3.3]{TY}, 
whence the statement. 

It remains to prove that $\SS_\l = G.\C^*\lambda \cup {\rm Ind}_\l^\g(0)$. 
We have $\SS_\l=(G.(\C \lambda+\mf{p}_u))^\reg$, 
and the inclusion $G.\C^*\lambda \cup {\rm Ind}_\l^\g(0) \subset \SS_\l$ 
is clear. So it suffices to prove that 
$(\C \lambda+\mf{p}_u)^\reg \subset G.\C^*\lambda \cup {\rm Ind}_\l^\g(0)$ 
since $\overline{\SS_\l}$ and $\SS_\l$ are $G$-invariant. 
The above argument shows that for $x \in (\C \lambda+\mf{p}_u)^\reg 
\setminus \mf{p}_u$, $x \in G.\C^* \lam$. And if $x \in (\mf{p}_u)^\reg$, 
then $x \in 
(\mf{p}_u)^\reg \cap \mc{N} \subset \SS_\l \cap \mc{N}= {\rm Ind}_\l^\g(0)$, 
whence the statement. 

(2) The inclusion
    $\bigcup_i \overline{\SS_{\l_i}} \subset X$ is clear.
Conversely, let $x\in X$. 
    If $x$ is nilpotent, then 
    $x \in
 \bigcup_i \overline{\SS_{\l_i}}$ by the assumption. 
Assume that $x$ is not nilpotent, that is $x_s \not=0$. 
Since $X$ is $G$-stable, we can assume 
that $x_s \in \h$.  If $x_n=0$ then 
$x_s \in X \cap \h \subset \bigcup_{i=1}^n \SS_{\l_i}$ by hypothesis. 

Assume that $x_n\not=0$ and 
let $(e,h,f)$ be an $\mf{sl}_2$-triple of $\g$ with $e=x_n$. 
We can assume that $h \in \h$ so that $[h,x_s]=0$. 
Let $\gamma : \C^* \to G$ be the one-parameter subgroup generated by $h$. 
Since $X$ is $G$-invariant, for any $t\in \C^*$, the element 
$$\gamma(t).x= x_s + t^{2} x_n$$
belongs to $X$. Since $X$ is closed, 
we deduce that $x_s \in X$. 
     So, by the assumption, $x_s=  c \lam_i$ for some $i$ and
     $c\in \C^*$.  
Therefore, because $X$ is a  cone, we can assume 
     that $x_s =\lam_i$.
     Thus 
     $\l_i=\fing^{x_s}$ and $x_n\in \l_i$.

For any $t\in \C^*$, 
the element 
$$t^{2} \gamma(t^{-1}).x= t^{2} (\lam_i+ t^{-2} x_n) = t^{2} \lam_i + x_n $$
belongs to $X$. 
This shows that $\C^* \lam_i+ x_n \subset X$. 
Then  
$$G.(\C^* \lam_i+ x_n) =G.(\z(\l_i)^{\rm reg} + x_n) = \J(x) \subset 
X,$$
whence $ \overline{\J(x)} \subset X$ because 
    $X$ is closed.
    
    Let $\O_{\l_i,x_n}$ be  the nilpotent orbit of $x_n$ in $\l_i$.
    One knows that 
     $\on{Ind}^\g_{\l_i} (\mathbb{O}_{\l_i,x_n})\subset \overline{\J(x)}$ \cite{Borh}. 
So  $\on{Ind}^\g_{\l_i} (\mathbb{O}_{\l_i,x_n}) \subset X$,
and the assumption gives that 
    $\on{Ind}^\g_{\l_i}(\mathbb{O}_{\l_i,x_n}) \subset
    \bigcup_j \on{Ind}^\g_{\l_j}(0)$. 
     In particular,
    $$\dim \on{Ind}^\g_{\l_i}(\mathbb{O}_{\l_i,x_n}) = \dim \on{Ind}^\g_{\l_j}(0)$$ 
    for any $j$ (this makes sense by our assumption on the Levi subalgebras 
    $\l_j$), whence $\on{codim}_{\l_i}(\mathbb{O}_{\l_i,x_n})=\on{codim}_{\l_j}(0)
    =\on{codim}_{\l_i}(0)=\dim \l_i$ by the properties of induced nilpotent orbits. 
    So $\mathbb{O}_{\l_i,x_n}=\{0\}$, that is $x_n=0$, 
    and $x=\lam_i \in \overline{\SS_{\l_i}}$. 
    \end{proof}

Let $P$ be the  
connected parabolic subgroup of $G$ with Lie algebra $\mf{p}=\mf{l}\+ \mf{p}_u$. 
The $G$-action on
\begin{align*}
Y:=G/[P,P], 
\end{align*}
 where $[P,P]$ is the commutator-subgroup of $P$, 
induces an algebra homomorphism 
\begin{align*}
 \psi_Y \colon U(\g) \to \mc{D}_Y
\end{align*}
from $U(\g)$ to the algebra $\mc{D}_Y$ of global differential operators on $Y$. 
Let 
$$\JJ_{Y}:=\ker \psi_Y$$ 
be the kernel of this homomorphism. It is a two-sided ideal of $U(\g)$. 
It has been shown by
Borho and Brylinski \cite[Corollary 3.11 and Theorem 4.6]{BorhBry82} 
that
$\sqrt{{\rm gr} \JJ_Y}$ is the defining ideal of the Dixmier sheet
closure determined by $P$, 
that is, $\overline{\SS_{\l}}$. 
Furthermore, 
$$\JJ_Y = \bigcap_{\lam \in \z(\l)^*} {\rm Ann}\,U(\g) \otimes_{U(\mf{p})} \C_\lam.$$
Here, for $\lam \in\mf{p}^*$, $\C_\lam$ stands for the one-dimensional representation of $\mf{p}$ 
corresponding to $\lam$, and we extend a linear form $\lam \in \z(\l)^*$ to $\mf{p}^*$ 
by setting $\lam(x)=0$ for $x \in [\l,\l]\oplus \mf{p}_u$. 
Identifying $\g$ with $\g^*$ through $(~|~)$, $\z(\l)^*$ identifies with $\z(\l)$. 
In particular, if $\z(\l)=\C \lam$ for some nonzero semisimple element $\lam \in\g$, 
we get 
 $$\JJ_Y = \bigcap_{t \in \C} {\rm Ann}\,U(\g) \otimes_{U(\mf{p})} \C_{t \lam}.$$
In fact 
\begin{align}
\JJ_Y = \bigcap_{t \in Z} {\rm Ann}\,U(\g) \otimes_{U(\mf{p})} \C_{t
 \lam}.
 \label{eq:enough-for-zariski-dense}
\end{align}
 for
any Zariski dense subset $Z$ of $\C$ (\cite{BorJan77}). 
 
\smallskip

 In this paper, we shall consider sheets in Lie algebras of 
 classical types $A_{r }$ and $D_{r }$.  
 Let us introduce more specific notations. 
Let $n\in \N^{*}$, and denote by $\P(n)$ the set of partitions of $n$. 
As a rule, we write an element $\bs{\lambda}$ of $\P(n)$ 
as a decreasing sequence $\bs{\lam}=(\lambda_1,\ldots,\lam_s)$ 
omitting the zeroes. 

\smallskip 

\noindent
{\bf Case $\sl_n$.} 
According to \cite[Theorem 5.1.1]{CMa}, nilpotent orbits of $\sl_{n}$ are 
parametrized by $\P(n)$. 
For $\bs{\lam}\in \P(n)$, we denote by 
$\O_{\bs{\lam}}$ the corresponding nilpotent orbit of $\sl_n$. 
In $\sl_n$, all sheets are Dixmier and each nilpotent orbit is contained 
in exactly one sheet. 
The Levi subalgebras of $\sl_n$, 
and so the (Diximer) sheets, are parametrized by compositions of $n$. 
More precisely, if $\bs{\lam}\in \P(n)$, then the (Dixmier) sheet associated 
with $\bs{\lam}$ is the unique sheet containing 
$\O_{^t \! \bs{\lam}}$ where $^t\!\bs{\lam}$ is the dual partition of $\bs{\lam}$. 

\smallskip 

\noindent
{\bf Case $\so_n$.} 
Set 
$$
\P_{1}(n):=\{\bs{\lam} \in \P(n)\; ; \; \text{number of parts of each even 
number is even}\}.
$$ 
According to \cite[Theorem 5.1.2 and Theorem 5.1.4]{CMa}, nilpotent orbits of $\so_{n}$ 
are parametrized by $\P_{1}(n)$, with the exception that each {\em very even} 
partition $\bs{\lam} \in\P_{1}(n)$ (i.e., $\bs{\lam}$ has only even parts) 
corresponds to two nilpotent orbits.
For $\bs{\lam}\in \P_1(n)$, not very even, we denote by 
$\O_{\bs{\lam}}$ the corresponding nilpotent orbit of $\so_n$. 
For very even $\bs{\lambda}\in \P_1(n)$, we denote by $\O_{\bs{\lam}}^{I}$ 
and $\O_{\bs{\lambda}}^{I\!I}$ the two corresponding nilpotent orbits of $\so_n$.  
In fact, their union form a single ${\rm O}_{n}$-orbit. 

Contrary to the $\sl_n$ case, it may happen in the $\so_n$ 
case that a given nilpotent orbit 
belongs to different sheets, and not all sheets are Dixmier. 

\section{Some useful lemmas}\label{section:Ginzburg}
Let $f$ be a nilpotent element of $\fing$ that we embed into an $\sl_2$-triple $(e,h,f)$ 
of $\g$ and let 
$$\Slo_f := \chi + (\g^{f})^*$$ 
be the {\em Slodowy slice associated with $(e,h,f)$} where 
$$\chi := (f | \,\cdot \,) \in \g^*.$$
Denote by $\g(h,i)$ the $i$-eigenspace of ${\rm ad}(h)$ for $i\in\Z$.
Choose a Lagrangian subspace 
$\mathfrak{L} \subset \g(h,1)$ and set 
$$\mf{m} := \mathfrak{L} \oplus \bigoplus_{i \geq 2} \g(h,i), \qquad 
J_\chi :=  \sum_{x \in \mf{m}} \C[\g^*](x-\chi(x)).$$
Let $M$ be the unipotent subgroup of $G$ corresponding to $\mf{m}$.

Let
\begin{align*}
 \mu:\fing^*\ra \mf{m}^*
\end{align*}
be the moment map for the $M$-action, which is just a restriction map.
By \cite{GanGin02}, the adjoint action map gives the isomorphism
\begin{align*}
 M\times \Slo_f \stackrel{\sim\,}{\longrightarrow} \mu^{-1}(\chi),
\end{align*}
and thus,
\begin{align*}
 \Slo_f\cong \mu^{-1}(\chi)/M.
\end{align*}
In particular,
\begin{align*}
 \C[\Slo_f]\cong \C[\mu^{-1}(\chi)]^M=\left(\C[\fing^*]/J_{\chi}\right)^M.
\end{align*}

Let $\overline{\mc{HC}}$ be the
category of finitely generated
$(\C[\fing^*],G)$-modules,
that is,
the category of  finitely generated $\C[\fing^*]$-modules $K$
equipped with the $G$-module structure such that
$g .(f.m)=(g(f)). (g.m)$ for $g\in G$, $f\in \C[\fing^*]$, $m\in K$.

 \begin{Th}[\cite{Ginz}, see also \cite{A2012Dec}]\label{Th:ginzburg}
  \begin{enumerate}
   \item
	The functor
\begin{align*}
H_f\colon\overline{\mc{HC}}\ra \C[\Slo_f]\on{-mod},\quad
K\mapsto (K/{J_{\chi} K})^M, 
\end{align*}is exact.
   \item For any $K \in \overline{\mc{HC}}$, $\on{supp}_{\C[\Slo_f]}H_f(K)=
	 (\on{supp}_{\C[\fing^*]}K)
	 \cap \Slo_f$.
  \end{enumerate}
 \end{Th}
 \begin{lemma}\label{lem:pre-criterion}
  Let $K\in \overline{\mc{HC}}$.
  Then
  $\overline{G.f}\subset \on{supp}_{\C[\fing^*]}K$
  if and only if
    $K\ne {J_{\chi} K}$.
 \end{lemma}
  \begin{proof}
   Since $\Slo_f$ admits a $\C^*$-action contracting to $f$,
 Theorem \ref{Th:ginzburg} (2) implies that 
  $\overline{G.f}\subset \on{supp}_{\C[\fing^*]}K$
  if and only if
    $H_f(K)\ne 0$.
   However
   $H_f(K)\ne 0$ if and only if
   $    K/{J_{\chi} K}\ne 0$ by
      \cite[Proposition 3.3.6]{Ginz}.
  \end{proof}
Let $I$ be an $\ad \fing$-invariant ideal of $\C[\fing^*]$, 
  so that
$\C[\fing^*]/I\in \overline{\mc{HC}}$.
Applying Lemma~\ref{lem:pre-criterion}
to $K=\C[\fing^*]/I$ we obtain the following assertion.
\begin{lemma}\label{lem:criterion}
 Let $I$ be an $\ad \fing$-invariant ideal of $\C[\fing^*]$.
 Then $\overline{G.f} \not\subset V(I)$
 if and only if
 \begin{align*}
\C[\fing^*]=I+ J_{\chi}
 \end{align*}
where $V(I)$ is the zero locus of $I$ in $\fing^*$.
\end{lemma}

Let $\l$ be a Levi subalgebra of $\g$ 
and $\h$ a Cartan subalgebra of $\g$ 
contained in $\l$. Thus $\z(\l) \subset \h$. 
Let $\mf{p}$ be a parabolic subalgebra of $\g$ with Levi factor $\l$ 
and nilradical $\mf{p}_u$.  
Assume that $e \in (\mf{p}_u)^{\rm reg}$ and 
$h\in\h$. Identifying $\g$ with $\g^*$ through $(~|~)$, we get 
$$\Slo_f \cong f + \g^{e},$$ 
and by \cite[Lemma 3.2]{Kat} (see also \cite[Proposition 3.2]{Bu}), 
we have 
$$\SS_\l = G.(f+\z(\l)).$$ 
Note that we have the following decomposition:  
$$\g(h,0) =[f,\g(h,2)] \oplus (\g(h,0) \cap \g^{e}),$$ 
and since ad$(f)$ induces a bijection from $\g(h,2)$ to $[f,\g(h,2)]$, 
for any $x\in\g(h,0)$ there is a well-defined element 
$\eta(x) \in \g(h,2)$ such that 
$$x - [f,\eta(x)] \in \g(h,0) \cap \g^{e}.$$ 

\begin{lemma} \label{lem:Kat} 
Assume that $\z(\l)$ is generated by a nonzero element $\lambda$ 
of $\h$ and that $\g(h,i)=0$ for $i >2$. 
\begin{enumerate}
\item The set $\{ \exp(\ad \eta(t \lam)) (f+t \lam) \; |\; t \in\C\}$ 
is an irreducible component of $\SS_\l \cap \Slo_f$. 
Moreover, if $\g$ is classical, then $\SS_\l \cap \Slo_f$ is irreducible and 
$\SS_\l \cap \Slo_f =\overline{\SS_\l} \cap \Slo_f = \{ \exp(\ad \eta(t \lam)) (f+t \lam) \; |\; t \in\C\}.$
\item If $\g$ is classical then, as schemes,  
$\overline{\SS_\l} \cap \Slo_f \cong \Spec \C[z]$. 
Here we endow $\overline{\SS_\l}$ with its natural structure of irreducible 
reduced scheme. 
\end{enumerate}
\end{lemma}

\begin{proof}
(1) The first assertion results from \cite[Lemma 2.9]{ImHof} 
and the proof of \cite[Lemma 3.4]{Bu} (which is a reformulation of 
\cite[Lemma 3.2 and Lemma~5.1]{Kat}). 
In the case that $\g$ is classical, $\SS_\l \cap \Slo_f$ 
is irreducible (see \cite[Theorem~5.2 and Theorem~6.2]{ImHof}) 
and so 
$\SS_\l \cap \Slo_f = \{ \exp(\ad \eta(t \lam)) (f+t \lam) \; |\; t \in\C\}$. 
On the other hand, by Lemma \ref{Lem:sheet}(1), 
\begin{align*}
\overline{\SS_\l} = \overline{G.\C^*\lambda} = G.\C^*\lambda  \cup 
\overline{G.f}
\end{align*}
and 
\begin{align*}
{\SS_\l} = G.\C^*\lambda \cup {G.f}.
\end{align*}
So, 
$\overline{\SS_\l} \cap \Slo_f = \{ \exp(\ad \eta(t \lam)) (f+t \lam) \;
 |\; t \in\C\}$
since $\overline{G.f} \cap \Slo_f = \{f\}$ and 
$f \in \{ \exp(\ad \eta(t \lam)) (f+t \lam) \; |\; t \in\C\}$. 
In particular, $\overline{\SS_\l} \cap \Slo_f$ is an irreducible 
variety of dimension one. 

(2) According to \cite[Corollary 1.3.8(1)]{Ginz}, 
$\overline{\SS_\l} \cap \Slo_f$ is a reduced complete intersection in 
$\overline{\SS_\l}$ since $\overline{\SS_\l}$ is reduced, 
whence $\overline{\SS_\l} \cap \Slo_f
\cong \Spec \C[z]$ as a scheme by~(1).  
\end{proof}

\begin{Rem}  \label{rem:Kat} 
Assume that $\g$ is classical. 
Define a one-parameter subgroup $\tilde{\gamma} : \C^* \to G$ 
by: 
$$\forall\, t\in \C^*, \; \forall\, x \in \g, \qquad \tilde{\gamma}(t).x : = t^2 \gamma(t).x$$
where $\gamma(t)$ is the one-parameter of $G$ generated by ${\rm ad}(h)$. 
In the notation of Lemma~\ref{lem:Kat}, 
the sets $\SS_\l$ and $\Slo_{f}$ are both stabilized by $\tilde{\gamma}(t)$, 
and we have: 
$$\overline{\SS_\l} \cap \Slo_{f}=\SS_\l \cap \Slo_{f} = \overline{\{\tilde{\gamma}(t).\mu \; |\; t \in \C^*\}},$$ 
for any nonzero semisimple element $\mu$ in $\SS_\l \cap \Slo_{f}$, 
for example 
$$\mu = \exp(\ad \eta(z)) (f+\lambda).$$ 
\end{Rem}

Let $\Omega$ be the Casimir element of 
$\g$ and denote by $I_\Omega$ the ideal of $\C[\g^*]$ 
generated by $\Omega$. 

 \begin{lemma} \label{lem:prime-criterion} 
 Assume that $\g$ is classical. 
  Let $I$ be a homogeneous $\ad \fing$-invariant ideal of $\C[\fing^*]$,
 $\SS_{\l}$ a Dixmier sheet of rank one,
  $\mf{p}$ a parabolic subalgebra of $\g$ with Levi factor $\l$
  and nilradical $\mf{p}_u$.
  Let $(e,h,f)$ be an $\mf{sl}_2$-triple of $\g$ such that $e\in
  (\mf{p}_u)^{reg}$. \color{black}
  Further, assume 
  that the following conditions are satisfied: 
  \begin{enumerate}
   \item $\g(h,i)=0$ for $i>2$,
\item $\on{supp}_{\C[\g^*]}(\C[\g^*]/I) = \overline{\SS_\l}$,  \label{old1}
\item $I+ I_{\Omega}$ is the defining ideal of $\overline{G.f}$, \label{old2}
\item {$H_f(\C[\g^*]/I) \cong \C[z]$}  {as algebras}, \label{old3} 
\item $\Omega(\lambda) \not=0$. \label{old4}
  \end{enumerate}
Then $I$ is prime, that is, $I=\sqrt{I}$. 
\end{lemma}

Condition (\ref{old1}) implies 
that $\sqrt{I}$ is defining ideal of $ \overline{\SS_\l}$ 
since $ \overline{\SS_\l}$ is irreducible. 
In particular, $\sqrt{I}$ is prime. 
Also, 
condition (\ref{old2}) means that $I+I_\Omega$ is prime.  
Note that conditions (\ref{old1}) and (\ref{old2}) imply that  
$$\sqrt{I} \subset I+I_\Omega.$$ 
since $\overline{G.f} \subset \overline{\SS_\l}$. 

 \begin{proof}
Set $J:=\sqrt{I}$. Then $J/I \in \overline{\mc{HC}}$. 
Since the sequence 
$$0 \to J/I \to \C[\g^*]/I \to \C[\g^*]/J \to 0$$ 
is exact, we get an exact sequence 
$$0 \to H_f (J/I)  \to H_f ( \C[\g^*]/I) \to 
H_f (\C[\g^*]/J )\to 0$$
by Theorem \ref{Th:ginzburg} (1). 
Furthermore by Theorem \ref{Th:ginzburg} (2), 
$$\Spec H_f (\C[\g^*]/J )= \Spec (\C[\g^*]/J ) 
\cap \Slo_{f} = \Spec \C[z].$$ 
The last equality comes from Lemma \ref{lem:Kat}~(2) 
since $J$ is the defining ideal of $\overline{\SS_\l}$. 
Hence by condition (\ref{old3}), we get 
$$ H_f(J/I)= 0.$$
By  \cite[Proposition 3.3.6]{Ginz}, 
$H_f(J/I)\ne 0$ if and only if
   $\on{supp}_{\C[\g^*]} (J/I) \supset \overline{G.f}$. 
However, $\on{supp}_{\C[\g^*]} (J/I) \subset 
\on{supp}_{\C[\g^*]} (\C[\g^*]/I) = \overline{\SS_\l}$ 
and any $G$-invariant closed cone of $\overline{\SS_\l}$ 
which strictly contains $G.f$ contains $\overline{G.f}$.  
Therefore, 
$$\on{supp}_{\C[\g^*]} (J/I) \subset G.f$$
since $\on{supp}_{\C[\g^*]} (J/I)$ 
is a $G$-invariant closed cone of $\g$ 
and $\overline{\SS_\l}=G.\C^* \lam \cup \overline{{\rm Ind}_{\l}^{\g}(0)}$ 
by Lemma \ref{Lem:sheet} (1), with $\lam \in \z(\l) \setminus \{0\}$.  
In particular, $\on{supp}_{\C[\g^*]} (J/I)$ is contained in the nilpotent 
cone $\mc{N}$. Since $\Omega$ is a nonzero homogeneous 
  element in 
  the defining ideal of $\mc{N}$, we deduce 
that $\Omega$ acts nilpotently on $J/I$. 
Hence for $n$ sufficiently large, 
\begin{align} \label{eq:Omega}
\Omega^n J/I = 0. 
\end{align}

We can now achieve the proof of the lemma. 
We have to show that $J \subset I$. 
Let $a\in J$. Since $J \subset I+I_\Omega$, for some $b_1 \in 
I$ and $f_1 \in \C[\g^*]$, we have 
$$a= b_1+ \Omega f_1.$$ 
Since $J$ is prime and $\Omega \not\in J$ by condition (\ref{old4}), $f_1\in J$. 
Applying what foregoes to the element $f_1$ of $J$, 
we get that for some $b_2 \in I$ and $f_2 \in J$, 
$$a = b_1 + \Omega (b_2+ \Omega f_2)= 
c_2+ \Omega^2 f_2,$$
with $c_2:= b_1+ \Omega b_2 \in I$. 
A rapid induction shows that for any $n\in \Z_{>0}$, 
there exist $c_n \in I$ and $f_n \in J$ such that
$$a = c_n +  \Omega^n f_n.$$ 
But $c_n +  \Omega^n f_n \in I$ for $n$ big enough by 
(\ref{eq:Omega}), whence $a \in I$.  
 \end{proof}

\section{Associated variety and singular support of affine vertex
 algebras}
\label{sec:associated-variety}
Let $$\affg=\fing[t,t^{-1}]\+ \C K \+ \C D$$be the affine Kac-Moody Lie algebra
associated with $\fing$ and $(~|~)$, with the commutation relations
\begin{align*}
 [x(m),y(n)]=[x,y](m+n)+ m(x|y)\delta_{m+n,0}K,\quad
 [D,x(m)]=mx(m),\quad
 [K,\affg]=0,
\end{align*}
for $m,n\in\Z$ and $x,y \in \g$, where $x(m)=x\otimes t^m$.
For $k\in \C$, set
\begin{align*}
 V^k(\fing)=U(\affg)\*_{U(\fing[t]\+ \C K \+ \C D)}\C_k,
\end{align*}
where 
$\C_k$ is the one-dimensional representation of
$\fing[t]\+ \C K {\+ \C D} $ on which $\fing[t]\+ \C D$ acts trivially and $K$ 
acts as multiplication by $k$.
The space $V^k(\fing)$ is naturally a vertex algebra, 
and it is called the {\em universal affine vertex algebra associated with
$\fing$ at level $k$}. 
By the PBW theorem, $V^k(\fing) \cong U(\fing[t^{-1}]t^{-1})$ as $\C$-vector spaces. 

The vertex algebra $V^k(\fing)$ is naturally graded:
\begin{align*}
V^k(\fing)
 =\bigoplus_{d\in \Z_{\geq 0}}V^k(\fing)_d,\quad V^k(\fing)_d=\{a\in
 V^k(\fing)\mid { D a = - da}\},
\end{align*} 
Let $V_k(\fing)$ be the unique simple graded quotient of $V^k(\fing)$.
As a $\affg$-module, $V_k(\fing)$ is isomorphic to the irreducible highest weight
representation of $\affg$ with highest weight $k\Lam_0$, 
where $\Lam_0$ is the dual element of $K$. 

A $V^k(\fing)$-module is the same as a smooth $\affg$-module of level
$k$,
where a $\affg$-module $M$ is called smooth if $x(n)m=0$ for $n\gg 0$ for all $x\in \g$, $m\in M$.

As in the introduction, let $X_V$ be the {\em associated variety} \cite{Ara12} of a vertex algebra $V$,
which is the maximum 
spectrum of the {\em Zhu $C_2$-algebra} 
$$ R_V := V/C_2(V)$$
of $V$.
In the case that $V$ is a quotient of $V^k(\fing)$,
$V/C_2(V)=V/\fing[t^{-1}]t^{-2} V$ 
and we have a surjective Poisson algebra homomorphism
\begin{align*}
 \C[\fing^*]=S(\fing)\twoheadrightarrow R_V=V/\fing[t^{-1}]t^{-2}V,\quad
x\mapsto \overline{x(-1)}+\fing[t^{-1}]t^{-2}V,
\end{align*}
where $\overline{x(-1)}$ denotes the image of $x(-1)$ in the quotient 
$V$. 
Then $X_V$ is just the zero locus of the kernel of the above map in
$\fing^*$. It is $G$-invariant and conic.

   \begin{Conj}\label{Conj:equidim}
    Let $V{=\oplus_{d \geq 0} V_d}$ be a simple, finitely strongly generated (i.e., $R_V$ is finitely generated), positively graded 
    conformal vertex algebra such that $V_0 \cong \C$.
    \begin{enumerate}
     \item $X_V$ is equidimensional. 
     \item
	      Assume that $X_V$ has finitely many symplectic leaves. 
    Then $X_V$  is irreducible.
    In particular $X_{V_k(\fing)}$ is irreducible if $X_{V_k(\fing)}\subset \mc{N}$.

    \end{enumerate}
   \end{Conj}

   For a scheme $X$ of finite type, let $J_m X$ be the $m$-th jet scheme of $X$, 
and $J_\infty X$ the infinite jet scheme of $X$ (or the arc space of $X$). 
Recall that the scheme $J_m X$ is determined by its functor of points: 
for every $\C$-algebra $A$, there is a bijection 
$$\Hom(\Spec A, J_m X) \cong \Hom (\Spec A[t]/(t^{m+1}),X).$$
If $m >n$, we have a natural morphism $J_m X \to J_n X$. 
This yields a projective system $\{J_m X\}$ of schemes, 
and the infinite jet scheme $J_\infty X$ is the projective limit $\varprojlim J_m X$ 
in the category of schemes. 
Let  $\pi_{m} \colon  J_{m} X \to X$, $m >0$, and  
$\pi_{\infty} \colon  J_{\infty} X \to X$  be
the natural morphisms.  

If $X$ is an affine Poisson scheme then
its coordinate ring
$\C[J_{\infty}X]$ is naturally a Poisson vertex algebra (\cite{Ara12}).

Let $V=F^0V\supset F^1 V\supset \dots $ be the canonical decreasing filtration 
of the 
vertex algebra $V$ defined by Li \cite{Li05}.
The associated graded algebra $\gr V=\bigoplus_{p\geq 0}F^pV/F^{p+1}V$
is naturally a Poisson vertex algebra.
In particular,
 it has the structure of a commutative algebra.
 We have  $F^1 V=C_2(V)$ by definition, 
 and by restricting the Poisson vertex algebra structure of $\gr V$
 we obtain the Poisson structure of $R_V=V/F^1 V\subset \gr V$. 
 {There is} a surjection
 \begin{align*}
  \C[J_{\infty}X_V]\ra \gr V
 \end{align*}
 of  Poisson vertex  algebras
(\cite{Li05,Ara12}).
 By definition \cite{Ara12},
 the  {\em singular support} of $V$
is the subscheme 
\begin{align*}
 SS(V)=\on{Spec}(\gr V)
\end{align*}
of $J_{\infty}\tilde X_V$.

\begin{Th}  \label{Th:arc}
 Let $V$  be a quotient
 vertex algebra $V^k(\fing)$.
 Suppose that 
 $X_V=\overline{G.\C^* x}$ for some $x\in \fing$ 
 and that $J_{\infty} \C^*.x$ is contained in the reduced singular 
 support $SS(V)_{\on{red}}$, where the action of $J_{\infty} \C^*$ 
 on $J_{\infty} \g$ is 
 induced by the $\C^*$-action on $\g$. 
 Then
$$SS(V)=J_{\infty}
 X_V=J_{\infty}\overline{G.\C^*x}=
 \overline{J_{\infty}G.\C^*x},$$
 as topological spaces. 
 In particular, $SS(V)_{\on{red}}=(J_{\infty}
 X_V)_{\on{red}}$. 
\end{Th}

\begin{proof}
By \cite[Lemma 3.3.1]{Ara12}, 
 $\tilde{X}_V = \pi_{\infty} (SS(V))$.
 We know that
 $SS(V) \subset J_\infty \tilde{X}_V$,
 so that   $SS(V)_{\on{red}} \subset (J_\infty X_V)_{\on{red}}$.
Let us prove the other inclusion. 

Set $X=X_V$ and $U=G.\C^*x$. 
Since $U$ is an irreducible open dense subset of 
$X$, $\pi_{m}^{-1}(U) = J_{m} U$ for any $m >0$ \cite[Lemma 2.3]{EM}, 
and $\overline{\pi_{m}^{-1}(U)} = \overline{\pi_m^{-1}(X_{\rm reg})}$ 
is an irreducible component of $J_m X$. 
Hence $$J_\infty X=
\overline{\pi_\infty^{-1}(U)}=\overline{J_\infty U }$$ 
because $J_\infty X$ is irreducible \cite{Kol} {and closed}. 
Therefore, it is enough to prove that $SS(V)_{\rm red}$ contains $J_{\infty} U$ 
since $SS(V)_{\rm red}$ is closed. 

The map 
$$\mu \colon G \times (\C^* x) \to G.\C^*x,\quad  (g,tx ) \mapsto g . (t x)$$
is a submersion at each point, so it is smooth (cf.~\cite[Ch.~III, 
Proposition 10.4]{Ha}). 
Hence by \cite[Remark 2.10]{EM}, we get that the induced map 
$$\mu_\infty \colon J_\infty G \times J_\infty (\C^* x)  \to J_\infty (G.\C^*x)$$ 
is surjective and formally smooth, and so
$$\mu_\infty (J_\infty G \times J_\infty (\C^*.x)) \cong J_\infty (G.\C^*x).$$
Since $SS(V)_{\rm red}$ is $J_\infty G$-invariant and contains $J_\infty \C^*.x$,  
we deduce that $J_\infty (G. \C^* x) = J_\infty U $ is  contained in $SS(V)_{\rm red}$. 
\end{proof}



\section{Zhu's algebra of affine Vertex algebras}   \label{sec:Zhu}
For a $\Z_{\geq 0}$-graded vertex algebra $V=\bigoplus_{d}V_d$, 
let $A(V)$ be  the Zhu algebra of $V$, 
\begin{align*}
A(V)=V/V\circ V,
\end{align*}
where $V\circ V$ is the $\C$-span of the vectors
\begin{align*}
a\circ b:=\sum_{i\geq 0}\begin{pmatrix}
			 \Delta\\ i
			\end{pmatrix}a_{(i-2)}b
\end{align*}for $a\in V_{\Delta}$, $\Delta\in \Z_{\geq 0}$, $b \in V$,
and
$V\ra (\End V)[[z,z^{-1}]]$,  $a\mapsto \sum_{n\in \Z}a_{(n)}z^{-n-1}$,
denotes the state-field correspondence.
The space $A(V)$ is a unital associative algebra
with respect to the multiplication defined by
\begin{align*}
 a * b:=\sum_{i\geq 0}\begin{pmatrix}
			 \Delta\\ i
			\end{pmatrix}a_{(i-1)}b
\end{align*}
for $a\in V_{\Delta}$, $\Delta\in \Z_{\geq 0}$, $b \in V$.

Let $M=\bigoplus_{d\geq d_0}M_d$, $M_{d_0}\ne 0$,
be a positive energy representation of $V$. 
Then $A(V)$ naturally acts on its top weight space
$M_{top}:=M_{d_0}$,
and
the correspondence $M\mapsto M_{top}$ defines a bijection between
isomorphism
classes of
simple positive energy representations of $V$ and simple $A(V)$-modules (\cite{Zhu96}).

The vertex algebra $V$ is called a {\em chiralization} of an algebra $A$ if
$A(V)\cong A$.

For instance, 
consider the universal affine vertex algebra $V^k(\fing)$.
The Zhu algebra
 $A(V^k(\fing))$ is naturally isomorphic to $U(\fing)$ (\cite{FreZhu92},
 see also \cite[Lemma 2.3]{A12-2}),
and hence, 
$V^k(\fing)$ is a chiralization of $U(\fing)$.

Let $\widehat{\mc{J}}_k
$ be the unique maximal ideal of $V^k(\fing)$, so that 
\begin{align*}
 V_k(\fing)=V^k(\fing)/\widehat{\mc{J}}_k.
\end{align*}
Since the functor $A(?)$ is
right exact,
$A(V_k(\fing))$ is the quotient of $U(\fing)$ by the image
$\mc{I}_k$
of $\widehat{\mc{J}}_k$ in $A(V^k(\fing))=U(\fing)$:
\begin{align*}
 A(V_k(\fing))=U(\fing)/\mc{I}_k.
\end{align*}

Fix a triangular decomposition  $\fing=\mf{n}_-\+\mf{h}\+\mf{n}_+$ of
$\fing$.
Then $\widehat{\mf{h}}=\mf{h}\+ \C K {\+ \C D}$ is a Cartan subalgebra of
$\affg$.
A weight $\lam\in \widehat{\h}^*$ is called of level $k$
if $\lam(K)=k$.
The top degree component 
of $L(\lam)$ is 
 $L_{\fing}(\bar \lam)$,
where $\bar \lam $ is the 
restriction of $\lam$ to $\mf{h}$.
Hence,
by Zhu's Theorem,
the level $k$ representation 
$L(\lam)$ is a $V_k(\fing)$-module if and only if
$\mc{I}_kL_{\fing}(\bar\lam)=0$.

Set $U(\g)^\h := \{u \in U(\g) \; | \; [h,u]=0 \text{ for all } h \in \h \}$ and 
let 
$$\Upsilon \colon U(\g)^\h \to U(\h)$$
be the {\em Harish-Chandra projection map} which is the 
restriction of the projection map $U(\g)=U(\h) \oplus (\n_- U(\g) + U(\g)\n_+) \to U(\h)$ 
to $U(\g)^\h$. It is known that $\Upsilon$ is an algebra homomorphism. 
For a two-sided ideal $I$ of $U(\g)$, the {\em characteristic variety of $I$} 
(without $\rho$-shift) is defined as 
$$\mc{V}(I) = \{\lambda \in \h^* \; | \; p(\lam)=0 \text{ for all } p\in \Upsilon(I^\h)\}$$
where $I^\h =I\cap U(\g)^\h$, cf.~\cite{Jos}. 
Identifying $\g^*$ with $\g$ through $(~|~)$, and thus $\h^*$ with $\h$, 
we view $\mc{V}(I)$ as a subset of $\h$.

 \begin{Pro}[{\cite[Proposition 2.5]{A12-2}}]\label{Pro:ch-variety-vs-Zhu}
  For a level $k$ weight $\lam\in \widehat{\mf{h}}^*$, 
  $L(\lam)$ is a $V_k(\fing)$-module if and only if
  $\bar \lam\in \mc{V}(\mc{I}_k)$.
\end{Pro}

The Zhu algebra $A(V)$ is related with 
 the Zhu $C_2$ algebra $R_V$
as follows.
The grading of $V$ induces a filtration of $A(V)$ 
and the associated graded algebra $\gr A(V)$ is naturally a Poisson algebra 
(\cite{Zhu96}).
There is a natural surjective homomorphism
\begin{align*}
\eta_V: R_V\twoheadrightarrow \gr A(V)
\end{align*}
of Poisson algebras (\cite[Proposition 2.17(c)]{De-Kac06}, \cite[Proposition 3.3]{ALY}). 
In particular, $\Spec (\gr A(V))$ is a subscheme of $\tilde X_V$.

For $V=V_k(\fing)$ this implies the following.
We have
\begin{align*}
 V(\gr \mc{I}_k)
 \subset X_{V_k(\fing)},\end{align*}
where $V(\gr \mc{I}_k)$ is the zero locus of $\gr \mc{I}_k\subset
\C[\fing^*]$
in $\fing^*$.

The
map $\eta_{V_k(\fing)}$ is not necessarily an isomorphism. However, 
conjecturally \cite{A2012Dec} we have
$ V(\gr \mc{I}_k)
= X_{V_k(\fing)}$.

\section{Affine $W$-algebras}\label{sec:affine-W-algebras}
For a nilpotent element  $f$ of $\fing$,
let $\W^k(\fing,f)$ be the $W$-algebra associated with
$(\fing,f)$ at level $k$,
defined by the generalized quantized Drinfeld-Sokolov reduction
\cite{FF90,KacRoaWak03}:
\begin{align*}
 \W^k(\fing,f)=H^{\frac{\infty}{2}+0}_f
 (V^k(\fing)).
\end{align*}
Here $H^{\frac{\infty}{2}+0}_f(M)$ is the corresponding BRST
cohomology with coefficients in a smooth $\affg$-module $M$.

Let $(e,h,f)$ be an $\mf{sl}_2$-triple associated with $f$.
The $W$-algebra $\W^k(\fing,f)$ is conformal provided that
$k\ne -h^{\vee}$, which we assume in this paper,
and
the central charge $c_f(k)$ of $\W^k(\fing,f)$
is given by
\begin{align*}
 c_f(k)=\dim \fing(h,0)-\frac{1}{2}\dim \fing(h,1)
 -12\left|\frac{\rho}{\sqrt{k+h^{\vee}}}-\sqrt{k+h^{\vee}}\frac{
 h}{2} \right|^2,
\end{align*}
where $\rho$ is the half sum of the positive roots of $\fing$.
We have
\begin{align*}
 \W^k(\fing,f)&=\bigoplus_{\Delta\in \frac{1}{2}\Z_{\geq
 0}}\W^k(\fing,f)_{\Delta},\\& \W^k(\fing,f)_0=\C,\quad 
 \W^k(\fing,f)_{1/2}=0,\quad 
 \W^k(\fing,f)_1
 \cong \fing^{\natural}, 
\end{align*}
where
$\fing^{\natural}$ is the centralizer in $\fing$ of the $\mf{sl}_2$-triple
$(e,h,f)$.

We have \cite{De-Kac06,Ara09b} a natural isomorphism
$ R_{\W^k(\fing,f)}\cong \C[\Slo_{f}]$ of Poisson algebras, so that
\begin{align*}
 \tilde{X}_{\W^k(\fing,f)}= \Slo_{f}.
\end{align*}
Let $\W_k(\fing,f)$ be the unique simple quotient of
$\W^k(\fing,f)$.
Then
$X_{\W_k(\fing,f)}$
 is a $\C^*$-invariant Poisson
subvariety of $\Slo_f$.
Since it is $\C^*$-invariant,
$\W_k(\fing,f)$ is lisse if and only if
$X_{\W_k(\fing,f)}=\{f\}$.

Let $\mc{O}_k$ be the category $\mc{O}$ of
$\affg$ at level $k$.
We have a functor
\begin{align*}
 \mc{O}_k\ra \W^k(\fing,f)\on{-Mod}
 ,\quad M\mapsto
 H^{\frac{\infty}{2}+0}_f(M),
\end{align*}
where
$\W^k(\fing,f)\on{-Mod}$ denotes the category
of $\W^k(\fing,f)$-modules.

Let $\on{KL}_k$ be the full subcategory of $\mc{O}_k$ consisting of
objects $M$ on which $\fing$ acts  locally finitely.
Note that $V^k(\fing)$ and $V_k(\fing)$ are objects of $\on{KL}_k$.
 \begin{Th}[{\cite{Ara09b}, $k$, $f$ arbitrary}]\label{Th:W-algebra-variety}
 \begin{enumerate}
  \item $H_{f}^{\frac{\infty}{2}+i}(M)=0$ for all $i\ne 0$, $M\in
	\on{KL}_k$.
	In particular, the functor
	$\on{KL}_k\ra \W^k(\fing,f)\on{-Mod}$, $M\mapsto
	H_{f}^{\frac{\infty}{2}+0}(M)$, is exact.

\item
     For any quotient $V$ of $V^k(\fing)$ we have
\begin{align*}
R_{H_{f}^{\frac{\infty}{2}+0}(V)}\cong H_{f}(R_{V}),
\end{align*}
     where $H_{f}(R_{V})$ is defined in
     Theorem \ref{Th:ginzburg}.
     Hence,
     $\tilde{X}_{H_f^{\frac{\infty}{2}+0}(V)}$ is isomorphic to the
     scheme theoretic intersection
     $\tilde{X}_{V}\times_{\fing^*} \Slo_f$.
     \label{item:intersection}
     \item
	  $H_{f}^{\frac{\infty}{2}+0}(V)
	  \ne 0$ if and only if
	  $\overline{G.f}\subset X_V$.
  \item
       $H_{f}^{\frac{\infty}{2}+0}(V)$ is lisse if
       $X_V=       \overline{G.f}$.
 \end{enumerate}
\end{Th}

Let $f_{\theta}$ be  a root vector of the highest root $\theta$ of $\g$,
which is a minimal nilpotent element so that $f \in \O_{\min}$ 
where $\O_{min}$ is the minimal nilpotent orbit of $\g$.
For the first statement
of  Theorem \ref{Th:W-algebra-variety},
we have the following stronger statement
in the case that $f=f_{\theta}$.
\begin{Th}[{\cite{Ara05}, $f=f_{\theta}$}]\label{Th:minimal}
\begin{enumerate}
 \item $H^{\frac{\infty}{2}+i}_{f_{\theta}}(M)
       =0$ for all $i\ne 0$,
       $M\in \mc{O}_k$.
       	In particular, the functor
	$\mc{O}_k\ra \W^k(\fing,f)\on{-Mod}$, $M\mapsto
	H_{f}^{\frac{\infty}{2}+0}(M)$, is exact.
 \item $H^{\frac{\infty}{2}+0}_{f_{\theta}}(L(\lam))
       \ne 0$ if and only
       if $\lam(\alpha_0^{\vee})\not\in \Z_{\geq 0}$.
       If this is the case, 
       $H^{\frac{\infty}{2}+0}_{f_{\theta}}(L(\lam))$ is a simple
       $\W^k(\fing,f)$-module,
       where $\alpha_0^{\vee}=-\theta^{\vee}+K$.
\end{enumerate} 
\end{Th}

Recall that
$f$ is a {\em short} nilpotent element
if
\begin{align*}
\fing=\fing(h,-2)\+\fing(h,0)\+ \fing(h,2).
\end{align*}
If this is the case,
we have $\frac{1}{2}h
\in P^{\vee}$, 
where $P^{\vee}$ is the coroot lattice of $\fing$,
and
 $\frac{1}{2}h$
 defines an element of the extended affine Weyl group
 $\widetilde{W}=W\ltimes P^\vee$ of $\affg$,
which we denote by 
$t_{\frac{1}{2}h}$.
 Here $W$ is the Weyl group of $\fing$.
Let $\tilde{t}_{\frac{1}{2}h}$ be a Tits lifting
of $t_{\frac{1}{2}h}$ to an automorphism of $\affg$.

Set
\begin{align*}
 D_h=D+\frac{1}{2}h,
\end{align*}
and put
$M_{d,h}=\{m\in M\mid D_h m=d m\}$ for
a $\affg$-module $M$.
The operator  $D_h$ extends to the
grading operator of $\W^k(\fing,f)$ (\cite{KacRoaWak03,Ara05}).

The subalgebra
$\tilde{t}_{\frac{1}{2}h}(\fing)
$ acts on 
each homogeneous
component $M_{d,h}$
because $[D_h,t_{\frac{1}{2}h}(\fing)]=0$.
Note that
$\tilde{t}_{\frac{1}{2}h}(\fing)$
is  the subalgebra of $\affg$ generated
by the 
root vectors of the roots
$t_{\frac{1}{2}h}(\alpha)=\alpha-\frac{1}{2}\alpha(h)
\delta$, 
$\alpha\in \Delta$,
where $\delta\in \widehat{\h}^*$ is the dual element of $D$.
In particular,
$\tilde{t}_{\frac{1}{2}h}(\fing)$
 contains
\begin{align*}
\mf{l}:=\fing(h,0)
\end{align*}
 as a Levi subalgebra.

We regard
$M_{d,h}$ as a $\fing$-module through the isomorphism
 $\tilde{t}_{\frac{1}{2}h}(\fing)\cong\fing$.

 Since
 $\alpha(D_h)\geq 0$ for all positive roots  $\alpha$ of $\affg$
by the assumption that $f$ is short,
we have
\begin{align*}
 L(\lam)=\bigoplus_{d\leq \lam(D_h)}L(\lam)_{d,h}.
\end{align*}

Let
$\mc{O}_{k}^{\mf{l}}$ be the full subcategory of category
$\mc{O}_k$ of $\affg$ consisting
of objects on which $\mf{l}$ acts locally finitely.
\begin{Th}\label{Th:short}
 Let $f$ be a short nilpotent element as above.
\begin{enumerate}
 \item  $H^{\frac{\infty}{2}+i}_f(M)=0$
	for all $i\ne 0$, $M\in
	\mc{O}_k^{\mf{l}}$.
	       	In particular, the functor
	$\mc{O}_k^{\mf{l}}\ra \W^k(\fing,f)\on{-Mod}$, $M\mapsto
	H_{f}^{\frac{\infty}{2}+0}(M)$, is exact.
 \item Let $L(\lam)\in \mc{O}_k^{\mf{l}}$.
       Then
    $H^{\frac{\infty}{2}+0}_f(L(\lam))
       \ne 0$ if and only if
       $\on{Dim} L(\lam)_{\lam(D_h),h}=1/2 \dim G.f
       (=\dim \fing(h,2))$, 
       where $\on{Dim} M$ is the Gelfand-Kirillov dimension
       of
       the  $\fing$-module
 $M$.
If this is the case,
       $H^{\frac{\infty}{2}+0}_f(L(\lam))$ is
       almost irreducible over
       $\W^k(\fing,f)$,
       that is,
       any nonzero submodule
       of  $H^{\frac{\infty}{2}+0}_f(L(\lam))$
       intersects
       its top weight {component}
       non-trivially (cf.\ \cite{Ara08-a}). 
 \item Suppose that $\fing$ is of type $A$.
Then, for $L(\lam)\in \mc{O}_k^{\mf{l}}$,
       $H^{\frac{\infty}{2}+0}_f(L(\lam))$ is zero or irreducible.
\end{enumerate}
\end{Th}
\begin{proof}
The above theorem is  just a  restatement of main results of \cite{Ara08-a}.
Indeed,
as $f$ is short,
$\W^k(\fing,f)$
is integer graded.
Therefore the
Ramond twisted representations of $\W^k(\fing,f)$ studied in
 \cite{Ara08-a}
 are nothing but the usual representations.
 Considering the action of $\affg$ 
 defined by $x\mapsto \tilde{t}_{\frac{1}{2}h}(x)$,
the category $\mc{O}_{k}^{\mf{l}}$
can be identified with the category 
 $\mathcal{O}_{0,k}$ of  \cite[5.5]{Ara08-a}.
 Moreover,
 the functor
 	$\mc{O}_k^{\mf{l}}\ra \W^k(\fing,f)\on{-Mod}$, $M\mapsto
 H_{f}^{\frac{\infty}{2}+0}(M)$, gets  identified with the
 ``$-$''-reduction functor $H^{\on{BRST}}_0(?)$ in  \cite[Theorem 5.5.3]{Ara08-a}.
\end{proof}

\section{Level $-1$ affine vertex algebra of type $A_{n-1}$, $n\geq 4$}
\label{sec:A-theta1}
We assume in this section that $\g=\sl_n$ with $n
\geq 4$.

Let $$\Delta=\{\eps_i -\eps_j \; | \; i,j= 1,\ldots,n, i\not=j\}$$ be the root 
system of $\g$. Fix the set of positive roots 
$\Delta_+=\{\eps_i -\eps_j \; | \; i,j= 1,\ldots,n, i < j\}$. 
Then the simple roots are $\alpha_i = \eps_i -\eps_{i+1}$ 
for $i=1,\ldots,{n }-1$. The highest root is $\theta=\eps_1-\eps_{n}= 
\alpha_1+\cdots+\alpha_{n-1}$. 
Denote by $(e_i,h_i,f_i)$ the Chevalley generators of $\g$, and 
fix the root vectors $e_\alpha,f_\alpha$, 
$\alpha \in \Delta_+$ as follows. 
\begin{align*} 
e_{\eps_i -\eps_j}  = e_{i,j} \quad 
\text{ and }\quad f_{\eps_i -\eps_j}  = e_{j,i} \quad \text{ for }\quad  i <j, 
\end{align*}
where $e_{i,j}$ is the standard elementary matrix associated with the coefficient $(i,j)$. 
For $\alpha \in \Delta_+$, denote by $h_\alpha =[e_\alpha,f_{\alpha}]$ 
the corresponding coroot. 
In particular, 
\begin{align*} 
h_i=e_{i,i}-e_{i+1,i+1} \quad \text{ for }\quad i=1,\ldots,n-1.
\end{align*} 

Let $\g= \n_- \oplus \h \oplus \n_+$ be the corresponding triangular decomposition. 
Denote by $\varpi_1,\ldots,\varpi_{n-1 }$ the fundamental 
weights of $\g$, 
$$\varpi_{i}:= (\eps_{1} +\cdots + \eps_{i}) - \frac{i}{n}(\eps_{1}+\cdots 
+ \eps_{n}).$$
Identify $\g$ with $\g^*$ through $(~|~)$. Thus, the fundamental 
weights are viewed as elements of $\g$. 

Let $\theta_1$ be the highest root of the root system generated 
by the simple roots perpendicular to $\theta$, i.e.,  
$$\theta_1=\alpha_2+\ldots +\alpha_{n-2}=\eps_2-\eps_{n-1},$$
Set 
$$\beta:= \alpha_1+\cdots+\alpha_{n-2}=\eps_1-\eps_{n-1}, \qquad 
\gamma:= \alpha_{2}+\cdots+\alpha_{n-1}=\eps_2-\eps_{n},$$
and put
$$v_1=e_{\theta} e_{\theta_1} - e_{\beta} e_{\gamma}\in S^2(\fing),$$
where $S^d(\fing)$ denotes the component of degree $d$ in $S(\fing)$ for $d\geq 0$.
Then $v_1$ is a singular vector with respect to the adjoint action
of $\fing$ and generates an irreducible finite-dimensional
representation $W_1$  of $\g$ in $S^2(\fing)$ 
isomorphic to $L_{\fing}(\theta+\theta_1)$. 

Recall that we have a $\fing$-module embedding \cite[Lemma 4.1]{AM15},
\begin{align}
\sigma_d \colon S^d(\fing)\hookrightarrow V^k(\fing)_d, 
\quad x_1\dots x_d\mapsto \frac{1}{d!}\sum_{\sigma\in
 \mathfrak{S}_d}x_{\sigma(1)}(-1)\dots x_{\sigma(d)}(-1),
 \label{eq:sigma}
\end{align}
where $ \mathfrak{S}_d$ is the $d$-order symmetric group.
We will denote simply by $\sigma$ this embedding for $d=2$. 

 \begin{Pro}[\cite{Ada03}] 
  For $l\geq 0$,
  The vector $\sigma(v_1)^{l+1}$ is {a}
  singular vector of
 $V^k(\fing)$ if and only if $k=l-1$.
\end{Pro}

    \begin{Th}\label{Th:irreduciblity}
The vector $\sigma(v_1)$ generates the maximal submodule of
     $V^{-1}(\fing)$,
that is,
$V_{-1}(\fing)=V^{-1}(\fing)/U(\affg)\sigma(v_1)$.
 \end{Th}
   Set
   \begin{align*}
  \tilde V_{-1}(\fing) =V^{-1}(\fing)/U(\affg)\sigma(v_1).
  \label{eq:XVtilde}
   \end{align*}
 To prove Theorem \ref{Th:irreduciblity}, 
 we shall use the minimal $W$-algebra $\W^k(\fing,f_{\theta})$.

 Let $\fing^{\natural}$ be the centralizer in $\g$ 
 of the $\mf{sl}_2$-triple $(e_{\theta},
 h_{\theta},f_{\theta})$. 
 Then
 \begin{align*}
  \fing^{\natural}=\fing_0\+\fing_1,
 \end{align*}
 where $\fing_0$ is the one-dimensional center of $\fing^{\natural}$
 and $\fing_1 =[\fing^{\natural},\fing^{\natural}]$.
 Note that
 $\fing_0=\C  (h_1-h_{n-1})$
 and
 $\fing_1=\langle e_{\alpha_i},f_{\alpha_i}\mid i=2,\dots,n-2\rangle
 \cong \mf{sl}_{n-2}$.

 There is an embedding, \cite{KacWak04}, \cite[\S7]{AM15}, 
\begin{align*}
 V^{k_0^{\natural}}(\fing_0)\otimes
 V^{k_1^{\natural}}(\fing_1)\hookrightarrow \W^k(\fing,f_{\theta})
\end{align*}
of vertex algebras,
where $k_0^{\sharp}=k+n/2$ and $k_1^{\natural}=k+1$.
Note that
$V^{k_0^{\natural}}(\fing_0)$ is isomorphic to the rank one Heisenberg
vertex algebra $M(1)$
provided that $k_0^{\natural}\ne 0$.

For $k=-1$,
we have
$k_1^{\natural}=0$ and
$V_{k_1^{\natural}}(\fing_1)=V^{k_1^{\natural}}(\fing)/{U(\widehat{\fing_1})
e_{\theta_1} (-1)\mathbf{1}}
\cong \C$.

By Theorem \ref{Th:minimal}
the exact sequence
$0\ra U(\affg)\sigma(v_1)\ra V^{-1}(\fing)\ra \tilde{V}_{-1}(\fing)\ra
0$ induces an exact sequence
\begin{align}
 0\ra H^{\frac{\infty}{2}+0}_{f_{\theta}}
(U(\affg)\sigma(v_1))\ra
 \W^{-1}(\fing,f_{\theta})
\ra H^{\frac{\infty}{2}+0}_{f_{\theta}}(\tilde{V}_{-1}(\fing))\ra 0.
\end{align}
  \begin{Lem}\label{Lem:singular-vector}
   The image of $\sigma(v_{1})$ in $\W^{-1}(\fing,f_{\theta})$
coincides with the image of the singular vector
  ${e_{\theta_1}(-1)\mathbf{1}}$
   of $V^0(\fing_1)\subset \W^{-1}(\fing,f_{\theta})$.
  \end{Lem}
   \begin{proof}
    Since it
    is singular in $V^{-1}(\g)$, $\sigma(v_1)$ defines a singular vector  of
    $\W^{-1}(\fing,f_{\theta})$.
    Its image in $R_{\W^{-1}(\fing,f_{\theta})}$
    is the image
    of $v_1$ in $\C[\Slo_{f_{\theta}}]=(\C[\fing^*]/J_{\chi})^M$,
   where $\chi=(f_{\theta}|\,\cdot\,)$, see \S\ref{section:Ginzburg}.    
Since
   $v_1\equiv e_{\theta_1}\pmod{J_{\chi}}$,
    we have  $\sigma(v_1)
    \equiv {e_{\theta_1}(-1)\mathbf{1}}
    \pmod{C_2(\W^{-1}(\fing,f_{\theta}))}$.
    Then
    $\sigma(v_1)$ and ${e_{\theta_1}(-1)\mathbf{1}}$
must coincide
since they 
    both
    have the
    same weight
    and are singular vectors
    with respect to the action of
    $\affg_1$.
   \end{proof}
\begin{Th} \label{Th:Heisenberg}
We have
$H^{\frac{\infty}{2}+0}_{f_{\theta}}(\tilde{V}_{-1}(\fing))
\cong
M(1) 
$, the rank one Heisenberg vertex algebra.
 In particular $ H^{\frac{\infty}{2}+0}_{f_{\theta}}(\tilde{V}_{-1}(\fing))$ is simple, and hence,
 isomorphic to $\W_{-1}(\fing,f_{\theta})$. \end{Th}
 \begin{proof}
By Theorem \ref{Th:minimal},
$H^{\frac{\infty}{2}+0}_{f_{\theta}}(V_{-1}(\fing))$ is
 isomorphic to $\W_{-1}(\fing,f_{\theta})$.
Since
$V_{-1}(\fing)$ is a quotient of $\tilde{V}_{-1}(\fing)$,
$\W_{-1}(\fing,f_{\theta})$
is a quotient 
of $H^{\frac{\infty}{2}+0}_{f_{\theta}}(\tilde{V}_{-1}(\fing))$
by the exactness result of  Theorem \ref{Th:minimal}.
In particular, 
$H^{\frac{\infty}{2}+0}_{f_{\theta}}(\tilde{V}_{-1}(\fing))$ is nonzero.

Recall that 
$\W^k(\fing,f_{\theta})$ is generated by fields
$J^{\{a\}}(z)$, $a\in \fing^{\natural}$,
$G^{\{v\}}(z)$, $v\in \fing_{-1/2}$,
and $L(z)$ described in \cite[Theorem 5.1]{KacWak04}.
Since $V^0(\fing_1)/U(\affg_1)e_{\theta_1}(-1)\mathbf{1}=\C$,
Lemma \ref{Lem:singular-vector}
implies that 
$J^{\{a\}}(z)$, $a\in \fing_1$, are all zero in 
$H^{\frac{\infty}{2}+0}_{f_{\theta}}(\tilde{V}_{-1}(\fing))$.
Then, since
{$[J^{\{a\}}_{\lam}G^{\{v\}}]=G^{[a,v]}$}
  and 
 $\fing_{-1/2}$ is a direct sum of non-trivial irreducible
 finite-dimensional representations
  of $\fing_1$,
it follows that $G^{(v)}(z)=0$ for all $v\in \fing_{-1/2}$.
Finally,
\cite[Theorem 5.1 (e)]{KacWak04} implies that 
$L$ coincides with the conformal vector of
  {$V^{\frac{n-2}{2}}(\fing_1)\cong M(1)$ } in
  $H^{\frac{\infty}{2}+0}_{f_{\theta}}(\tilde{V}_{-1}(\fing))$.
We conclude that 
$H^{\frac{\infty}{2}+0}_{f_{\theta}}(\tilde{V}_{-1}(\fing))$ is generated by
  $J^{\{a\}}(z)$, $a\in \fing_0$, which proves the assertion.
\end{proof}

 \begin{proof}[Proof of Theorem \ref{Th:irreduciblity}]
  Suppose that $\tilde{V}_{-1}(\fing)$ is not simple.
  Then there is at least one singular vector $v$ of weight, say $\mu$.
Then 
\begin{align*}
 \mu&\equiv s\varpi_1-\Lam_0 \quad \text{ or } 
 \quad s\varpi_{l}-\Lam_0 \quad\text{ for
 some }\quad  s\in \mathbb{N}\pmod{ \C \delta}
\end{align*} 
  by \cite[Proposition 5.5]{AP08}.
  Let $M$ be the submodule of $\tilde{V}_{-1}(\fing)$
  generated by $v$.
  Since $M$ is a $\tilde{V}_{-1}(\fing)$-module,
  $H^{\frac{\infty}{2}+0}_{f_{\theta}}(M)$ is a
  module over   $H^{\frac{\infty}{2}+0}_{f_{\theta}}(\tilde{V}_{-1}(\fing))=M(1)$.
  Because  $\mu(\alpha_0^{\vee})<0$,
  Theorem \ref{Th:minimal} implies that
  the image of $v$ in   $H^{\frac{\infty}{2}+0}_{f_{\theta}}(M)$
  is nonzero,
  and thus,
  it generates
   the irreducible highest weight representation
{$M(1,\frac{\mu(h_1-h_{n-1})}{\sqrt{2(n-2)}} )
 $ } of 
  $M(1)$ with highest weight
  {$ \frac{\mu(h_1-h_{n-1})}{\sqrt{2(n-2)}}$}.
Now the exactness of the functor 
  $H^{\frac{\infty}{2}+0}_{f_{\theta}}(?)$ 
  shows that
  $H^{\frac{\infty}{2}+0}_{f_{\theta}}(M)$ is a submodule of 
  $H^{\frac{\infty}{2}+0}_{f_{\theta}}(\tilde{V}_{-1}(\fing))=M(1)$,
  and therefore, so is  {$M(1,\frac{\mu(h_1-h_{n-1})}{\sqrt{2(n-2)}} )$}.
  But this   contradicts the simplicity of
  $M(1)$.
 \end{proof}
By Theorem \ref{Th:irreduciblity}, 
we have
\begin{align}
 R_{{V}_{-1}(\fing)}=\C[\fing^*]/I_{W_1},\quad \text{so}\quad
\tilde{X}_{{V}_{-1}(\fing)}=\Spec (\C[\fing^*]/I_{W_1}).
\end{align}
 \begin{Co}\label{Co:C[z]}
We have $H_f(\C[\fing^*]/I_{W_1}) \cong \C[z]$.
\end{Co}
   \begin{proof}
    By  Theorem \ref{Th:W-algebra-variety}\eqref{item:intersection},
    we know that 
   $R_{H^{\frac{\infty}{2}+0}_{f_{\theta}}(V_{-1}(\fing))}
   =H_f(R_{V_{-1}(\fing)})$.
The assertion follows since
    $R_{H^{\frac{\infty}{2}+0}_{f_{\theta}}(V_{-1}(\fing))}=R_{M(1)
    }\cong \C[z]$    by Theorem \ref{Th:Heisenberg}.
      \end{proof}

  As in Theorem \ref{Th:main1}, 
let $\l_{1}$ be the standard Levi subalgebra of $\sl_{n}$ 
generated by all simple roots except $\alpha_{1}$, that is,
   \begin{align*}
\l_{1}=\h+\bra e_{\alpha_i}, f_{\alpha_i}\mid i\ne 1\ket.
   \end{align*}
 The center of $\l_1$ is generated by $\varpi_1$.
Thus, the Dixmier sheet closure $\overline{\SS_{\l_1}}$ is given by
\begin{align*}
\overline{\mathbb{S}_{\mf{l}_1}}=\overline{G.\C^*\varpi_1},
\end{align*}
see \S\ref{sec:sheet}.
The unique nilpotent orbit contained in $\SS_{\l_1}$  is 
$\on{Ind}^{\fing}_{\mf{l}_1}(0)=\mathbb{O}_{min}$, 
that is, the $G$-orbit of $f_{\theta}$.

 \begin{lemma} \label{l1:A-theta1} 
We have $V(I_{W_1}) \cap \mc{N}\subset \overline{\O_{min}}$. 
\end{lemma}

\begin{proof}
First of all, we observe that $\O_{(2^2,1^{n-4})}$ it the smallest nilpotent orbit of $\g$ 
which dominates $\O_{min}=\O_{(2,1^{n-2})}$. 
By this, it means that 
if ${\bs{\lam}} \succ  (2,1^{n-2})$ then ${\bs{\lam}} \succcurlyeq (2^2,1^{n-4})$ 
where $\succcurlyeq$ is the Chevalley order on the set $\P(n)$. 
Therefore, it is enough to show that $V(I_{W_1})$ does not contain 
 $\overline{\O_{(2^2,1^{n-4})}}$.

Let $f \in \O_{(2^2,1^{n-4})}$ that we embed into an $\sl_2$-triple $(e,h,f)$ 
of $\g$. 
Denote by $\g(h,i)$ the $i$-eigenspace of ${\rm ad}(h)$ for $i\in\Z$ 
and by $\Delta_+(h,i)$ the set of 
positive roots $\alpha \in\Delta_+$ 
such that $e_\alpha \in \g(h,i)$. 
We have $f=\sum_{\alpha \in \Delta_+(h,2)} c_\alpha f_\alpha$ 
with $c_\alpha \in\C$. 
We call the set of $\alpha \in \Delta_+(h,2)$ 
such that $c_\alpha \not=0$ the {\em support} of $f$. 
Choose a Lagrangian subspace 
$\mathfrak{L} \subset \g(h,1)$ and set 
$$\mf{m} := \mathfrak{L} \oplus \bigoplus_{i \geq 2} \g(h,i), \qquad 
J_\chi :=  \sum_{x \in \mf{m}} \C[\g^*](x-\chi(x)),$$  
with $\chi = (f | \cdot ) \in \g^*$, as in \S\ref{section:Ginzburg}.
 By Lemma \ref{lem:criterion},
 it is sufficient to show that
 \begin{align*}
  \C[\fing^*]=I_{W_1}+ J_{\chi}.
 \end{align*}
 To see this, we shall use the vector
$$v_{1}=e_{\theta} e_{\theta_1} - e_{\beta} e_{\gamma}\in I_{W_1}.$$
If $n>4$, 
the weighted Dynkin diagram of the nilpotent orbit $G.f$ is 
$$
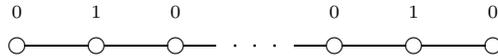
\begin{Dynkin}
	\Dbloc{\Dcirc\Deast\Dtext{t}{0}}
	\Dbloc{\Dcirc\Dwest\Deast\Dtext{t}{1}}
	\Dbloc{\Dcirc\Dwest\Deast\Dtext{t}{0}}
	\Dbloc{\Ddots}
	\Dbloc{\Dcirc\Dwest\Deast\Dtext{t}{0}}
	\Dbloc{\Dcirc\Dwest\Deast\Dtext{t}{1}}
	\Dbloc{\Dcirc\Dwest\Dtext{t}{0}}
\end{Dynkin}
$$ 
So we can assume that $h= \varpi_2 + \varpi_{n-2}$ 
and we can choose for $f$ the element $f_{\beta}+f_{\gamma}$. 
We see that ${\theta}$, ${\theta_1}$, ${\beta}$ 
and ${\gamma}$ all belong to $\Delta_+(h,2)$. 
Since $\beta$ and $\gamma$ are in the support of $f$, 
but not $\theta$ and $\theta_1$, for some nonzero complex number $c$, 
we have 
$$e_{\theta} e_{\theta_1} - e_{\beta} e_{\gamma} = c \pmod{{J_\chi}}.$$ 
So $v_1 = c \pmod{{J_\chi}}$ and $\IS_{W_1} + {J_\chi}=\C[\g^*]$, 
whence $\C[\g^*] / (\IS_{W_1} +{J_\chi}) = 0$. 

For ${n }=4$, the  
weighted Dynkin diagram of the nilpotent orbit $G.f$ is 
$$
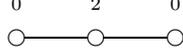
\begin{Dynkin}
	\Dbloc{\Dcirc\Deast\Dtext{t}{0}}
	\Dbloc{\Dcirc\Dwest\Deast\Dtext{t}{2}}
	\Dbloc{\Dcirc\Dwest\Dtext{t}{0}}
\end{Dynkin}
$$ 
and we conclude similarly. 
\end{proof}
\begin{lemma} \label{l2:A-theta1} 
Let $\lam$ be a nonzero semisimple element of $\g$. 
Then, $\lam\in V(I_{W_1})$ if and only if $\lam \in G.\C^*\varpi_1$.  
\end{lemma}

\begin{proof} 
Set for $i,j \in \{1,\ldots,n-2\}$, with $j-i \geq 2$,
$$p_{i,j} := h_i h_j,$$ and for $i \in \{2,\ldots,n-2\}$, 
$$q_i := h_i(h_{i-1}+h_i +h_{i+1}).$$ 
According to \cite{AP08}, the zero weight space of $W_1$ is 
generated by the elements $p_{i,j}$ and $q_i$. 
Clearly, 
$p_{i,j}(\varpi_1)=0$ for $j-i\geq 2$ and $q_i(\varpi_1)=0$ for any $i\in \{2,\ldots,n-2 \}$. 
So $\varpi_1 \in V(I_{W_1})$, and whence 
$G.\C^*\varpi_1 \subset V(I_{W_1})$ since 
$V(I_{W_1})$ is a 
$G$-invariant cone. 
This proves the converse implication. 

For the first implication, 
let $\lam$ be a nonzero semisimple element of $\g$, 
and assume that $\lam \in V(I_{W_1})$. 
Since $V(I_{W_1})$ is $G$-invariant, 
we can assume that $\lam\in\h$. 
Then write 
$$\lam= \sum_{i=1}^{n-1} \lam_i \varpi_i, \qquad \lam_i \in \C.$$
Since $p_{i,j}(\lam)=q_i(\lam)=0$ for all $i,j$, we get 
\begin{align} \label{eq:s} 
&\lam_i \lam_j = 0, \quad   i,j \in \{1,\ldots,n-2 \}, \;  j-i \geq 2, & \\ \label{eq:s2}
&\lam_i(\lam_{i-1}+\lam_i +\lam_{i+1}) =0, \quad  i=2,\ldots,n-2 . & 
\end{align}
Since $\lam$ is nonzero, $\lam_k \not =0$ for some $k \in \{1,\ldots,{n-1 }\}$. 

If $k \not\in \{1,{n-1 }\}$, 
then by (\ref{eq:s}), $\lam_j=0$ for $|j-k|\geq 2$. 
So by (\ref{eq:s2}), either $\lam_{k-1}+\lam_k=0$ or $\lam_k + \lam_{k+1}=0$ 
since $\lam_{k-1}\lam_{k+1}=0$ by (\ref{eq:s}). 

If $k=1$, then by (\ref{eq:s}), $\lam_j=0$ for $j\geq 3$. 
So by (\ref{eq:s2}), $\lam_2(\lam_{1}+\lam_2)=0$. 

If $k={n-1 }$, then by (\ref{eq:s}), $\lam_j=0$ for $j\leq n-3$. 
So by (\ref{eq:s2}), $\lam_{n-2 }(\lam_{n-2 }+\lam_{{n-1}})=0$.

\noindent
 We deduce that
 {\begin{align*}
\lam \in & \; \C^* \varpi_1 \cup \C^* \varpi_{n-1} \cup \bigcup_{1 
\leq j \leq n-2} \C^*(\varpi_i - \varpi_{i+1}) &\\ 
&\qquad  =  
\bigcup_{1 \leq j \leq n-1} \C^*(\eps_1+\cdots + \eps_{i-1} - (n-1) \eps_{i+1}  
+ \eps_{i+2} +\cdots +\eps_{n}).&
  \end{align*}} 
All the above weights are conjugate to $t \varpi_1$ for some $t \in \C^*$ 
under the Weyl group of $(\g,\h)$ 
which is the group of permutations of $\{\eps_1,\ldots,\eps_{n}\}$, 
whence the expected implication.  
\end{proof}

The following assertion follows immediately from
Lemma \ref{Lem:sheet},
Lemma \ref{l1:A-theta1}
and Lemma \ref{l2:A-theta1}.
 \begin{Pro}\label{Pro:{Pro:A-zero-1-pre}}
  We have
  $
  V(I_{W_1})=\overline{\SS_{\l_{1}}}$.
\end{Pro}

\begin{Rem}  \label{rem:Gar}
The zero locus of the Casimir element $\Omega$ in $V(I_{W_1})$ 
 is $\overline{\O_{min}}$.
Indeed, by Lemma \ref{Lem:sheet} the zeros 
locus of $\Omega$ in $\overline{\SS_{\l_{1}}}$ is contained in the nilpotent cone 
since $\Omega (\varpi_1) \not=0$. 
The statement follows 
since $\overline{\O_{min}}$ has codimension one in $\overline{\SS_{\mf{l}_1}}$.  

As a consequence, the zero locus of the ideal generated by $\C\Omega\+ 
W_1\subset S^2(\fing)$ is $\overline{\O_{min}}$. 
This latter fact is known by \cite{Gar} using a different approach. 
\end{Rem}

\begin{Pro} \label{Pro:prime} 
 The ideal $\IS_{W_1}$ is prime,
 and therefore it is the defining ideal of $\overline{\SS_{\l_1}}$. 
\end{Pro}

\begin{proof} 
We apply Lemma \ref{lem:prime-criterion} to the ideal $I:=I_{W_1}$. 
First of all, $\l_1$ and $(e_\theta,h_\theta,f_\theta)$ satisfy the conditions 
of Lemma \ref{lem:Kat} since $\z(\l_1)=\C\varpi_1$ and $\g(h_\theta,i)=0$ 
for $i>2$. It remains to verify that the conditions (1),(2),(3),(4) of Lemma \ref{lem:prime-criterion} 
are satisfied. 

Condition (1) is satisfied by Proposition~\ref{Pro:{Pro:A-zero-1-pre}}. 
According to \cite[Corollary~2 and Theorem~1, Chap.\,V]{Gar}, 
the ideal $I+I_\Omega$ is the defining 
 ideal of $\overline{\O_{min}}$. So condition~(2) is satisfied.
The condition (3) is satisfied too,
 by Corollary \ref{Co:C[z]}.
At last, because $\Omega(\varpi_1)\not=0$, condition (4) is satisfied. 
In conclusion, by Lemma \ref{lem:prime-criterion}, $I=I_{W_1}$ is prime.
\end{proof}

%
Let $\mf{p}=\mf{l}_1\oplus \mf{p}_u$ be a parabolic subalgebra of $\g$ 
with nilradical $\mf{p}_u$. 
It is a maximal parabolic subalgebra of $\fing$.
Let $P$ be the  
connected parabolic subgroup of $G$ with Lie algebra $\mf{p}$. 
Set
\begin{align*}
Y:=G/[P,P].
\end{align*}
As explained in  \S\ref{sec:sheet},
we have
$V(\gr \mc{J}_Y)=\overline{\SS_{\mf{l}_1}}$,
where $\mc{J}_Y=\ker \psi_Y$
and $\psi_Y$ is the algebra homomorphism
$U(\fing)\ra \mc{D}_Y$.

The variety $Y$ is a quasi-affine variety,
isomorphic to $\C^n\backslash \{0\}$.
It follows that the natural embedding
$Y=\C^n\backslash\{0\}\hookrightarrow \C^n$
induces an isomorphism
$\mc{D}_{\C^n}\isomap \mc{D}_Y$.
In this realization,
the map $\psi_Y$ 
is described as follows. 
 \begin{align*}
  \psi_Y:U(\fing)&\ra \mc{D}_{\C^n}=\langle z_1,\dots,
  z_n,\frac{\partial}{\partial z_1}, \dots
  \frac{\partial}{\partial z_n} \rangle,\\
  e_{i,j}&\mapsto -z_j \frac{\partial}{\partial z_i}
  \quad
  (i\ne j),\\
  h_i&\mapsto -z_i \frac{\partial}{\partial z_i}
  +z_{i+1} \frac{\partial}{\partial z_{i+1}}.
 \end{align*}
where $e_{i,j}$ is as before the standard elementary matrix element 
and $h_i=e_{i,i}-e_{i+1,i+1}$.

This has the following chiralization:
Let $\mc{D}^{ch}_{\C^n}$
 be the
$\beta\gamma$-system of rank $n$, 
generated by
fields
$\beta_1(z),\dots ,\beta_n(z)$,
$\gamma_1(z),\dots,\gamma_n(z)$,
satisfying OPE
\begin{align*}
 \gamma_i(z)\beta_j(w)\sim \frac{\delta_{ij}}{z-w},
 \quad \gamma_i(z)\gamma_j(w)\sim 0,\quad \quad \beta_i(z)\beta_j(w)\sim 0,
\end{align*}
see e.g. \cite{Kac98}. 
We have
\begin{align*}
 R_{\mc{D}^{ch}_{\C^n}}\cong \C[T^* \C^n],
 \qquad A(\mc{D}^{ch}_{\C^n})\cong \mc{D}_{\C^n},
\end{align*}
where $\mc{D}_{\C^n}$ denotes the Weyl algebra of rank $n$,
which is identified with the algebra of differential operators on the
affine space $\C^n$.
In particular,
$\mc{D}^{ch}_{\C^n}$ is a chiralization of
$\mc{D}_{\C^n}$.
 \begin{Lem}
  The following gives a vertex algebra homomorphism.
 \begin{align*}
  \hat{\psi}_Y
  : V^{-1}(\fing)&\rightarrow \mc{D}_{\C^n}^{ch},\\
  e_{i,j}(z)&\mapsto -:\beta_j(z)\gamma_i(z): \quad (i\ne j),\\
  h_i(z)&\mapsto -:\beta_i(z)\gamma_i(z):+:\beta_{i+1}(z)\gamma_{i+1}(z):.
 \end{align*}
 \end{Lem}
Note that   the map $  \hat{\psi}_Y$ induces 
an algebra  homomorphism
  $U(\fing)=A(V^{-1}(\fing))\ra
  A(\mc{D}_{\C^n}^{ch})=\mc{D}_{\C^n}$,
  which 
  is identical to $\psi_Y$.

  \begin{Th}[\cite{AdaPer14}]\label{Th:AP}
  The map $\hat{\psi}_Y$ factors through the
  vertex algebra embedding
  \begin{align*}
   V_{-1}(\fing)\hookrightarrow \mc{D}_{\C^n}^{ch}.
  \end{align*}
  \end{Th}
 By Theorem \ref{Th:AP},
 $  \hat{\psi}_Y$  induces a homomorphism
\begin{align}
 A(V_{-1}(\fing))\ra A(\mc{D}_{\C^n}^{ch})=\mc{D}_{\C^n}.
 \label{eq:Zhu-homo}
\end{align}
\begin{Th}\label{Th:Zhu-1}
 \begin{enumerate}
  \item
       The ideal  $\gr \JJ_Y\subset \C[\fing^*]$ 
  is prime and hence 
      it  is the defining ideal
       of $\overline{\mathbb{S}_{\mf{l}_1}}$.
 \item The natural homomorphism
       $R_{V_{-1}(\fing)}\ra \gr A(V_{-1}(\fing))$ is an isomorphism.
        \item  The map $\hat{\psi}_Y$ 
        induces
	an embedding
	\begin{align*}
	  A(V_{-1}(\fing))\hookrightarrow \mc{D}_{\C^n}.
	\end{align*}
 \end{enumerate}\end{Th}
\begin{proof}
 We have
 \begin{align*}
  A(V_{-1}(\fing))=U(\fing)/\JJ_{W_1},
 \end{align*}
 where $\JJ_{W_1}$ is 
 the two-sided ideal of $U(\fing)$ generated by
$W_1$.
  By \eqref{eq:Zhu-homo},
 we have the inclusion
 $\mc{J}_{W_1}\subset \mc{J}_Y=\ker \psi_Y$.
Also,
 ${\rm gr}\JJ_{W_1}\supset \IS_{W_1}$ by 
 definition, and
 $\sqrt{\JJ_Y}=I_{W_1}$
 by Proposition \ref{Pro:prime}.
Thus,
 \begin{align*}
  I_{W_1}\subset \gr \JJ_{W_1}\subset \gr \JJ_Y\subset \sqrt{\gr \JJ_Y}
  =I_{W_1}.
 \end{align*}
It follows that all the above inclusions are equalities:
\begin{align*}
  I_{W_1}= \gr \JJ_{W_1}= \gr \JJ_Y= \sqrt{\gr \JJ_Y}.
\end{align*}
 The statements
 (1), (2) and (3) follow from
 the third, the first and
 the second equality, respectively.
\end{proof}

\begin{Rem}
 Adamovi{\'c} and Per{\v{s}}e \cite{AdaPer14}
 showed that
 $\mc{D}_{\C^n}^{ch}$ decomposes into
 a direct sum of simple highest weight representations
 of $V_{-1}(\fing)$.
 From their results, it is possible to obtain an explicit character formula of
 $V_{-1}(\fing)$ as in \cite{KacWak01}. 
 In view of Theorem \ref{Th:main1},
 this gives the Hilbert series of
 $J_{\infty}\overline{\SS_{\l_1}}$. 
\end{Rem}

\begin{Lem} \label{Lem:hypothesis}
Assume that $V=V_{-1}(\mf{sl}_n(\C))$. 
Then the hypothesis of Theorem \ref{Th:arc} are satisfied. 
\end{Lem}
\begin{proof}
Let $x \in \g^* \cong \g$ be such that $X_V= \overline{\mathbb{S}_{min}}=\overline{G.\C^* x}$.  

We have a natural map $\varphi_x \colon \C^* \to X_V$, $t \mapsto t.x$. 
It gives 
a morphism from $R_V$ to $\mc{O}(\C^*)=\C[t,t^{-1}]$ which induces a morphism 
from $J_{\infty} R_V$ to $\mc{O}(J_{\infty} \C^*)=
\C[t_0,t_0^{-1},t_1,t_2,\ldots]$. 
Note that $J_{\infty} \C^*$ is nothing but $(\pi_{\infty,0}^{\C})^{-1}(\C^*)$, 
where $\pi_{\infty,0}^{\C}$ is the canonical projection from $J_{\infty} \C$ 
to $\C$. 


On the other hand, we have a surjective Poisson vertex algebra morphism 
$$\rho \colon J_{\infty} R_V \twoheadrightarrow \gr V.$$ 
The question is to know whether the morphism 
$J_{\infty} \varphi_x^* \colon J_{\infty} R_V \to \mc{O}(J_{\infty} \C^*)$ 
factorizes through $\rho$. Indeed, if so, then it implies that 
$J_{\infty} \C^*.x$ is contained in $SS(V)_{\on{red}}$. 
We will prove that it is true for some $x \in \g$ such that  
$X_V=\overline{G.\C^* x}$. 


Let $\hat{\mu}$ be the morphism
$$\hat{\mu} \colon \gr V \longrightarrow \C[J_{\infty} T^*\C^n]$$ 
induced 
from the embedding from $V=V_{-1}(\sl_n)$ to $\mc{D}^{ch}_{\C^n}$ 
(see Theorem \ref{Th:AP}). Denote 
by $\mu$ the restriction 
to $R_V \subset \gr V$ of $\hat{\mu}$.
Note that $\mu$ is the morphism 
from $$R_V= \C[\g]/I$$ to $\C[T^* \C^n]$ such that the image 
of $e_{i,j} + I  \in R_V$ is $ - \beta_j \gamma_i$ and  the 
image of $h_i +I \in R_V$ is $-\beta_i \gamma_i +\beta_{i+1}\gamma_{i+1}$, 
with $\beta_1,\ldots,\beta_n,\gamma_1,\ldots,\gamma_n$ 
the natural coordinates on $T^* \C^n \cong \C^{2n}$. 
Let 
$W\subset \mc{O}(T^*\C^n)$ be the image of  $\mu$.
Note that we have
$ J_{\infty}W\subset \mc{O}(J_{\infty}T^*\C^n)$.
It suffices to show that for some $x$ such that $X_V=\overline{G.\C^* x}$, 
there is a well-defined morphism $\nu_x$ from $W$ to $\mc{O}(\C^*)$ 
such that $\nu_x \circ \mu= \varphi_{x}^*$, that is, such that the following diagram 
commutes:
\begin{center}
\hspace{2cm} \xymatrix{ & \mc{O}(\C^*) & \\
R_V \ar[ru]^{\varphi_{x}^*} \ar[rr]^\mu  && W 
\ar@{-->}[lu]_{\nu_x} }
\end{center}
Indeed,
assume the statement proven, and pick $x$ satisfying the above conditions. 
We observe that $\hat{\mu} \circ \rho= J_\infty \mu$ 
(its comes from the identification of $J_{\infty} \g^*$ with $S(t^{-1}\g[t^{-1}])$). 
Hence, we get the commutative diagram:
\begin{center}
\hspace{2cm} \xymatrix{J_{\infty} R_V \ar@{-{>>}}[d]_{\rho} \ar[rr]^{J_{\infty} \varphi_x^*} 
\ar[drr]^{J_{\infty} \mu} &&  \mc{O}(J_{\infty} \C^*) \\ 
\gr V \ar[rr] _{\hat{\mu}} && J_{\infty}{W}
 \ar[u]_{J_{\infty} \nu_x}}
\end{center}
Therefore, for such an $x$, $J_\infty \varphi_x^*$ factorizes through $\gr V$: 
we have $\rho  \circ \hat{\mu} \circ J_{\infty} \nu_x= J_\infty \varphi_x^*$. 

To complete the proof, it thus remains to find such an $x$. 
It is enough to show that for some $x$, $\ker \mu \subset \ker \varphi_{x}$. 
Note that $$\ker \varphi^*_{x}=\{f \in  \C[\g]/I \mid f(tx)=0 \text{ for all } 
t \in \C^*\}= \bigcap\limits_{t \in \C^*}\mf{m}_{tx},$$ where $ \mf{m}_{x}$ 
is the maximal ideal of $R_V$ associated with $x$. 
We have $I \subset \ker \mu$ because  $\mu$ is 
well-defined. 
Since $I \subset \ker \mu$, we have $Z_\mu:={\rm Specm}(\C[\g] / \ker \mu) 
\subset {\rm Specm}(\C[\g] /I)= 
\overline{\mathbb{S}_{min}}$. 
Recall now that $\overline{\mathbb{S}_{min}}= G.\C^* \varpi_1 \cup \overline{\mathbb{O}_{min}}$, 
and 
let us now prove that there is $x \in G.\C^* \varpi_1$ such that 
$\ker \mu \subset \mf{m}_{tx}$ for all $t \in \C^*$. 
Since $Z_\mu$ is $G$-invariant and $\C^*$-invariant, 
it is enough to prove that there is $x \in G.\C^* \varpi_1$ such that 
$\ker \mu \subset \mf{m}_{x}$. 
Assume the contrary. We expect a contradiction. 
We have $Z_\mu  \cap G.\C^* \varpi_1 =\varnothing $ by our assumption, and so 
$Z_\mu \subset \overline{\mathbb{O}_{min}}$. 
So the defining ideal of $\overline{\mathbb{O}_{min}}$ would be contained in 
$\ker \mu$. But this is not true,
since the Casimir element,
$$\Omega = \sum_{1 \leq i\ne j\leq n} e_{i,j}e_{j,i} + \sum_{i=1}^n 
h_i \varpi_i,$$
is not in $\ker \mu$. Indeed, the coefficient of $\beta_2 \gamma_1 
\beta_1\gamma_2$ in $\mu(\Omega)\in \C[\beta_i,\gamma_i\mid 
i=1,\dots,n]$ is 
\begin{align*}
&2 - 2 h_1^*(\varpi_1) + h_2^* (\varpi_1) +h_1^*(\varpi_2) \\
&=\begin{cases} 2 - \dfrac{2(n-1)}{n}+\dfrac{n-2}{n}+\dfrac{n-2}{n} 
= \dfrac{2(n-1)}{n}\not=0 & \text{ if } n >4,\\
 \dfrac{3}{2} \ne 0 & \text{ if } n=4.\end{cases}
\end{align*}
This proves the expected statement, and so completes the proof. 
\end{proof}

  \begin{proof}[Proof of Theorem \ref{Th:main1}~(1)]
   The first statement follows immediately from \eqref{eq:XVtilde}
   and Proposition \ref{Pro:prime}.
   The second statement follows Theorem \ref{Th:arc} and Lemma \ref{Lem:hypothesis}. 
  \end{proof}

\begin{Rem}\label{Rem:fusion}
 Let $V_k(\fing)\on{-Mod}^{\fing[t]}
 $ be the full subcategory of
 the category of $V_k(\fing)$-modules consisting of objects that belong
 to $\on{KL}_k$.
 Since $H^{\frac{\infty}{2}+0}_{f_{\theta}}(V_{-1}(\fing))\cong M(1)$,
 we have a functor
 \begin{align*}
 V_{-1}(\fing)\on{-Mod}^{\fing[t]}
  \ra M(1)\on{-Mod}  ,\quad
  M\mapsto H^{\frac{\infty}{2}+0}_{f_{\theta}}(M),
 \end{align*}
 where $M(1)\on{-Mod}$ denotes the category of
 the modules over the Heisenberg vertex algebra  $M(1)$.
 The proof 
 of Theorem \ref{Th:main1}~(1) and \cite[Theorem 6.2]{AdaPer14}  imply that
 this is a {\em fusion functor} (cf.~Question \ref{ques:fusion}).
\end{Rem}

\section{Level $-m$ affine vertex algebra of type $A_{2m-1}$}  \label{sec:A-theta} 
Set
$$v_0 := \sum_{i=1}^{2m-1 } \frac{{m-i}}{m} h_i e_{\theta}
+ \sum_{i=1}^{2m-2} e_{1,{i+1}}  e_{{i+1},{2m}}
\in S^2(\fing).$$ 
Then $v_0$ is a singular vector and
generates
a finite-dimensional irreducible representation $W_0$
of $\g$ in $S^2(\fing)$
isomorphic to 
$L_\g(\theta)$.
We have
\begin{align*} 
 &\sigma(v_0)\\
 & = \sum_{i=1}^{2m-1 } \frac{m-i}{m} h_i(-1)e_{\theta}(-1) 
+ \sum_{i=1}^{2m-2} e_{1,{i+1}} (-1) e_{{i+1},{2m}}(-1) 
- (m-1) e_{ \theta}(-2), &
\end{align*}
where $\sigma=\sigma_2 $ is defined in \S \ref{sec:A-theta1}, 
equality \eqref{eq:sigma}.

\begin{lemma} \label{lem:sing0} 
The vector $\sigma(v_0)$ 
 is a singular vector of $V_k(\g)$ if and only if  
 $k=-m$. 
\end{lemma}
\begin{proof} 
 The verifications are identical
 to
  \cite[Lemma 4.2]{Per} 
and  we omit the details\footnote{In 
 \cite{Per}, the author only deals with
 $\mf{sl}_{r+1}$ for even ${r }$ in order to have an admissible 
level $k=-\frac{1}{2}({r }+1)$. Note that there is a change of sign 
because our Chevalley basis slightly differs from that of \cite{Per}.}. 
\end{proof}

  \begin{Th}\label{Th:maximal-sub-typeA-2}
The singular vector $\sigma(v_0)$ generates the maximal submodule
 of $V^{-m}(\fing)$, that is,
$V_{-m}(\fing)=V^{-m}(\fing)/U(\affg)\sigma(v_0)$.
  \end{Th}
Let $(e,h,f)$ be the $\mf{sl}_2$-triple of $\fing$ defined by
\begin{align*}
 e=\sum_{i=1}^{m}e_{i,m+i}, 
& \qquad 
 h=\sum_{i=1}^{m}e_{i,i}-\sum_{i=1}^{m}e_{m+i,m+i}
 =2\varpi_{m}, &\\
 & f=\sum_{i=1}^{m}e_{m+i,i}\in\fing.&
\end{align*}
Then
  $f$ is a nilpotent element of $\fing$
  corresponding to the partition $(2^{m})$. 
  We have
  \begin{align*}
   \fing=\fing(h,-2)\+ \fing(h,0)\+\fing(h,2),
  \end{align*}
  where
  $\fing(h,2)=\on{span}_{\C}\{e_{i,m+j}\mid 1\leq i,j\leq m\}$,
  $\fing(h,-2)=\on{span}_{\C}\{e_{m+i,j}\mid 1\leq i,j\leq
  m\}$,
  {$\fing(h,0)=\on{span}_{\C}\{e_{i,j},e_{m+i,m+j}\mid 1\leq
  i,j\leq m\}
  \cap \fing$}. 
  Thus,
    $f$ is a short nilpotent element.

Let $\fing^{\natural}\subset \fing(h,0)$ be the 
 centralizer in $\fing$ of $(e,h,f)$. 
Then
\begin{align*}
\fing^{\natural}\cong \mf{sl}_{m}.
\end{align*}
By \cite[Theorem 2.1]{KacWak04} we have a vertex algebra embedding
\begin{align*}
 V^{k^\natural}(\fing^{\natural})\hookrightarrow \W^k(\fing,f),
\end{align*}
where
\begin{align*}
 k^{\natural}=2k+2m,
\end{align*}
which
restricts to the isomorphism of spaces of weight $1$
\begin{align*}
 \fing^{\natural}\cong  V^{k^\natural}(\fing^{\natural})_1\isomap \W^k(\fing,f)_1.
\end{align*}
Since
\begin{align*}
 \fing(h,-2)\cong \fing^{\natural}\+ \C 
\end{align*}
as $\fing^\natural$-modules,
there exists \cite{KacWak04} a
$\fing^{\natural}$-submodule
$E$ of $\W^k(\fing,f)_2$
such that
$\W^k(\fing,f)$ is generated by
$\W^k(\fing,f)_1\cong \fing^{\natural}$,
$E$ and the conformal vector $\omega_{W}\in \W^k(\fing,f)_2$, 
provided that $k\ne -2m$.

\begin{Th}\label{Th:W-is-Visasoro}
We have
  \begin{align*}
   H^{\frac{\infty}{2}+0}_{f}(\tilde{V}_{-m}(\fing))
   \cong \on{Vir}_1,
  \end{align*}
  where $ \on{Vir}_1$ is the simple Virasoro vertex algebra of central
  charge $1$.
\end{Th}

Let $e_{\theta^{\natural}}=e_{1,m}+e_{m+1,2m}\in \fing^{\natural}$ be 
a root vector of the highest root $\theta^{\natural}$ of $\fing^{\natural}$.
 \begin{Lem}\label{Lem:singular-vector-0}
 The image of $\sigma(v_0)$ in  $\W^{-m}(\fing,f_0)$ coincides with
  the
  image of the singular vector
  ${e_{\theta^{\natural}}(-1)\mathbf{1}}$ of
  $V^{0}(\fing^{\natural})$ up to nonzero constant multiplication.
 \end{Lem}
  \begin{proof}
Let $w$ be the image  of
   $\sigma(v_0)$ in  $\W^{-m}(\fing,f_0)$.
   One finds that
\begin{align*}
 w\equiv {e_{\theta^{\natural}}({-1})\mathbf{1}} 
 \pmod{C_2(\W^{-m}(\fing,f_0))},
\end{align*}
   and hence, it is nonzero and has the same weight as
   ${ e_{\theta^{\natural}}({-1})\mathbf{1}}$.
   Since  it is a singular vector, $w$ is singular vector for
   $V^0(\fing^{\natural})$.
   The assertion follows since
the corresponding weight space is one-dimensional.
  \end{proof}
 \begin{Pro}\label{Pro:non-zero}
  Let $\lam=s\varpi_{m}-m\Lam_0$  with $s\in \Z_{\geq 0}$.
  Then
  $H^{\frac{\infty}{2}+0}_f(L(\lam))\ne 0$.
 \end{Pro}
      \begin{proof}
       Since
      $L(\lam)$ belongs to $\mc{O}_k^{\mf{l}_0}$,
it is sufficient  to show that
      the Gelfand-Kirillov dimension of      
$L(\lam)_{\lam(D_W),h}$
       is maximal
             by Theorem \ref{Th:short}.
  Observe that   $L(\lam)_{\lam(D_W),h}$ is an irreducible highest weight
      representation of $\fing$ with highest weight
       $\mu=-(m+s)\varpi_{m}$.
       It follows from  Jantzen's simplicity criterion
       \cite[Satz 4]{Jan77} (see also \cite[Theorem~9.13]{Hum08})
       that $L_{\fing}(\mu)
       $ is isomorphic
       to the generalized Verma module $U(\fing)\*_{U(\mf{p})}L_{\l_0}(\mu)$,
       where $\mf{p}$ is a parabolic subalgebra of $\fing$ with Levi
       subalgebra
       $\l_0$, and $L_{\l_0}(\mu)$ is the irreducible finite-dimensional
       representation
       of $\l_1$ with highest weight $\mu$.
This completes the proof.
     \end{proof}
   \begin{proof}[Proof of Theorem \ref{Th:W-is-Visasoro}]
    First,
    $H^{\frac{\infty}{2}+0}_{f}(\tilde{V}_{-m}(\fing))\ne 0$ 
since 
    $H^{\frac{\infty}{2}+0}_{f}({V}_{-m}(\fing))\ne 0$ 
by Proposition \ref{Pro:non-zero}.
Second,  the exact sequence
    $0\ra N\ra V^{-m}(\fing)\ra \tilde{V}_{-m}(\fing)\ra 0$,
    where $N=U(\affg)\sigma(v_0)\subset V^{-m}(\fing)$,
    induces 
an exact sequence
    \begin{align*}
     0\ra H^{\frac{\infty}{2}+0}_f(N)
     \ra \W^{-m}(\fing,f)\ra
     H^{\frac{\infty}{2}+0}_f(\tilde{V}_{-m}(\fing))\ra 0.
    \end{align*}
        By Lemma \ref{Lem:singular-vector-0},
 the weight $1$ subspace 
$\W^{-m}(\fing,f)_1(\cong \fing^{\natural})$
  is contained in     $H^{\frac{\infty}{2}+0}_f(N)$.
  Hence,
  ${x_{(0)}\W^{-m}(\fing,f)}\subset H^{\frac{\infty}{2}+0}_f(N)$ 
  for all $x\in  \fing^{\natural}$.
  This gives that $E\subset H^{\frac{\infty}{2}+0}_f(N)$.
It follows that
  $H^{\frac{\infty}{2}+0}(\tilde{V}_{-m}(\fing))
  $ is generated by the single element 
  $\omega_W$.
  Since $\omega_W$ has central charge $1$,
  there is a surjective vertex algebra homomorphism
  $\on{Vir}_1\twoheadrightarrow
  H^{\frac{\infty}{2}+0}(\tilde{V}_{-m}(\fing))$.
The assertion follows as $\on{Vir}_1$ is simple. 
\end{proof}

  \begin{proof}[Proof of Theorem \ref{Th:maximal-sub-typeA-2}]
   Suppose that
   $\tilde{V}_k(\fing)$
   is not simple.
   Then 
   there exists at least one  singular weight
   vector, say $v$,
   which generates a proper submodule of 
   $\tilde{V}_k(\fing)$.
   As  $   A(\tilde{V}_k(\fing))=U(\fing)/\mc{I}_{W_0}$,
   Proposition \ref{pro:weight-0} below
   implies that
   the weight of $v$ has the form 
   $\mu=s\varpi_{m}-m\Lam_0$ with $s\in \Z_{\geq 0}$.
   Consider the submodule $M$ of $\tilde{V}_{k}(\fing)$ generated by
   $v$.
Since 
       $H^{\frac{\infty}{2}+0}_{f}(L(\mu))\ne 0$
   by Proposition \ref{Pro:non-zero},
Theorem \ref{Th:short}(1) implies that
   $H^{\frac{\infty}{2}+0}_{f}(M)$ is a nonzero submodule of
   $H^{\frac{\infty}{2}+0}_{f}(\tilde{V}_{k}(\fing))$.
But by Theorem \ref{Th:W-is-Visasoro},
   $   H^{0}_{f_0}(\tilde{V}_{k}(\fing))\cong \on{Vir}_1$ is
   simple.
   Contradiction.
 \end{proof}

By Theorem \ref{Th:maximal-sub-typeA-2}, 
{we have}
 \begin{align}
R_{\tilde{V}_{-1}(\fing)}=\C[\fing^*]/I_{W_0},\quad \text{and so,} \quad
  X_{\tilde{V}_{-1}(\fing)}
  =\Spec(\C[\fing^*]/I_{W_0}).
  \label{eq:Xtilde-0}
 \end{align}
Thus
 we get the following assertion (see Corollary \ref{Co:C[z]}).
\begin{Co}\label{Co:C[z]2}
We have $H_f(\C[\fing^*]/I_{W_0}) \cong \C[z]$.
\end{Co}

 As in Theorem \ref{Th:main1}, 
define the Levi subalgebra $\l_0$ of $\fing$ by
   \begin{align*}
\l_{0}=\h+\bra e_{\alpha_i}, f_{\alpha_i}\mid i\ne m\ket.
   \end{align*}
    The center of $\l_0$ is spanned  by $\varpi_m$.
Thus,
\begin{align*}
\overline{\mathbb{S}_{\mf{l}_0}}=\overline{G.\C^*\varpi_m},
\end{align*}
see \S \ref{sec:sheet}.
We have $\on{Ind}^{\fing}_{\mf{l}_0}(0)=\mathbb{O}_{(2^m)}$,
and hence,
 $\mathbb{O}_{(2^m)}$ is the unique nilpotent orbit 
 contained in  
the Diximier sheet $\mathbb{S}_{\mf{l}_0}$.

\begin{lemma} \label{lem:nil-0} 
    $V(I_{W_0}) \cap \mc{N}
\subset \overline{\O_{(2^{m})}}$. 
\end{lemma}

Lemma \ref{lem:nil-0} is proven in the same way  
as Lemma \ref{l1:A-theta1}, 
and we omit the proof.   
Note that here  
$\O_{(3,2^{m-2},1)}$ it the smallest nilpotent orbit of $\g$ 
which dominates $\O_{(2^{m})}$.

We  now view $W_0$ as a submodule of $U(\g)$ through the identification 
$S(\g) \cong U(\g)$ given by the symmetrization map, 
and shall determine the characteristic variety $\mc{V}(I_{W_0}) $.


Set for $s \in \{1,\ldots,{2m-1 }\}$,
$$\Lambda_s := \{(i_1,\ldots,i_1) \in {\{1,\ldots,2m-1 \}^s} \; | \; i_1 < \cdots < i_s \text{ and } 
\sum_{k=1}^{s} (-1)^k i_k = (-1)^{s} m\}.$$
For example, $\Lambda_1= \{m\}$ and $\Lambda_2=\{1\leq i_1 < i_2 \leq {2m-1 } \; | 
-i_1+i_2 = m\}$. 
\begin{Pro} \label{pro:weight-0}
The characteristic variety $\mc{V}(I_{W_0}) $ of $I_{W_0}$ is the set 
\begin{align*}
\hat{\Xi} :=
\bigcup_{1 \leq s \leq {2m-1 } } \bigcup_{(i_1,\ldots,i_s) \in \Lambda_s} &
\{t \varpi_{i_1} + 
\sum_{j=2}^s (-1)^j (-t + c_{i_j}) \varpi_{i_j} \; ; \; t\in \C\}.&
\end{align*}
where for $j\in\{2,\ldots,s\}$, 
\begin{align*} 
c_j := i_1+2 \sum_{k=2}^{j-1} (-1)^{k+1}i_k +(-1)^{j+1}i_j.
\end{align*} 
 In particular, the only integral
 dominant weights which lie in $\mc{V}(I_{W_0}) $ are those of the 
form $t  \varpi_{m}$ with $t \in \Z_{\geq 0}$. 
\end{Pro}

\begin{proof}
Set for $i\in \{1,\ldots,{2m-1 }\}$, 
$$\hat{p}_i := h_i \hat{q}_i,$$
where 
$$\hat{q}_i:= \sum_{j=1}^{i-1} \frac{-j}{m} h_j + 
\frac{m-i}{m} h_i + {\sum_{j=i+1}^{2m-1}} \frac{2m-j}{m} h_j 
+m -i.$$
According to \cite[Lemma 5.1]{Per}, $\Upsilon(W_0^\h)$ is 
generated by $\hat{p}_{1},\ldots,\hat{p}_{2m-1 }$. Indeed, this part of Per{\v{s}}e's proof does 
 not use the parity of
 the rank of $\mf{sl}_{n}$
 and all computations hold for $n$ even.

We first verify that any $\lam \in \hat{\Xi}$ is a solution of the 
system of equations $\hat{p}_1(\lam)=0,\ldots,\hat{p}_{{r }}(\lam)=0$. 
The verifications are left to the reader. 

Conversely, let $\lam = \sum_{i=1}^{2m-1 } \lam_i \varpi_i \in\h$ be such that $\hat{p}_i(\lam)=0$ 
for all $i =1,\ldots,{2m-1 }$. Assume that $\lam\not=0$. 
Let us show that $\lam \in \hat{\Xi}$. 
Since $\lam \not= 0$, there exist integers $i_1,i_2,\ldots,i_s$ in $\{1,\ldots,{2m-1 }\}$, with 
$i_1< i_2\cdots < i_s$, such that 
$\lam_{i_j} \not=0$ if $j \in \{1,\ldots,s\}$ and $\lam_k =0$ for all 
$k \not\in \{i_1,\ldots,i_1\}$.  
Thus, $\hat{q}_{i_j}(\lam)=0$ for all $j\in\{1,\ldots,s\}$.

Assume $s=1$. Then 
$$0=\hat{q}_{i_1}(\lam)= 
\frac{m-i_1 }{m} \lam_{i_1}  +m -i_1.$$
Either $i_1= m$ and then $\lam = \lam_m \varpi_{m} \in \hat{\Xi}$,  
or $i_1\not=m$ and then 
$$ \lam_{i_1}  = - m\quad \text{ so }\quad  \lam= - m \varpi_{i_1}.$$ 
Since $i_1 \not= m$, either $i_1 > m$ or $i_1 < m$. 
By symmetry, we can assume that $i_1< m$. 
Indeed, if $i_1 > m$, then $2m - i_1 < m$, 
{but if $\hat{p}_i(\lam')=0$ } 
for all $i =1,\ldots,{2m-1 }$, then $\lam' := \lam_{2m-i_1} \varpi_{2m-i_1}$ 
verifies $\hat{p}_i(\lam')=0$ 
 for all $i =1,\ldots,{2m-1}$, too.
 Now
one can choose $i_2 \in \{i_1+1,\ldots,{2m-1 }\}$ such 
that $-i_1+i_2=m$. 
Then 
\begin{align*}
\lam =  \lam_{i_1} \varpi_{i_1}  = & \lam_{i_1} \varpi_{i_1}  + (- \lam_{i_1} +i_1-i_2) \varpi_{i_2} & \\ 
& \in  \bigcup_{(i_1,i_2) \in \Lambda_2} 
\{ t \varpi_{i_1} + (-t +  i_1- i_2) \varpi_{i_2} \; ; \; t\in \C\} \subset \hat{\Xi}&
\end{align*}
since $\lam_{i_1}  = - m = i_1-i_2$ and $c_{i_2} = i_1-i_2$, and we are done. 

Assume now $s\geq 2$. 
Since $\lam_{i_j}\not=0$ for $j=1,\ldots,s$, we get $\hat{q}_{i_j}(\lam)=0$ 
for $j=1,\ldots,s$. Using the equations $\hat{q}_{i_1}(\lam)-\hat{q}_{i_2}(\lam)=0$, \ldots, 
$\hat{q}_{i_s}(\lam)-\hat{q}_{i_s}(\lam)=0$, we get 
\begin{align*} 
\forall \, j \in\{1,\ldots,s-1\},\qquad 
\lam_{i_{j+1}} + \lam_{i_j} - i_j + i_{j+1}=0.
\end{align*}
By induction on $j$ we obtain that 
\begin{align} \label{eq:hij} 
\forall \, j \in\{2,\ldots,s\},\qquad  \lam_{i_j} = (-1)^j (- \lam_{i_1} + c_j) 
\end{align}
where $c_j$ is defined as in the proposition: 
\begin{align*} 
c_j := i_1+2 \sum_{k=2}^{j-1} (-1)^{k+1}i_k +(-1)^{j+1}i_j, \quad j=2,\ldots,s.
\end{align*} 
Using the equations $\hat{q}_{i_s}(\lam)=0$ and (\ref{eq:hij}), we get 
\begin{align} \label{eq:lam}
0 = & \left(  \frac{1}{m} \sum_{k=1}^s (-1)^k i_k -(-1)^s \right) \lam_{i_1} & \\ \nonumber
&\hspace{1cm} - \frac{1}{m}\sum_{k=2}^s (-1)^k i_k c_k + \frac{m- i_s}{m}(-1)^s c_s 
+m- i_s.&
\end{align}
Either $$\frac{1}{m} \sum_{k=1}^s (-1)^k i_k -(-1)^s =0, 
\quad \text{ that is, } \quad 
\sum_{k=1}^s (-1)^k i_k=(-1)^s m$$
and then $\lam \in \hat{\Xi}$, 
or $$ \sum_{k=1}^s (-1)^k i_k \not =(-1)^s m$$ and then 
$\lam_{i_1}$ is entirely determined by the equation (\ref{eq:lam}), and so 
is $\lam_{i_2},\ldots,\lam_{i_s}$ by~(\ref{eq:hij}). 

\begin{Claim} \label{claim:j} 
Set $i_0:=0$. There exist $l \in \{1,\ldots,s\}$ and 
$j \in \{1,\ldots,2m-1\} \setminus\{i_1,\ldots,i_s\}$ such that 
$i_{l-1} < j< i_{l}$ and 
$$ \sum_{k=1}^{s+1} (-1)^k i'_k =(-1)^{s+1} m$$ 
where $i'_k = i_k$ for $k=1,\ldots,l-1$, $i'_{l} =j$ and 
$i'_{k+1} = i_{k}$ for $k=l,\ldots,s$. 
\end{Claim}

\begin{proof} 
First of all, we observe that the sequence $(1,2,\ldots,2m-1)$ always belongs to 
$\Lambda_{2m-1}$. Hence, if $ \sum_{k=1}^s (-1)^k i_k \not =(-1)^s m$, 
then there exist $l \in \{1,\ldots,s\}$ and 
$j \in \{1,\ldots,2m-1\} \setminus\{i_1,\ldots,i_s\}$ such that 
$i_{l-1} < j< i_{l}$. 
We prove the statement by induction on $s$. 
The claim is known for $s=1$. 
Let us prove it prove for $s=2$. There are two cases:
\begin{itemize}
\item[a)] First case: $-i_1 +i_2 < m$. Then set 
$$j := m-i_1+i_2.$$ 
Possibly replacing $i_1,i_2$ by $2m-i_2,2m-i_2$, we can assume 
that $i_1 < m$. Hence $j > i_2$. 
In addition, $j < 2m$ since $-i_1 +i_2 < m$. So $j$ 
suits the conditions of the claim. 
\item[b)] Second case: $-i_1 +i_2 > m$. 
Then $i_2  > m$, $i_1 <   m$ and we set: 
$$j:= -m +i_1+i_2.$$ 
We have $i_1 < j < i_2$ and $j$ suits the conditions of the claim.  
\end{itemize}

Assume now that $s\geq 2$ and that the claim is true for all strictly 
smaller integers. 
We assume that $s$ is even. 
The case where $s$ is odd is dealt similarly. 
There are two cases: 
\begin{itemize}
\item[a)] First case: $\sum_{k=1}^s (-1)^k i_k < m$.  
Either $$m + \sum_{k=1}^{s-1} (-1)^k i_k >0,$$ 
and  then the integer 
$$j:= m + \sum_{k=1}^{s} (-1)^k i_k > i_s$$ 
suits the conditions of the claims. 
Or  $m + \sum_{k=1}^{s-1} (-1)^k i_k \leq 0$, that is 
$ i_{s-1} \geq \frac{r+1}{1} + \sum_{k=1}^{s-2} (-1)^k i_k$. 
Since $i_{s-1} < 2m -\delta$, with $\delta:=i_s-i_{s-1}$, we get 
$$\sum_{k=1}^{s-2} (-1)^k i_k < m-\delta.$$ 
Apply the induction hypothesis to the sequence $i_1,\ldots,i_{s-1}$ 
and $2m-2\delta-1$, which is an odd integer. Then there exists $j \in \{1,\ldots,2m-2\delta-1\} 
\setminus \{i_1,\ldots,i_{s-2}\}$ such that 
$\sum_{k=1}^{s-1}(-1)^k i'_{k} = -(m-\delta)$  where the sequence 
$i'_1,\ldots,i'_{s-1}$ is defined as in the claim with respect to $i_1,\ldots,i_{s-2}$ 
and $j$. We easily verify that $j < i_{s_1}$, and 
we have 
$$\sum_{k=1}^{s-1}(-1)^k i'_{k} + i_{s-1} - i_s =  -(m-\delta)- \delta = 
-m.$$ 
So $j$ suits the conditions of the claim. 
\item[b)] Second case: $\sum_{k=1}^s (-1)^k i_k > m$.
Then $\sum_{k=1}^{s-2} (-1)^k i_k -i_{s-1}+i_s > m$, that is, 
$ \sum_{k=1}^{s-2} (-1)^k i_k  > m-\delta $ with $\delta:=i_s-i_{s-1}$. 
Applying the induction hypothesis to the sequence $i_1,\ldots,i_{s-1}$ 
and $2m-2\delta-1$, we conclude as in case a). 
\end{itemize}

We illustrate in Figures \ref{fig:j} and \ref{fig2:j} the construction of $j$.  
In Figure \ref{fig:j}, the positive integer $- \sum_{k=1}^s (-1)^k i_k$ corresponds 
to the sum of the lengths of the thick lines while 
the positive integer ${2m + \sum_{k=1}^s (-1)^k i_k} $ 
corresponds to the sum of the lengths of the thin lines. 
In Figure \ref{fig2:j}, the positive integer 
$\sum_{k=1}^s (-1)^k i_k$ corresponds 
to the sum of the lengths of the thick lines while 
the positive integer $2m - \sum_{k=1}^s (-1)^k i_k $ 
corresponds to the sum of the lengths of the thin lines. 

\begin{figure}
{\setlength{\unitlength}{0.25in}
\begin{center}
\begin{picture}(1,1)(0,0)
\linethickness{0.125mm}
\put(-3.95,0){\line(1,0){8}}
\put(-4,0){$|$}
\put(-3,0){$|$}
\put(-2,0){$|$}
\put(-1,0){$|$}
\put(0,0){$|$}
\put(1,0){$|$}
\put(2,0){$|$}
\put(3,0){$|$}
\put(4,0){$|$}

\linethickness{0.5mm}
\put(-3.95,0){\line(1,0){1}}
\put(-0.95,0){\line(1,0){1}}
\put(2.05,0){\line(1,0){1}}

\put(-3.05,0.7){\tiny{$i_1$}}
\put(-1.05,0.7){\tiny{$i_2$}}
\put(-0.05,0.7){\tiny{$i_3$}}
\put(1.95,0.7){\tiny{$i_4$}}
\put(2.95,0.7){\tiny{$i_5$}}

\put(-4,-0.5){\tiny{$0$}}
\put(-3,-0.5){\tiny{$1$}}
\put(-2,-0.5){\tiny{$2$}}
\put(-1,-0.5){\tiny{$3$}}
\put(0,-0.5){\tiny{$4$}}
\put(1,-0.5){\tiny{$5$}}
\put(2,-0.5){\tiny{$6$}}
\put(3,-0.5){\tiny{$7$}}
\put(4,-0.5){\tiny{$8$}}
\end{picture}
\end{center}}
{\setlength{\unitlength}{0.25in}
\begin{center}
\begin{picture}(1,2)(0,0)
\linethickness{0.125mm}
\put(-3.95,0){\line(1,0){8}}
\put(-4,0){$|$}
\put(-3,0){$|$}
\put(-2,0){$|$}
\put(-1,0){$|$}
\put(0,0){$|$}
\put(1,0){$|$}
\put(2,0){$|$}
\put(3,0){$|$}
\put(4,0){$|$}

\linethickness{0.5mm}
\put(-3.95,0){\line(1,0){1}}
\put(-0.95,0){\line(1,0){1}}
\put(1.05,0){\line(1,0){1}}
\put(3.05,0){\line(1,0){1}}

\put(-3.05,0.7){\tiny{$i_1$}}
\put(-1.05,0.7){\tiny{$i_2$}}
\put(-0.05,0.7){\tiny{$i_3$}}
\put(0.95,0.7){\tiny{$\mathbf{j}$}}
\put(1.95,0.7){\tiny{$i_4$}}
\put(2.95,0.7){\tiny{$i_5$}}

\put(-4,-0.5){\tiny{$0$}}
\put(-3,-0.5){\tiny{$1$}}
\put(-2,-0.5){\tiny{$2$}}
\put(-1,-0.5){\tiny{$3$}}
\put(0,-0.5){\tiny{$4$}}
\put(1,-0.5){\tiny{$5$}}
\put(2,-0.5){\tiny{$6$}}
\put(3,-0.5){\tiny{$7$}}
\put(4,-0.5){\tiny{$8$}}
\end{picture}
\end{center}}
\caption{{\footnotesize Construction of $j$ for $m=4$, $s=5$ and $(i_1,i_2,i_3,i_4,i_5)=
(1,3,4,6,7)$}} \label{fig:j}
\end{figure}
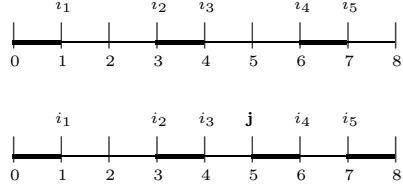
\begin{figure}
{\setlength{\unitlength}{0.25in}
\begin{center}
\begin{picture}(1,1)(0,0)
\linethickness{0.125mm}
\put(-3.95,0){\line(1,0){8}}
\put(-4,0){$|$}
\put(-3,0){$|$}
\put(-2,0){$|$}
\put(-1,0){$|$}
\put(0,0){$|$}
\put(1,0){$|$}
\put(2,0){$|$}
\put(3,0){$|$}
\put(4,0){$|$}

\linethickness{0.5mm}
\put(-2.95,0){\line(1,0){2}}
\put(1.05,0){\line(1,0){1}}

\put(-3.05,0.7){\tiny{$i_1$}}
\put(-1.05,0.7){\tiny{$i_2$}}
\put(0.95,0.7){\tiny{$i_3$}}
\put(1.95,0.7){\tiny{$i_4$}}

\put(-4,-0.5){\tiny{$0$}}
\put(-3,-0.5){\tiny{$1$}}
\put(-2,-0.5){\tiny{$2$}}
\put(-1,-0.5){\tiny{$3$}}
\put(0,-0.5){\tiny{$4$}}
\put(1,-0.5){\tiny{$5$}}
\put(2,-0.5){\tiny{$6$}}
\put(3,-0.5){\tiny{$7$}}
\put(4,-0.5){\tiny{$8$}}
\end{picture}
\end{center}}
{\setlength{\unitlength}{0.25in}
\begin{center}
\begin{picture}(1,2)(0,0)
\linethickness{0.125mm}
\put(-3.95,0){\line(1,0){8}}
\put(-4,0){$|$}
\put(-3,0){$|$}
\put(-2,0){$|$}
\put(-1,0){$|$}
\put(0,0){$|$}
\put(1,0){$|$}
\put(2,0){$|$}
\put(3,0){$|$}
\put(4,0){$|$}

\linethickness{0.5mm}
\put(-2.95,0){\line(1,0){2}}
\put(1.05,0){\line(1,0){1}}
\put(3.05,0){\line(1,0){1}}

\put(-3.05,0.7){\tiny{$i_1$}}
\put(-1.05,0.7){\tiny{$i_2$}}
\put(0.95,0.7){\tiny{$i_3$}}
\put(1.95,0.7){\tiny{$i_4$}}
\put(2.95,0.7){\tiny{$\mathbf{j}$}}

\put(-4,-0.5){\tiny{$0$}}
\put(-3,-0.5){\tiny{$1$}}
\put(-2,-0.5){\tiny{$2$}}
\put(-1,-0.5){\tiny{$3$}}
\put(0,-0.5){\tiny{$4$}}
\put(1,-0.5){\tiny{$5$}}
\put(2,-0.5){\tiny{$6$}}
\put(3,-0.5){\tiny{$7$}}
\put(4,-0.5){\tiny{$8$}}
\end{picture}
\end{center}}
\caption{{\footnotesize Construction of $j$ for $m=4$, $s=4$ and $(i_1,i_2,i_3,i_4)=
(1,3,5,6)$}} \label{fig2:j}
\end{figure}
\end{proof}
 Let $j$ and $i'_1,\ldots,i'_{s+1}$ be as in Claim \ref{claim:j}. 
Possibly replacing the sequence $i_1,\ldots,i_s$ by 
$2m-i_s-1,\ldots,2m-i_s-1$ we can assume that $l \not=1$, that is, 
$j \not= i_1$. 
Set 
$$\lam_{i'_l} : = 0$$
so that 
$$\lam = \sum_{k=1}^s \lam_{i_k} \varpi_{i_k}  = 
\sum_{k=1}^{s+1} \lam_{i'_k} \varpi_{i'_k}.$$
where for $j=1,\ldots,l-1$, $\lam_{i'_k} := \lam_{i_k}$ 
and for $k=l,\ldots,s$, $\lam_{i'_{k+1}} := \lam_{i_{k}}$.
Then one can verify that for all $k \in\{1,\ldots,s+1\}$, 
$$(-1)^k(- \lam_{i'_1} +c_{i'_k}) = \lam_{i'_k}.$$ 
The verifications are left to the reader. 
This proves that $\lam \in \hat{\Xi}$ since 
$(i'_1,\ldots,i'_{s+1}) \in \Lambda_{s+1}$. 
\end{proof}
\begin{lemma} \label{lem:A-zero0} 
Let $\lam$ 
be a nonzero semisimple element of $\g$ 
which lies in $V(I_{W_0})$. 
Then $\lam\in G.\C^* \varpi_{m}$.  
\end{lemma}

\begin{proof} 
Set $S(\g)^\h := \{x \in S(\g) \; | \; [h,x]=0 \text{ for all } h \in \h \}$ and 
let 
$$\Psi \colon S(\g)^\h \to S(\h)$$
be the {\em Chevalley projection map} which 
is the restriction of the projection map $S(\g)=S(\h) \oplus (\n_- +\n_+)S(\g) \to S(\h)$ 
to $S(\g)^\h$. It is known that $\Psi$ is an algebra homomorphism. 
We have 
$$V(I_{W_0}) \cap \h^* = \{\lambda \in \h \; | \; p(\lam)=0 \text{ for all } p\in 
\Psi(I_{W_0} \cap S(\g)^\h)\}\subset \h^*\cong \h.$$ 
Since $V(I_{W_0})$ is $G$-invariant, it is enough to prove the lemma 
for nonzero elements $\lam\in V(I_{W_0})\cap \h^*$. 

It follows from the proof of Proposition \ref{pro:weight-0} that 
$V(I_{W_0}) \cap \h^*$ is the zero 
locus in $\h$ of $p_{1},\ldots,p_{2m-1 }$ where 
for $i\in \{1,\ldots,2m-1\}$, 
$$p_i:= 
h_i\left(\sum_{j=1}^{i-1} \frac{-j}{m} h_j + 
\frac{m-i}{m} h_i + \sum_{j=i+1}^{2m-1} \frac{2m-j}{m} h_j\right).$$
Then it also follows from the proof of Proposition \ref{pro:weight-0} that the zero 
locus in $\h$ of $p_{1},\ldots,p_{2m-1 }$ is the set 
\begin{align*}
\Xi := \bigcup_{1 \leq s \leq {2m-1 } } \bigcup_{(i_1,\ldots,i_s) \in \Lambda_s} 
\C (\sum_{j=1}^s (-1)^j \varpi_{i_j} ).
\end{align*}

Let $\lam \in \Xi$. We can assume that $\lam =\sum_{j=1}^s (-1)^j \varpi_{i_j} $ 
for $s \in \{1,\ldots,2m-1\}$ and $(i_1,\ldots,i_s) \in \Lambda_s$. 
We have to show that $\lam$ is conjugate to $\varpi_{m}$ 
under the Weyl group 
 $W(\g,\h)$ of $(\g,\h)$ which is the group of permutations of
 $\{\eps_1,\ldots,\eps_{2m}\}$. 
Observe that 
$$\varpi_{m}= \frac{1}{2}(\sum_{i=1}^{m} {\eps_{i}}
- \sum_{i=m+1}^{2m} {\eps_{i}}).$$ 
So it suffices to show that $\lam$ can be written as 
$$\lam = \sum_{i=1}^{2m} \sigma_i \eps_i$$ 
with $\sigma_i \in \{-\frac{1}{2},\frac{1}{2}\}$ for all $i \in\{1,\ldots,2m\}$ 
and $${\rm card}(\{i \; |\; \sigma_i =\frac{1}{2} \})= {\rm card}(\{i \; |\; \sigma_i =-\frac{1}{2} \})=m.$$ 

If $s$ is even, we have 
\begin{align*}
\lam =\sum_{j=1}^s (-1)^j \varpi_{i_j} =& \sum_{j=1}^{s/2} 
(\eps_{i_{2j-1}+1}+\cdots + \eps_{i_{2j}}) - \frac{\sum_{j=1}^s (-1)^{j} i_j}{2m} 
(\eps_1+\cdots +\eps_{2m}) &\\
=& \sum_{j=1}^{s/2} 
(\eps_{i_{2j-1}+1}+\cdots + \eps_{i_{2j}}) - \frac{1}{2}  
(\eps_1+\cdots +\eps_{2m}) &\\
=& \frac{1}{2}( -(\eps_{1} +\cdots +\eps_{i_1})+ \sum_{j=1}^{(s-2)/2} 
(\eps_{i_{2j-1}+1}+\cdots + \eps_{i_{2j}}) - (\eps_{i_{2j}+1} +\cdots + \eps_{i_{2j+1}}) &\\
& \quad + (\eps_{i_{s-1}}+\cdots +\eps_{i_s}) - (\eps_{i_{s}+1}+\cdots+\eps_{2m})). &
\end{align*}
since $(i_1,\ldots,i_s) \in \Lambda_s$. 
We are done because $\sum_{j=1}^s (-1)^j i_j = m$. 

If $s$ is odd, we have 
\begin{align*}
\lam =\sum_{j=1}^s (-1)^j \varpi_{i_j} =& \sum_{j=1}^{(s-1)/2} 
(\eps_{i_{2j-1}+1}+\cdots + \eps_{i_{2j}}) - (\eps_1+\cdots+\eps_{i_s})
- \frac{\sum_{j=1}^s (-1)^{j} i_j}{2m} 
(\eps_1+\cdots +\eps_{2m}) &\\
=& - (\eps_1+\cdots+\eps_{i_1}) - \sum_{j=1}^{(s-1)/2} 
(\eps_{i_{2j}+1}+\cdots + \eps_{i_{2j+1}}) + \frac{1}{2}  
(\eps_1+\cdots +\eps_{2m}) &\\
=& \frac{1}{2}( -(\eps_{1} +\cdots +\eps_{i_1})+ \sum_{j=1}^{(s-1)/2} 
(\eps_{i_{2j-1}+1}+\cdots + \eps_{i_{2j}}) - (\eps_{i_{2j}+1} +\cdots + \eps_{i_{2j+1}}) &\\
& \quad + (\eps_{i_{s}+1}+\cdots+\eps_{2m})). &
\end{align*}
since $(i_1,\ldots,i_s) \in \Lambda_s$. 
We are done because $\sum_{j=1}^s (-1)^j i_j = -m$. 
\end{proof}
The following assertion follows immediately from
Lemma \ref{Lem:sheet},
Lemma \ref{lem:nil-0},
and Lemma \ref{lem:A-zero0}.
\begin{Pro} \label{Pro:A-zero0}
We have $V(I_{W_0})=\overline{\SS_{\l_{0}}}$.
\end{Pro}

\begin{Pro} \label{Pro:prime0} 
 The ideal $\IS_{W_0}$ is prime,
 and therefore, it is 
 the defining ideal of $\overline{\SS_{\l_0}}$.
\end{Pro}
\begin{proof}
We apply Lemma \ref{lem:prime-criterion} to the ideal $I:=I_{W_0}$. 
First of all, $\l_0$ and $(e,h,f)$ satisfy the conditions 
of Lemma \ref{lem:Kat} for $f \in \O_{(2^{m})}$. 
Indeed, $\z(\l_0)=\C\varpi_{m}$ and $\g(h,i)=0$ 
for $i>2$. It remains to verify that the conditions (1), (2), (3), (4) of Lemma \ref{lem:prime-criterion} 
are satisfied. 

Condition (1) is satisfied by Proposition~\ref{Pro:A-zero0}~(2). 
Let us show that the ideal $I+I_\Omega$ is the defining ideal of $\O_{(2^{m})}$. 
According to Proposition \ref{Pro:A-zero0}, the zero locus of $I+I_\Omega$ in $\g$ 
is $\overline{\O_{(2^{m})}}$ since $\Omega(\varpi_{m})\not=0$. 
On the other hand, by \cite[Theorem~1]{We2}, 
the defining ideal $I_0$ of $\overline{\O_{\bs{\lam}_0}}$ is generated by 
the entries of the matrix $X^{2}$ as functions of $X\in\sl_{2m}(\C)$.
In particular, $I_0$ is 
generated by homogeneous elements of degree 2. 
Assume that $I+I_\Omega$ is strictly contained in $I_0$. 
A contradiction is expected. 
Since $I_0 \supsetneq I+I_\Omega$, it results from the decomposition 
\begin{align*}
 S^2(\fing)=L_{\fing}(2\theta)\+ W_0\+ L_{\fing}(0) \+W_1 
\end{align*}
that either $I_0$ contains a nonzero element of $L_{\fing}(2\theta)$,  
or $I_0$ contains an element of $W_1$. 
Since $I_0$ is $\g$-invariant, either $I_0$ contains $L_{\fing}(2\theta)$ 
or $I_0$ contains $W_1$. 
The zero locus in $\g$ of the ideal generated by $L_{\fing}(2\theta)$ 
is $\{0\}$ since $L_{\fing}(2\theta)$ is generated as a $\g$-module by $(e_\theta)^2$. 
In addition, by Remark \ref{rem:Gar}, the zero locus in $\g$ of the ideal 
generated by $W_1$ and $\Omega$ is $\overline{\O_{min}}$. 
Hence in both cases we go to a contradiction since $\overline{\O_{(2^{m})}}$  
strictly contains $\overline{\O_{min}}$ and $\{0\}$. So $I+I_\Omega=I_0$ is prime. 
Finally, condition~(2) is satisfied.

Condition (3) is satisfied too, by Corollary \ref{Co:C[z]2}.

At last, because $\Omega(\varpi_{m})\not=0$, condition (4) is satisfied. 
In conclusion, by Lemma~\ref{lem:prime-criterion}, $I=I_{W_0}$ is prime.
\end{proof}

   \begin{proof}[Proof of Theorem \ref{Th:main1}~(2)]
    The 
    statement follows from
    \eqref{eq:Xtilde-0}
    and Proposition~\ref{Pro:prime0}.
   \end{proof}
\begin{Rem}
 If $n$ is odd, then
 the level $-n/2$ is admissible for $\widehat{\mf{sl}}_n$,
 and we have
 \begin{align*}
X_{V_{-n/2}(\mf{sl}_n)}=\overline{\mathbb{O}_{(2^{(n-1)/2},1)}}
 \end{align*}
 by \cite{Ara09b}.
\end{Rem}

%
%

Let $\mf{p}=\mf{l}_0 \oplus \mf{p}_u$ be a parabolic 
subalgebra 
of $\g$, and
let $P$ be the  
connected parabolic subgroup of $G$ with Lie algebra $\mf{p}$.
Let $\psi_Y$ be  the algebra homomorphism
$U(\fing)\ra \mc{D}_Y$ as in \S \ref{sec:sheet},
where 
$Y:=G/[P,P]$.
Recall that
$V(\gr \mc{J}_Y)=\overline{\SS_{\mf{l}_0}}$, 
where $\mc{J}_Y=\ker \psi_Y$.
 \begin{Th}\label{Th:Zhu-2}
  \begin{enumerate}
     \item
       The ideal  $\gr \JJ_Y\subset \C[\fing^*]$ 
  is prime and hence 
       it is the defining ideal
       of $\overline{\mathbb{S}_{\mf{l}_0}}$.
 \item The natural homomorphism
       $R_{V_{-m}(\fing)}\ra \gr A(V_{-m}(\fing))$ is an isomorphism.
        \item  The map $\psi_Y$ induces
	an embedding
	\begin{align*}
	  A(V_{-m}(\fing))\hookrightarrow \mc{D}_{Y}.
	\end{align*}
  \end{enumerate}
 \end{Th}
  \begin{proof}
    We have
 \begin{align*}
  A(V_{-m}(\fing))=U(\fing)/\JJ_{W_0},
 \end{align*}
 where $\JJ_{W_0}$ is 
 the two-sided ideal of $U(\fing)$ generated by
   $W_0$.

  Recall that
  \begin{align*}
\JJ_Y = \bigcap_{t \in Z} {\rm Ann}\,U(\g) \otimes_{U(\mf{p})} \C_{t \varpi_{m}}
\end{align*}
 for
  any Zariski dense subset $Z$ of $\C$,
  see \S \ref{sec:sheet}, equation \eqref{eq:enough-for-zariski-dense}.
  On the other hand, for a generic point $t$ of $\C$ we have
  {$L_{\fing}(t\varpi_{m})\cong U(\g) \otimes_{U(\mf{p})} \C_{t
  \varpi_{m}}$}, 
   and thus,
   $\mc{J}_{W_0}\subset {\on{Ann}L_{\fing}(t\varpi_{m})}$ by
   Proposition \ref{Pro:ch-variety-vs-Zhu} and
  Proposition \ref{pro:weight-0}.
  Therefore $  \mc{J}_{W_0}\subset \mc{J}_Y$.
  This gives 
\begin{align*}
 I_{W_0}\subset  \gr \mc{J}_{W_0}\subset \gr \mc{J}_Y
 \subset \sqrt{\gr \mc{J}_Y}.
\end{align*}
  Since
  $\sqrt{\gr \mc{J}_Y}=I_{W_0}$,
  all above inclusions  are equality.
  \end{proof}

   Since $H^{\frac{\infty}{2}+0}_f(V_{-m}(\fing))\cong \on{Vir}_1$
   by Theorem \ref{Th:W-is-Visasoro},
 we have a functor
 \begin{align}
V_{-m}(\fing)\on{-Mod}^{\fing[t]}
\ra \on{Vir}_1\on{-Mod}  ,\quad
  M\mapsto H^{\frac{\infty}{2}+0}_f(M),
  \label{eq:fusion?}
 \end{align}
 where $\on{Vir}_1\on{-Mod}$ denotes the category of
 $\on{Vir}_1$-modules.
By Proposition \ref{pro:weight-0},
the simple objects of 
$V_{-m}(\fing)\on{-Mod}^{\fing[t]}$ are $L(t\varpi_m-m\Lam_0)
$, $t\in
 \Z_{\geq 0}$.
From Theorem \ref{Th:short} it follows that
\begin{align*}
H^{\frac{\infty}{2}+0}_f(L(t\varpi_m-m\Lam_0))\cong
 L(1,\frac{t(t+m+1)}{4}),
\end{align*}where $L(c,h)$ denotes the irreducible highest weight 
representation of the Virasoro algebra of central charge $c$ and lowest 
weight $h$.
  \begin{ques} \label{ques:fusion}
Is  the functor \eqref{eq:fusion?} fusion,
that is,
$H^{\frac{\infty}{2}+0}_f(M_1\boxtimes M_2)\cong
H^{\frac{\infty}{2}+0}_f(M_1)\boxtimes H^{\frac{\infty}{2}+0}_f(M_2)$?
Here $\boxtimes$ denotes the fusion product.
  \end{ques}

\color{black}
 \section{Level $-(r-2)$ affine vertex algebra of type $D_r$, $r\geq 5$
 } \label{sec:BD}
We assume in this section that $\g=\mf{so}_{2r }$ with $r  \geq 5$. 
Let $$\Delta=\{\pm \varepsilon_i\pm  \varepsilon_j, 
\; |\; 1\leq i< j\leq {r }\}$$ be the root system of $\g$ 
and take $$\Delta_+=\{\varepsilon_i\pm  \varepsilon_j, 
\; |\; 1\leq i < j\leq {r }  \}$$ for the set of positive roots. 

Denote by $(e_i,h_i,f_i)$ the Chevalley generators of $\g$, and 
fix the root vectors $e_\alpha,f_\alpha$, 
$\alpha \in \Delta_+$ so that $(h_i,\,i=1,\ldots,{r }) \cup 
(e_\alpha,f_\alpha, \, \alpha \in \Delta_+)$ is a Chevalley 
basis satisfying the conditions of \cite[Chapter IV, Definition 6]{Gar}. 
For $\alpha \in \Delta_+$, denote by $h_\alpha =[e_\alpha,f_{\alpha}]$ 
the corresponding coroot. 
Let $\g= \n_- \oplus \h \oplus \n_+$ be the corresponding triangular decomposition. 

The fundamental weights are: 
$$
\varpi_i = \varepsilon_1+\cdots +\varepsilon_i \quad (1 \leq i \leq {r }-2), 
$$
$$
{\varpi_{r-1 } = \frac{1}{2}( \varepsilon_1+\cdots +\varepsilon_{{r }-2}+\varepsilon_{r-1 }
-\varepsilon_{{r }}), \qquad 
\varpi_{{r }} = \frac{1}{2}( \varepsilon_1+\cdots +\varepsilon_{{r }-2}+\varepsilon_{r-1 }
+\varepsilon_{{r }})}.
$$

Let
$$w_1:= \sum_{i=2}^r  e_{\eps_1-\eps_i} 
e_{\eps_1+\eps_i}\in S^2(\fing).$$
Then $w_1$ is a singular vector with respect to the adjoint action of
$\fing$
and generates an irreducible finite-dimensional representation $W_1$ 
of $\g$ in $S^2(\fing)$  
isomorphic to $L_{\fing}(\theta+\theta_1)$.

%

 \begin{Pro}[{\cite[Theorem 3.1]{Per13}}]
  The vector
$\sigma(w_1)^{n+1}$ is a singular vector of $V^k(\g)$ if and only if $k=n-r+2 $. 
 \end{Pro}
 Let $\tilde{V}_{2-r}(\fing)=V^{2-r}(\fing)/U(\affg)\sigma(w_1)$.
By definition,
we have
\begin{align}\label{eq:X-and-V-type-D}
X_{\tilde{V}_{2-r}(\fing)}=V(I_{W_1}).
\end{align}


 \smallskip
 
Let $\l^{I}$ and $\l^{II}$ be the standard Levi subalgebras of $\g$ 
generated by the simple roots $\alpha_1,\ldots,\alpha_{r-2},\alpha_{r}$ 
and $\alpha_1,\ldots,\alpha_{r-2},\alpha_{r-1}$, respectively. 
Denote by $\SS_{\l^{I}}$ and $\SS_{\l^{II}}$ the corresponding 
Diximier sheets. 
Let $W(\g,\h) \cong \mathfrak{S}_{r} \rtimes (\Z/2\Z)^{r-1}$ 
denotes the Weyl group of $(\g,\h)$ which is 
the group of permutations and sign changes involving 
only even numbers of signs of the set $\{\eps_1,\ldots,\eps_{r}\}$, 
see,~e.g.,~\cite[\S12.1]{Hum72}. 

If $r$ is odd, then $\varpi_{r-1 }$ and $-\varpi_{r}$ are $W(\g,\h)$-conjutate 
by the above description of $W(\g,\h)$,  
and so the two Levi subalgebras $\l^{I}$ and $\l^{II}$ are $G$-conjugate\footnote{Therefore 
the statement \cite[Lemma 7.3.2(ii)]{CMa} is not correct for odd rank, as was  
mentioned to us by David Vogan.} 
since their centers $\z(\l^{I})$ and $\z(\l^{II})$ are generated by $\varpi_{r-1 }$ and 
$\varpi_{r}$ respectively. Hence 
the two Dixmier sheets $\SS_{\l^{I}}$ and $\SS_{\l^{II}}$ are equal. 
We set $\l_r=\l^{II}$ so that $\SS_{\l_{r}}=\SS_{\l^{I}}=\SS_{\l^{II}}$.  
Note that $\O_{(2^{r-1},1^2)}$ is induced in a unique way from the zero orbit 
in a conjugate of $\l_{r}$ (see e.g.~\cite[Corollary 7.3.4]{CMa}) 
so that $\O_{(2^{r-1},1^2)}={\rm Ind}_{\l_r}^{\g}(0)$ is 
the unique nilpotent orbit 
contained in $\SS_{\l_{r}}$. 

If $r$ is even, then $\l^{I}$ and $\l^{II}$ are not $G$-conjugate 
and so the Dixmier sheets $\SS_{\l^I}$ and $\SS_{\l^{II}}$ are distinct. 
Possible changing the numbering, we can assume that 
$\O_{(2^{r})}^{I}$ is the unique nilpotent orbit 
contained in $\SS_{\l^I}$ and that  
$\O_{(2^{r})}^{II}$ is the unique nilpotent orbit 
contained in $\SS_{\l^{II}}$. 
To see this, first observe that $\O_{(2^{r})}^{I}$ 
and $\O_{(2^{r})}^{II}$ are induced only from the zero 
orbit in a Levi subalgebra conjugate to $\l^{I}$ or $\l^{II}$ 
(see e.g.~\cite[Corollary 7.3.4]{CMa}). 
So the only possible sheets containing $\O_{(2^{r})}^{I}$ 
and $\O_{(2^{r})}^{II}$ are $\SS_{\l^{I}}$ and $\SS_{\l^{II}}$. 
On the other hand, one knows that $\SS_{\l^{I}}$  
contains a unique nilpotent orbit; so it is either $\O_{(2^{r})}^{I}$,  
or $\O_{(2^{r})}^{II}$. Assume for instance that it is $\O_{(2^{r})}^{I}$. 
This implies that $\O_{(2^{r})}^{II}$ must be contained in $\SS_{\l^{II}}$. 
If we had assumed that $\O_{(2^{r})}^{II}$ is contained in $\SS_{\l^{I}}$, 
we get in the same way that $\O_{(2^{r})}^{I}$ must be contained in $\SS_{\l^{II}}$.

\begin{lemma} \label{lem:Dnil}
\begin{enumerate} 
\item If $r$ is odd,  $V(I_{W_1}) \cap \mc{N} \subset \overline{\O_{(2^{r -1},1^2)}}$. 
\item If $r$ is even,  $V(I_{W_1}) \cap \mc{N} \subset \overline{\O_{(2^{r })}^{I}}
\cup \overline{\O_{(2^{r })}^{II}}$. 
\end{enumerate}
\end{lemma}

\begin{proof} 
Let $\bs{{\lam}}$ be the element of $\P_1(2r )$ defined by 
\begin{align*}
\bs{{\lam}} &:= (3,2^{r -3},1^3)  \text{ if } r \text{ is odd},& \\
\bs{{\lam}} &:= (3,2^{r -2},1)  \text{ if } r \text{ is even}.  &
\end{align*}
We observe that $\O_{\bs{{\lam}}}$ it the smallest nilpotent 
orbit of $\g$ which dominates $\O_{(2^{r -1},1^2)}$ 
if $r$ is odd, and both $\O_{(2^{r })}^{I}$ and 
$\O_{(2^{r })}^{II}$ if $r$ is even. 
By this, it means that 
if for ${\bs{\mu}} \in \P_1(2r )$, 
${\bs{\mu}} \succ  (2^{r -1},1^2)$ if $r$ is odd, 
and ${\bs{\mu}} \succ  (2^{r })$ if $r$ is even, 
then ${\bs{\mu}} \succcurlyeq \bs{{\lam}}$ 
where $\succcurlyeq$ is the partial order on the set $\P_1(n)$ 
induced by the Chevalley order.  
Therefore, it is enough to show that $V(I_{W_2})$ does not contain 
$\overline{\O_{\bs{{\lam}}}}$. 

Let $f \in \O_{\bs{{\lam}}}$ that we embed into an $\sl_2$-triple $(e,h,f)$ 
of $\g$. For $i \in \Z$, denote by $\g(h,i)$ the $i$-eigenspace of ${\rm ad}(h)$ 
and by $\Delta_+(h,i)$ the set of 
positive roots $\alpha \in\Delta_+$ 
such that $e_\alpha \in \g(h,i)$. 
Choose a Lagrangian subspace 
$\mathfrak{L} \subset \g(h,1)$ and set 
$$\mf{m} := \mathfrak{L} \oplus \bigoplus_{i \geq 2} \g(h,i), \qquad 
J_\chi :=  \sum_{x \in \mf{m}} \C[\g^*](x-\chi(x)),$$
with $\chi = (f | \cdot ) \in \g^*$,   
as in \S\ref{section:Ginzburg}.
By Lemma \ref{lem:criterion},
 it is sufficient to show that
 \begin{align*}
  \C[\fing^*]=I_{W_1}+ J_{\chi}.
 \end{align*}
To see this, we shall use the vector
$$w_1= \sum_{i=2}^{r } e_{\eps_1-\eps_i} 
e_{\eps_1+\eps_i}.$$

Assume first that $r$ is odd.
It follows from \cite[Lemma 5.3.5]{CMa} that the weighted 
Dynkin diagram of the nilpotent orbit $G.f$ is 
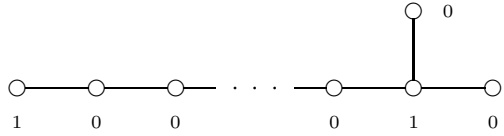
\vspace{-0.25cm}
$$\begin{Dynkin}
\Dspace\Dspace\Dspace\Dspace\Dspace\Dbloc{\Dcirc\Dsouth\Dtext{r}{0}}
\Dskip
\Dbloc{\Dcirc\Deast\Dtext{b}{1}}
\Dbloc{\Dcirc\Dwest\Deast\Dtext{b}{0}}
\Dbloc{\Dcirc\Dwest\Deast\Dtext{b}{0}}
\Dbloc{\Ddots}
\Dbloc{\Dcirc\Dwest\Deast\Dtext{b}{0}}
\Dbloc{\Dcirc\Dwest\Deast\Dnorth\Dtext{b}{1}}
\Dbloc{\Dcirc\Dwest\Dtext{b}{0}}
\end{Dynkin}$$
So we can choose for $h$ the element 
$$h=\varpi_1+\varpi_{r -2}=2 \alpha_1+
\sum_{j=2}^{r -2} (j+1) \alpha_j +\frac{r }{r -1}(\alpha_{r -1} 
+\alpha_{r }).$$ 
We see from the above diagram that ${\eps_1+\eps_i} \in\Delta_+(h,3)$ 
for  $i \in \{2,\ldots,r -2\}$, and 
$$\Delta_+(h,2)=\{\eps_1\pm \eps_{r-1},\, \eps_{1} \pm \eps_r ,\, 
\eps_i+\eps_j,\; {2 \leq i < j < r-2}\}.$$
We can choose $e,f$ so that 
$$e = \sum_{\alpha \in \Delta_+(h,2)} a_\alpha e_\alpha\quad \text{ and } 
\quad f = \sum_{\alpha \in \Delta_+(h,2)} b_\alpha f_\alpha$$
with $a_\alpha,b_\alpha \in\C$ for all $\alpha \in \Delta_+(h,2)$. 
Set for $\alpha \in \Delta_+(h,2)$, 
$$c_\alpha := a_\alpha b_\alpha.$$ 
From the relation $[e,f]=h$, we obtain the equations:
\begin{align} \label{eq:c1}
c_{\eps_1- \eps_{r-1}}+c_{\eps_1+ \eps_{r-1}}+
c_{\eps_1- \eps_{r}}+c_{\eps_1+ \eps_{r}} &= 2 \\ \label{eq:c2}
c_{\eps_1+ \eps_{r-1}}+
c_{\eps_1- \eps_{r}} + \sum_{2\leq i <j < r-2} c_{\eps_i+\eps_j}&= \frac{r-1}{2}\\ 
\label{eq:c3}
c_{\eps_1+ \eps_{r-1}}+
c_{\eps_1+ \eps_{r}} + \sum_{2\leq i <j < r-2} c_{\eps_i+\eps_j}&= \frac{r-1}{2}\\ 
\label{eq:c4}
c_{\eps_1- \eps_{r-1}}+c_{\eps_1+ \eps_{r-1}}+
c_{\eps_1- \eps_{r}} +c_{\eps_1+ \eps_{r}} + 
2 \sum_{2\leq i <j < r-2} c_{\eps_i+\eps_j}&= r-1
\end{align}
by considering the coefficients of $h$ in $\alpha_1$, $\alpha_{r-1}$, 
$\alpha_{r}$ and $\alpha_{r-2}$ respectively. 
By (\ref{eq:c2}) and (\ref{eq:c3}) we get that 
$$c_{\eps_1- \eps_{r}} =c_{\eps_1+ \eps_{r}}.$$ 
Then by (\ref{eq:c4}), (\ref{eq:c2}) and (\ref{eq:c3}) we get 
that $$c_{\eps_1- \eps_{r-1}} =c_{\eps_1+ \eps_{r-1}}.$$ 
So from (\ref{eq:c1}), we obtain: 
$$ c_{\eps_1- \eps_{r}} +c_{\eps_1- \eps_{r-1}}=1.$$ 
In particular, $c_{\eps_1- \eps_{r}} \not=-c_{\eps_1- \eps_{r-1}}$ 
and $c_{\eps_1- \eps_{r-1}}$ and $c_{\eps_1- \eps_{r}}$ cannot 
be both zero. 

Hence for some nonzero complex number $c$, we have 
$$w_1 = \sum_{i=2}^{r -2} e_{\eps_1-\eps_i} 
e_{\eps_1+\eps_i}  +  e_{\eps_1-\eps_{r-1} } 
e_{\eps_1+\eps_{r-1} }+e_{\eps_1-\eps_r } 
e_{\eps_1+\eps_r } = c {\pmod{J_\chi}}$$ 
and so $\IS_{W_2} + {J_\chi}=\C[\g^*]$.  

Assume that $r$ is even.\\
It follows from \cite[Lemma 5.3.5]{CMa} that the weighted 
Dynkin diagram of the nilpotent orbit $G.f$ is 
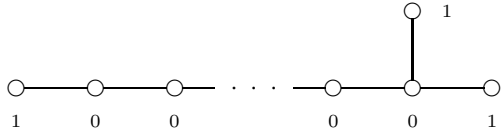
\vspace{-0.25cm}
$$\begin{Dynkin}
\Dspace\Dspace\Dspace\Dspace\Dspace\Dbloc{\Dcirc\Dsouth\Dtext{r}{1}}
\Dskip
\Dbloc{\Dcirc\Deast\Dtext{b}{1}}
\Dbloc{\Dcirc\Dwest\Deast\Dtext{b}{0}}
\Dbloc{\Dcirc\Dwest\Deast\Dtext{b}{0}}
\Dbloc{\Ddots}
\Dbloc{\Dcirc\Dwest\Deast\Dtext{b}{0}}
\Dbloc{\Dcirc\Dwest\Deast\Dnorth\Dtext{b}{0}}
\Dbloc{\Dcirc\Dwest\Dtext{b}{1}}
\end{Dynkin}$$ 
So we can choose for $h$ the element 
$$h=\varpi_1+\varpi_{r -1}+\varpi_{r } 
= 2\alpha_1+\sum_{j=2}^{r -2} (j+1) \alpha_j +\frac{r }{2}(\alpha_{r -1} 
+\alpha_{r }).$$ 
We see from the above diagram that ${\eps_1+\eps_i} \in\Delta_+(h,3)$ 
for $i \in\{2,\ldots, r -1\}$, and 
$$\Delta_+(h,2)= \{\eps_1\pm \eps_r , \, \eps_i+\eps_j, \; 
2 \leq i < j \leq r -1\}.$$ 
We can choose $e,f$ so that 
$$e = \sum_{\alpha \in \Delta_+(h,2)} a_\alpha e_\alpha\quad \text{ and } 
\quad f = \sum_{\alpha \in \Delta_+(h,2)} b_\alpha f_\alpha$$
with $a_\alpha,b_\alpha \in\C$ for all $\alpha \in \Delta_+(h,2)$. 
Set for $\alpha \in \Delta_+(h,2)$, 
$$c_\alpha := a_\alpha b_\alpha.$$
From the relation $[e,f]=h$, we obtain the equations:
\begin{align} \label{eq:d1} 
c_{\eps_1- \eps_{r}}+c_{\eps_1+ \eps_{r}} &= 2 \\ \label{eq:d2}
c_{\eps_1- \eps_{r}} + \sum_{2\leq i <j \leq r-1} c_{\eps_i+\eps_j}&= \frac{r}{2}\\ 
\label{eq:d3}
c_{\eps_1+ \eps_{r}} + \sum_{2\leq i <j \leq r-1} c_{\eps_i+\eps_j}&= \frac{r}{2}
\end{align}
by considering the coefficients of $h$ in $\alpha_1$, $\alpha_{r-1}$, 
and $\alpha_{r}$ respectively. 
By (\ref{eq:d2}) and (\ref{eq:d3}) we get that 
$c_{\eps_1- \eps_{r}} =c_{\eps_1+\eps_{r}}$. So from 
(\ref{eq:d1}), we get that 
$$c_{\eps_1- \eps_{r}} =c_{\eps_1+\eps_{r}}=1.$$

Hence for some nonzero complex number $c$, we have 
$$w_1 = \sum_{i=2}^{r -1} e_{\eps_1-\eps_i} 
e_{\eps_1+\eps_i}  +e_{\eps_1-\eps_r } 
e_{\eps_1+\eps_r } = c {\pmod{J_\chi}}$$ 
and so $\IS_{W_2} + {J_\chi}=\C[\g^*]$.  
\end{proof}

\begin{lemma} \label{lem:Dss} 
Let $\lam$ be a nonzero semisimple element of $\g$ 
which belongs to $V(I_{W_1})$. Then 
$\lam \in G.\C^* \varpi_{r-1}$ or $\lam \in G.\C^* \varpi_{r}$. 
\end{lemma}

\begin{proof} 
Arguing as in the beginning of the proof of Lemma \ref{lem:A-zero0}, 
we see that it is enough to prove the lemma 
for a nonzero element $\lam$ in $V(I_{W_1})\cap \h$. 

The image $\Upsilon(I_{W_1}^\h)$ of $I_{W_1}^\h$, viewed as an 
ideal of $U(\g)$, was determined by Per{\v{s}}e in the 
proof of \cite[Theorem~3.4]{Per13}. 
 From Per{\v{s}}e's
 proof, we easily deduce that 
$\Psi(I_{W_1} \cap S(\g)^\h)$ is generated by the elements 
\begin{align*}
p_i := h_i (h_i+2 h_{i+1} +\cdots + 2 h_{r-2} +h_{r-1}+h_r), \qquad 
i=1,\ldots,r, 
\end{align*}
and that 
\begin{align*}
V(I_{W_1}) \cap \h=  
\bigcup\limits_{\{i_1,\ldots,i_k\} \subset \{1,\ldots,r-2\} \atop i_1 < \cdots < i_k} 
(\C \left( \sum_{j=1}^k (-1)^{k-j+1}  \varpi_{i_j} +\varpi_{r-1}\right) &\\
\cup \;
\C  \left(\sum_{j=1}^k (-1)^{k-j+1}  \varpi_{i_j} +\varpi_{r} \right) ).
\end{align*}

Thus we are lead to show that for any nonempty sequence 
$(i_1,\ldots,i_k)$ in $\{1,\ldots,r-2\}$, with $i_1 < \cdots < i_k$, 
the element  
$\sum_{j=1}^k(-1)^{k-j+1}  \varpi_{i_j} +\varpi_{r-1}$ 
(respectively $\sum_{j=1}^k (-1)^{k-j+1}  \varpi_{i_j} +\varpi_{r}$) 
is either $W(\g,\h)$-conjugate to $\varpi_{r-1}$, 
or $W(\g,\h)$-conjugate to $\varpi_{r}$. 
(If the sequence is empty, the statement is obvious.) 

Let $(i_1,\ldots,i_k)$ in $\{1,\ldots,r-2\}$, 
with $i_1 < \cdots < i_k$. 
We prove the statement for 
$\sum_{j=1}^k(-1)^{k-j+1}  \varpi_{i_j} +\varpi_{r}$. 
Similar arguments hold for $\sum_{j=1}^k(-1)^{k-j+1}  \varpi_{i_j} +\varpi_{r-1}$. 

If $k$ is even, we have 
\begin{align*}
\sum_{j=1}^k(-1)^{k-j+1}  \varpi_{i_j}  
=\sum_{j=1}^{k-1} (\varpi_{i_j}-\varpi_{i_{j+1}})
= - \sum_{j=1}^{k/2} (\eps_{i_{2j-1}+1}+\cdots+\eps_{i_{2j}}). 
\end{align*}
Hence 
\begin{align*}
\sum_{j=1}^k(-1)^{k-j+1}  \varpi_{i_j} +\varpi_{r} = \frac{1}{2} (
(\eps_{1}+\cdots+\eps_{i_{1}} )- (\eps_{i_1}+\cdots+\eps_{i_2}) \hspace{2cm} \\
+\cdots +  ( \eps_{i_{k-2}}+\cdots+\eps_{i_{k-1}} )
- (\eps_{i_{k-1}+1}+\cdots+\eps_{i_k}) +\eps_{i_k+1}+\cdots+\eps_{i_r})  ),
\end{align*}
and this element is $W(\g,\h)$-conjugate to $\varpi_{r}$ 
or $\varpi_{r-1}$ depending on the parity 
of $(i_2-i_1)+\cdots+(i_k-i_{k-1})$. 

If $k$ is odd, then 
\begin{align*}
\sum_{j=1}^k(-1)^{k-j+1}  \varpi_{i_j}  
&=(-\varpi_{i_1}+\varpi_{i_2})+( -\varpi_{i_3}+\varpi_{i_4})+\cdots 
+ (- \varpi_{i_{k-2}} + \varpi_{i_{k-1}}) - \varpi_{i_k} &\\
&= \sum_{j=1}^{(k-1)/2} (\eps_{i_{2j-1}+1}+\cdots+\eps_{i_{2j}}) 
-(\eps_{1}+\cdots +\eps_{i_k}) &\\
& = -(\eps_{1}+\cdots+\eps_{i_1}) - 
\sum_{j=1}^{(k-1)/2} (\eps_{i_{2j}+1}+\cdots+\eps_{i_{2j}+1}) 
\end{align*} 
and we conclude as in the case where $k$ is even 
that $ \sum_{j=1}^k(-1)^{k-j+1}  \varpi_{i_j}+\varpi_r$ is 
conjugate to $\varpi_{r}$ 
of $\varpi_{r-1}$, depending here on the parity 
of $i_1+(i_3-i_2)+\cdots+(i_k-i_{k-1})$. 
\end{proof}
The following assertion follows immediately from
Lemma \ref{Lem:sheet},
Lemma \ref{lem:Dnil}
and Lemma \ref{lem:Dss} .
\begin{Pro} \label{Pro:Dzero}
We have $V(I_{W_1}) = \overline{\SS_{\l^{I}}} \cup 
 \overline{\SS_{\l^{II}}}$,
 and
hence,
 $ X_{\tilde{V}_{2-r}(\fing)}    =\overline{\SS_{\l^{I}}} \cup 
 \overline{\SS_{\l^{II}}}$.
\end{Pro}
    We are now in a position to prove Theorem \ref{Th:reducible}.
    \begin{proof}[Proof of Theorem \ref{Th:reducible}] 
    Assume that $r$ is odd. 
     Since $V_{2-r}(\fing)$ is a quotient of $\tilde{V}_{2-r}(\fing)$,
\begin{align}
X_{V_{2-r}(\fing)}
 \subset   X_{\tilde{V}_{2-r}(\fing)} =\overline{\SS_{\l^{I}}} \cup 
 \overline{\SS_{\l^{II}}}  =\overline{\SS_{\l_{r}}}
 \label{eq:contianed-in}
\end{align}
     by Proposition \ref{Pro:Dzero} since $\SS_{\l_{r}}=\SS_{\l^{I}}=\SS_{\l^{II}}$.
        On the other hand,  by   \cite[Theorem 7.2]{AdaPer14},
    one knows that
    $V_{2-r}(\fing)$ has infinitely many simple modules
     in the category $\mathcal{O}$.
     This gives 
     \begin{align} {
 X_{V_{2-r}(\fing)}
 \not\subset  \mc{N} }
 \label{eq:not-contained}
     \end{align}
     by \cite[Corollary 5.3]{AM15}.
     Therefore
     $X_{V_{r-2}(\fing)}
     \cap \h\ne 0$ 
     as $X_{V_{r-2}(\fing)}$ is closed and $G$-invariant.
     Hence, by \eqref{eq:contianed-in},
     either 
     $\C \varpi_{r-1}\subset X_{V_{2-r}(\fing)}$, or
     $\C \varpi_r\subset X_{V_{2-r}(\fing)}$.
Thus,
     $ \overline{\SS_{\l_{r}}}  \subset X_{V_{2-r}(\fing)}$ 
     since $X_{V_{2-r}(\fing)}$ is a $G$-invariant closed cone. 
This completes the proof.
    \end{proof}
    We now wish to prove Theorem \ref{Th:Dr-for-even-r}.
     \begin{Th}\label{Th:image-type-D} 
Assume $r$ is even, and let $f \in \mathbb{O}_{(2^r)}^{I} 
\cup \mathbb{O}_{(2^r)}^{II}$.
 Then
 $H^{\frac{\infty}{2}+0}_{f}(V_{2-r}(\fing))=0$.
\end{Th}
   \begin{proof}
    By symmetry,
    we may assume that 
the weighted Dynkin diagram of $f$ is given by 
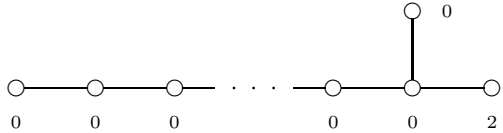
\vspace{-0.25cm}
$$\begin{Dynkin}
\Dspace\Dspace\Dspace\Dspace\Dspace\Dbloc{\Dcirc\Dsouth\Dtext{r}{0}}
\Dskip
\Dbloc{\Dcirc\Deast\Dtext{b}{0}}
\Dbloc{\Dcirc\Dwest\Deast\Dtext{b}{0}}
\Dbloc{\Dcirc\Dwest\Deast\Dtext{b}{0}}
\Dbloc{\Ddots}
\Dbloc{\Dcirc\Dwest\Deast\Dtext{b}{0}}
\Dbloc{\Dcirc\Dwest\Deast\Dnorth\Dtext{b}{0}}
\Dbloc{\Dcirc\Dwest\Dtext{b}{2}}
\end{Dynkin}$$
    and $h=2\varpi_r$.
Set $\mf{l}=\mf{l}^{II}=\fing(h,0)$.

    Since
    $f$ is a short nilpotent element,
    it is sufficient to show that 
\begin{align*}
 \on{dim} V_{2-r}(\fing)_{0,h}
 <\frac{1}{2}\dim \mathbb{O}_{(2^r)}^{II}
\end{align*}
    by Theorem \ref{Th:short}.
    Observe that,
    as a $\fing$-module,
    $ V_{2-r}(\fing)_{0,h}$ is isomorphic to
    $L_{\fing}((2-r)\varpi_r)$,
    which is a quotient of the {\em scalar generalized Verma module}
 $M_{\mf{l}}((2-r)\varpi_r)$.
Here
\begin{align*}
M_{\mf{l}}(a\varpi_r)=U(\fing)\*_{U(\mf{l}\+ \fing(h,2))}\C_{a\varpi_r}
\end{align*}
    for $a\in \C$,
    where 
    $\C_{a\varpi_r}$ is a one-dimensional representation
    of $\mf{l}\+ \fing(h,2)$ on which
    $[\mf{l},\mf{l}]\+ \fing(h,2)$ acts trivially and
    $h_r\in \mf{z}(\mf{l})$ acts as multiplication
    by $a$.
    By \cite[Theorem 3.2.3 (2c)]{Mat06},
    one knows that 
    there is an embedding
    $M_{\mf{l}}((-a-2r-2)\varpi_r)\hookrightarrow M_{\mf{l}}(a\varpi_r)$
    for $a\in \Z$, $a\geq 2-r$.
In particular,
    $M_{\mf{l}}(-r\varpi_r)\subset M_{\l}((2-r)\varpi_r)$.
Therefore, 
    \begin{align*}
     \on{Dim} V_{2-r}(\fing)_{0,h}=
     \on{Dim} L_{\fing}((2-r)\varpi_r)&\leq \on{Dim}
     (M_l((2-r)\varpi_r)/M_\l(-r\varpi_r))\\
     &<\on{Dim} (M_\l((2-r)\varpi_r))=
     \frac{1}{2}\dim \mathbb{O}_{(2^r)}^{II}. 
    \end{align*}
     This completes the proof.
   \end{proof}

     \begin{proof}[Proof of Theorem \ref{Th:Dr-for-even-r}]    First,
      we have 
    $ X_{V_{2-r}(\fing)}\subset 
\overline{\SS_{\l^{I}}} \cup 
 \overline{\SS_{\l^{II}}}
      $
      by Proposition~\ref{Pro:Dzero}.
      Suppose 
      that  $    X_{V_{2-r}(\fing)}\subsetneqq \mc{N}$.
      As in the proof of Theorem \ref{Th:reducible},
      it follows that 
      either $\C\varpi_{r-1}\subset X_{V_{2-r}(\fing)}$
      or $\C\varpi_{r}\subset X_{V_{2-r}(\fing)}$.
Thus
          $  X_{V_{2-r}(\fing)}=\overline{\SS_{\l^{I}}} \cup 
 \overline{\SS_{\l^{II}}}$.
But 
       we have ${\mathbb{O}_{(2^r)}^I, \mathbb{O}_{(2^r)}^{II} \not\subset 
    X_{V_{2-r}(\fing)}}$ 
    by
  Theorem~\ref{Th:W-algebra-variety}
  and 
      Theorem~\ref{Th:image-type-D}.
Since this is a contradiction we get that
      $    X_{V_{2-r}(\fing)}\subset \mc{N}$,
      and hence,
  $$  {X_{V_{2-r}(\fing)}\subsetneqq
\overline{\mathbb{O}_{(2^r)}^I},\
      \overline{\mathbb{O}_{(2^r)}^{II}}.}$$
We conclude that
    $X_{V_{2-r}(\fing)}
    \subset \overline{\mathbb{O}_{(2^{r-2},1^4)}}=
    \overline{\mathbb{O}_{(2^r)}^I}\cap
    \overline{\mathbb{O}_{(2^r)}^{II}}$.

On the other hand,  Theorem \ref{Th:minimal} gives that 
    $H^{\frac{\infty}{2}+0}_{f_{\theta}}(V_{r-2}(\fing))=\W_{r-2}(\fing,f_{\theta})$,
    which is not lisse by \cite[Theorem 7.1]{AM15} unless $r=4$.
    Therefore,
    $\overline{\mathbb{O}_{min}}\subsetneqq
    X_{V_{2-r}(\fing)}$
    by  Theorem~\ref{Th:W-algebra-variety}.

    The last assertion follows from  \cite[Corollary 5.3]{AM15}.
   \end{proof}
\begin{Conj}\label{Conj:r-even}
 Let $r$  be even.
 Then
 $ X_{V_{2-r}(\fing)}=\overline{\mathbb{O}_{(2^{r-2},
 1^4)}}$.
 Therefore,
 the $W$-algebra
 $\W_{2-r}(\fing,f)$
 is lisse
for  $f\in \mathbb{O}_{(2^{r-2},
 1^4)}$.
\end{Conj}
\begin{Th} \label{Th:r-6}
 Conjecture \ref{Conj:r-even} is true for $r=6$.
\end{Th}
\begin{proof}
 For $r=6$,
 $\mathbb{O}_{min}$
 is the only nilpotent orbit that is strictly contained in
 $\overline{\mathbb{O}_{(2^{r-2},1^4)}}$.
\end{proof}

\section{On associated varieties and minimal $W$-algebras} \label{sec:others} 

%

Let $f \in \mc{N}$, $\W^k(\g,f)$ the affine $W$-algebra associated with 
$(\g,f)$ at level $k$, 
and ${\W}_k(\g,f)$ its unique simple quotient as in \S \ref{sec:affine-W-algebras}. 
By Theorem \ref{Th:W-algebra-variety},  
we know that
$\W_k(\fing,f)$ is lisse if $X_{V_k(\g)}=\overline{G.f}$. 
It is natural to ask whether the converse holds.
We now prove Theorem \ref{Th:lisse-rig} which says that 
for the minimal nilpotent element  $f=f_{\theta}$, this  is 
true
provided that $k\not\in \Z_{\geq 0}$. 

\begin{proof}[Proof of Theorem \ref{Th:lisse-rig}]
First, if  $k\not \in \Z_{\geq 0}$, then
$\W_k(\fing,f_{\theta})=H^{\frac{\infty}{2}+0}(V_k(\fing))$
by 
Theorem \ref{Th:minimal}.
Assume that  ${\W}_k(\g,f_\theta)$ is lisse. Then by \cite[Theorem 4.21 
and Proposition 4.22]{Ara09b}, 
\begin{align} \label{eq:condition-lisse}
\overline{\O_{min}} \subset X_{V_k(\g)}\quad 
\text{ and } \quad \dim (X_{V_k(\g)} \cap \Slo_{f_\theta}) =0.
\end{align}
 
Let $x$ be a closed point of $\tilde{X}_{V_k(\g)}$.
We need to show that $x\in \overline{\O_{min}}$. 
Since $X_{V_k(\g)}$ is a $G$-invariant closed cone, it contains the $G$-invariant 
cone $C(x)=G.\C^*x$ generated by $x$, and its closure $\overline{C(x)}$. 
By \cite[Theorem 2.9]{CM}, $\overline{C(x)}\cap \mc{N}$ is the closure of 
the nilpotent orbit $\O$ induced from the nilpotent orbit of $x_n$ in $\g^{x_s}$.  
By \cite[Corollary 1.3.8(iii)]{Ginz}, the nilpotent orbit $\O$ is $\O_{min}$ or $0$;  
otherwise, $\dim( \overline{\O} \cap \Slo_{min})>0$ which would contradict 
(\ref{eq:condition-lisse}).  
If $\O=\O_{min}$ then $x_s$ must be zero; otherwise, 
again by \cite[Corollary 1.3.8(iii)]{Ginz}, 
$\dim( \overline{C(x)} \cap \Slo_{min})>0$ which would contradict 
(\ref{eq:condition-lisse}). 
If $\O=0$, then $x_s=x_n=0$ since $\on{codim}_\g(0)=\dim\g$ is the codimension 
of the nilpotent orbit of $x_n$ in $\g^{x_s}$.  

In both cases, $x=x_n$ and $x \in \overline{C(x)} \subset \overline{\O_{min}}$, 
whence the statement. 
\end{proof} 

 \begin{Rem}\label{rem;remark-for-positive-integer-level}
Let $k\in \Z_{\geq 0}$.  Then
$\W_k(\fing,f_{\theta})\cong H^{\frac{\infty}{2}+0}(L(s_0\circ k\Lam_0))$
where $s_0$ is the reflection corresponding to the simple root
$\alpha_0=-\theta+\delta$, see \cite{Ara05}.  Thus \cite{{Ara09b}}
\begin{align*}
X_{\W_k(\fing,f_{\theta})}=X_{L(s_0\circ \lam)}\cap \Slo_{f_{\theta}},
\end{align*}
where
$X_{L(\lam)}=\on{supp}_{\C[\fing^*]} (L(\lam)/C_2 (L(\lam)))$,
 $C_2(L(\lam))=\on{span}_{\C}\{a_{(-2)}v\mid a\in V^k(\fing), v \in L(\lam)\}$.
Therefore 
the proof of 
Theorem \ref{Th:lisse-rig}  
implies that,
for $k\in \Z_{\geq 0}$,
 ${\W}_k(\g,f_\theta)$ is lisse if and only if $X_{L(s_0\circ \lam)}=\overline{\O_{min}}$. 
 \end{Rem}

For a general $f\in \mc{N}$, we have the following result.
\begin{Pro}  
Let $f \in \mc{N}$ and $k \in \C$. 
Assume that $\W_k(\g,f)$ is lisse. Then $\overline{G.f}$ is an irreducible 
component of {$X_{V_k(\g)}$}. 
\end{Pro}

\begin{proof} 
By hypothesis and \cite[Theorem 4.21 and Proposition 4.22]{Ara09b}, 
we have 
$${\overline{G.f} \subset X_{V_k(\g)} }\quad 
\text{ and } \quad \dim (X_{V_k(\g)} \cap \Slo_{f} )=0.$$ 
Let $Y$ be an irreducible component of $X_{V_k(\g)}$
which contains $\overline{G.f}$. 
By \cite[Corollary 1.3.8(iii)]{Ginz}, 
$$\dim (Y \cap \Slo_f) =\dim Y -\dim \overline{G.f} $$ 
which forces $Y =\overline{G.f}$, whence the statement. 
\end{proof}

Let
$\fing^{\natural}$ be  the centralizer in $\fing$ of the $\mf{sl}_2$-triple
$(e_{\theta},h_{\theta},f_{\theta})$,
and let $\fing^{\sharp}=\bigoplus_{i\geq 0}\fing_i^{\natural}$
be the decomposition into the sum of its center $\fing_0^{\natural}$
and the simple summands $\fing^{\natural}_i$, $i\geq 1$.
Then we have a vertex algebra embedding
\begin{align*}
 \bigotimes_{i\geq 0}V^{k_i^{\natural}}(\fing_0)\hookrightarrow  \W^k(\fing,f_{\theta}),
\end{align*}
where $k_i^{\natural}$ is given in \cite[Tables 3 and 4]{AM15}.
Let $\theta_i$, $i\geq 1$, be the highest root of $\fing^{\natural}_i$.
 \begin{Lem}\label{Lem:singularVector}
Suppose that $k_i^{\natural}\in \Z_{\geq 0}$ for some $i\geq 1$.
Then the image of the singular vector
$e_{\theta_i}(-1)^{k_i^{\natural}+1}\mathbf{1}$ of
  $V^{k_i^{\natural}}(\fing_0)$
is a nonzero singular vector of $\W^k(\fing,f_{\theta})$.
 \end{Lem}
 \begin{proof}
It is nonzero because its image in
  $R_{\W^k(\fing,f_{\theta})}=\C[\Slo_{f_{\theta}}]$
is nonzero.
The rest follows from the commutation relations described in \cite[Theorem 5.1]{KacWak04}.
 \end{proof}
 \begin{Th}\label{Th:lisse}
Suppose that $\fing$ is not of type $A$.
Then the following conditions are equivalent:
\begin{enumerate}
 \item $\W_k(\fing,f_{\theta})$ is lisse,
\item $k_i^{\natural}\in \Z_{\geq 0}$ for all $i\geq 1$.
\end{enumerate}
 \end{Th}
 \begin{proof}
We have already showed the implication $(1)\Rightarrow (2)$ in our
  previous paper \cite{AM15}.
Let us show the implication  $(2)\Rightarrow (1)$.
Write
$R_{\W_k(\fing,f_{\theta})}=\C[\Slo_{f_{\theta}}]
/I$ for some
  Poisson ideal $I$.
Note that the radical $\sqrt{I}$ is also a Poisson ideal of 
$\C[\Slo_{f_{\theta}}]$, see e.g. \cite[Lemma 2.4.1]{Ara12}.
We identify
$\C[\Slo_{f_{\theta}}]$ with $S(\fing^e)$.

By assumption $e_{\theta_i}(-1)^{k_i^{\natural}+1}\mathbf{1}$ is zero  
in the simple quotient $\W_k(\fing,f_{\theta})$.
Hence
$e_{\theta_i}\in \sqrt{I}$ for all $i$.
Since 
the restriction of  the Poisson bracket to $S(\fing_i)\subset
  S(\fing^e)=\C[\mathcal{S}_{f_{\theta}}]$
coincides with the Kirillov-Kostant Poisson structure of $\fing_i^*$,
$\fing_i\subset \sqrt{I}$ for all $i$,
because 
$\fing_i$ is simple
(note that $\fing_0=0$).
Further, we have 
\begin{align*}
 \{x, \bar G^{\{a\}}
\}=G^{\{[x,a]\}}\quad\text{for }a\in \fing_{-1/2}
\end{align*}
by \cite[Theorem 5.1 (d)]{KacWak04},
where $\bar G^{\{a\}}$ is the image of $G^{\{a\}}$ in
  $R_{W^k(\fing,f_{\theta})}=\C[\Slo_{f_{\theta}}]$.
Since $\fing_{-1/2}$ is a direct sum of non-trivial representations of
  $\fing^{\natural}$ by \cite[Table 1]{KacWak04},
we get that 
$\bar G^{\{a\}}\in \sqrt{I}$ for all $a\in \fing_{-1/2}$.
Finally,
\cite[Theorem 5.1 (e)]{KacWak04} implies that
{\begin{align*}
 \{\bar G^{\{u\}},\bar G^{\{v\}}\}\equiv -2(k+h^{\vee}))(e|[u,v])\bar
 L\pmod{\sqrt{I}},
\end{align*}} 
where $\bar L$ is the image of $L$ in
  $R_{W^k(\fing,f_{\theta})}=\C[\Slo_{f_{\theta}}]$.
Thus, $\bar L\in \sqrt{I}$, and we conclude that all generators of 
$S(\fing^e)$ belong to $\sqrt{I}$.
 \end{proof}

 We are now in a position to prove Theorem \ref{Th:G2}, 
 which was conjectured in \cite{AM15}, 
 and Theorem \ref{Th:classification}.

   \begin{proof}[Proof of Theorem \ref{Th:G2}]
    By Theorem \ref{Th:lisse} and \cite[Table 4]{AM15},
    $\W_k(\fing,f_{\theta})$ is lisse if and only if $3k+5\in\Z_{\geq 0}$.
   \end{proof}
   
 \begin{proof}[Proof of Theorem \ref{Th:classification}]
  We know that $ X_{V_k(\fing)}=\{0\}$ if $k\in \Z_{\geq 0}$
  since $V_k(\fing)$  is lisse in this case.
  Thus, the assertion follows from
  Theorem \ref{Th:lisse-rig},
  Theorem \ref{Th:G2}
  and \cite[Theorem 7.1]{AM15}.
 \end{proof}
  

By Theorem \ref{Th:G2}, we have thus 
obtained a new family of lisse minimal $W$-algebras
 $\W_k(G_2,f_{\theta})$, for $k=-1,0,1,2,3\dots$.

Finally we remark that the following assertion follows from
 Remark \ref{rem;remark-for-positive-integer-level}
and Theorem \ref{Th:lisse}.
 \begin{Th}
We have $X_{L(s_0\circ k\Lam_0)}=\overline{\mathbb{O}_{min}}$ for $k\in \Z_{\geq 0}$
if $\fing$ is of type $D_r$, $r\geq 4$, $G_2$, $F_4$, $E_6$, $E_7$ and $E_8$.
 \end{Th}

\newcommand{\etalchar}[1]{$^{#1}$}



\end{document}